\documentclass[11pt]{article}

\usepackage{mathrsfs}
\usepackage[deletedmarkup=xout]{changes}

\usepackage[margin=1in]{geometry}
\usepackage[utf8]{inputenc}
\usepackage{enumerate}
\usepackage{amsfonts}
\usepackage{amsmath}
\usepackage{amssymb}
\usepackage{bbm}
\usepackage{amsthm}
\usepackage{xcolor}
\usepackage[title]{appendix}

\makeatletter
\newcommand{\raisemath}[1]{\mathpalette{\raisem@th{#1}}}
\newcommand{\raisem@th}[3]{\raisebox{#1}{$#2#3$}}
\makeatother

\newcommand{\bsigma}{\sigma}
\newcommand{\bSigma}{\Sigma}
\newcommand{\btau}{\tau}
\newcommand{\brho}{\rho}

\newcommand{\balpha}{\alpha}

\newcommand{\gbs}[1]{\langle #1 \rangle}
\newcommand{\bgbs}[1]{\big \langle #1 \big \rangle}
\newcommand{\bbgbs}[1]{\Big \langle #1 \Big \rangle}
\newcommand{\bbbgbs}[1]{\bigg \langle #1 \bigg \rangle}
\newcommand{\lrgbs}[1]{\left \langle #1 \right\rangle}
\newcommand{\E}{\mathbb E}
\newcommand{\e}{\mathbb E}
\newcommand{\la}{\langle}
\newcommand{\ra}{\rangle}

\theoremstyle{definition}

\newtheorem{theorem}{Theorem}[section]

\newtheorem{prop}[theorem]{Proposition}
\newtheorem{lemma}[theorem]{Lemma}
\newtheorem{remark}[theorem]{Remark}
\newtheorem{notation}[theorem]{Notation}

\begin{document}

\title{On the TAP equations via the cavity approach\\
	 in the generic mixed $p$-spin models}

\author{Wei-Kuo Chen\thanks{University of Minnesota, Minneapolis, USA. Email: wkchen@umn.edu} \and  Si Tang \thanks{Lehigh University, Bethlehem, USA. Email: sit218@lehigh.edu}}

\maketitle

\begin{abstract}
		In 1977, Thouless, Anderson, and Palmer (TAP) derived a system of consistent equations in terms of the effective magnetization in order to study the free energy in the Sherrington-Kirkpatrick (SK) spin glass model. The solutions to their equations were predicted to contain vital information about the landscapes in the SK Hamiltonian and the TAP free energy and moreover have direct connections to Parisi's replica ansatz. In this work, we aim to investigate the validity of the TAP equations in the generic mixed $p$-spin model. By utilizing the ultrametricity of the overlaps, we show that the TAP equations are asymptotically satisfied by the conditional local magnetizations on the asymptotic pure states.
\end{abstract}

\section{Introduction\label{intro}}

The study of mean-field spin glasses has been one of the central objectives in Statistical Physics over the past decades. Based on the replica method, this approach has attained great achievements following Parisi's celebrated ansatz \cite{parisi1979infinite,parisi1980sequence,parisi1983order} for the famous Sherrington-Kirkpatrick (SK) model \cite{SK72} as well as its variants, see physics literature in \cite{MPV87} and recent mathematical development in \cite{Pan13,Tal111,Tal112}. In a different direction, Thouless-Anderson-Palmer \cite{TAP} proposed an approach to investigating the free energy in the SK model by diagramatically expanding the free energy with respect to an effective magnetization and arrived at a new variational expression in terms of the TAP free energy, which involves a novel correlation energy of the spin fluctuations in addition to maintaining the spirit of the Gibbs variational principle and the mean-field approximation. From the first order optimality of the TAP free energy, they deduced a system of {\it self-consistent} equations, known as the TAP equations, where the solutions were predicted to contain crucial information about the landscapes of the TAP free energy as well as the SK Hamiltonian and have strong connections with Parisi's replica ansatz, see \cite{MPV87}.

In this paper, we aim to present an elementary derivation for the TAP equations in the generic mixed $p$-spin model and show that they are asymptotically satisfied by the local magnetization associated to the asymptotic pure states. Let $\beta=(\beta_p)_{p\geq 2}$ be a real sequence with $C_\beta :=\sum_{p=2}^\infty 2^p\beta_p^2 <\infty$ and let $h>0$ be fixed. For any $N\geq 1$ and $p\geq 2,$ denote by $\mathcal{I}_{N,p}$ the collection of all index vectors $(i_1,\ldots,i_p)\in \{1,\ldots,N\}^p$ of distinct entries. Let  $ g_{i_1,\ldots,i_p}$ be i.i.d. standard Gaussian for all $(i_1,\ldots,i_p)\in \mathcal{I}_{N,p}$, $N\geq 1$, and $p\geq 2.$ The Hamiltonian of the mixed $p$-spin model with (inverse) temperatures $(\beta_p)_{p\geq 2}$ and an external field $h$ is defined as 
\begin{align}
	\label{eqn:hamiltonian}
	{H}_N (\bsigma) = \sum_{p\ge 2} \frac{\beta_p}{N^{\frac{p-1}{2}}}\sum_{(i_1,\ldots,i_p)\in \mathcal{I}_{N,p}}g_{i_1,\ldots, i_p}\sigma_{i_1}\cdots \sigma_{i_p}+h\sum_{i=1}^N\sigma_i,
\end{align}
\noindent for $\bsigma = (\sigma_1, \ldots, \sigma_N) \in \bSigma_N:=\{-1,1\}^N$.
Note that when $\beta_p=0$ for all $p\neq 2$ and $\beta_2\neq 0,$ the mixed $p$-spin model recovers the famous SK model.
Using the independence among $g_{i_1,\ldots,i_p}$, it can be computed directly that the covariance of $ H_{N}$ is essentially determined by the function $\zeta(s) := \sum_{p \ge 2} \beta^2_p s^p$ and the overlap $R(\sigma^1,\sigma^2):=N^{-1}\sum_{i=1}^N\sigma_i^1\sigma_i^2$, namely,
\[
\Bigl|\text{Cov} \big( H_{N}(\bsigma^1), H_{N}(\bsigma^2)\big)-N \zeta \bigl(R(\sigma^1,\sigma^2)\bigr)\Bigr|\leq  C_\beta,\qquad\forall \bsigma^1, \bsigma^2 \in \bSigma_N.
\]
The free energy and the Gibbs (probability) measure are defined respectively by 
\[
F_N=\frac{1}{N}\ln Z_N\ \ \mbox{and}\ \ G_{N}(\bsigma) = \frac{e^{{H}_{N}(\bsigma)}}{ Z_N},
\]
where $Z_N :=\sum_{\bsigma \in \bSigma_N} e^{{H}_{N}(\bsigma)}.$
For any measurable function $f$ on $\mathbb R^{k}$ and independent samples $\bsigma^1, \ldots, \bsigma^k\sim G_N$ (also called the ``{\it replicas}'' in the physics literature), we denote by  $\gbs{f(\bsigma^1, \ldots, \bsigma^k)}$ the average under the Gibbs measure $G_{N}$, i.e., 
\[
\gbs{f(\bsigma^1, \ldots, \bsigma^k)} = \sum_{\bsigma^1, \ldots, \bsigma^k}f(\bsigma^1, \ldots, \bsigma^k) G_{N}(\bsigma^1)\cdots G_{N}(\bsigma^k).
\]
 It is well-known that the limiting free energy $F:=\lim_{N\to\infty}F_N$ exists and can be expressed as the Parisi formula (see \cite{panchenko2014parisi,talagrand2006parisi}), a variational representation of a nonlinear functional defined on the space of all probability measures on the interval $[0,1]$, in which the optimizer, called the Parisi measure $\mu_P$, exists and is unique, see \cite{AC15}. \footnote{The Hamiltonian of the mixed $p$-spin model defined in \cite{AC15,panchenko2014parisi, talagrand2006parisi} includes all indices $(i_1,\ldots,i_p)\in \{1,\ldots,N\}^p.$ However, the same conclusions remain valid in our setting since dropping the repeated indices is of a smaller order term in the free energy.}

Throughout this paper, we shall assume that the mixed $p$-spin model is {\it generic} in the sense that the linear span of the collection of all monomials, $t^p,$  for those $p$ with $\beta_p\neq 0$ is dense in $C[0,1]$ under the supremum norm. Under this assumption, this model possesses two important properties. First of all,  as $N \to \infty$, the limiting distribution of the overlap $R(\sigma^1,\sigma^2)$ under $\E\gbs{ \cdot}$ is described by the Parisi measure, see \cite{panchenko2008differentiability}. More importantly, the Gibbs measure satisfies the ultrametricity \cite{panchenko2013parisi}, namely, for any three replicas $\sigma^1,\sigma^2,\sigma^3$, we have 
\begin{align}
	\label{eqn:ultrametricity}
	\lim_{\epsilon\downarrow 0}\limsup_{N\to\infty}\E\bgbs{\mathbbm1_{\{R(\sigma^1,\sigma^2)\geq \min(R(\sigma^1,\sigma^3),R(\sigma^2,\sigma^3))-\epsilon\}}}=1.
\end{align}
As a consequence of ultrametricity, asymptotically, it allows one to decompose the spin configuration space $\Sigma_{N}$ into clusters and ultimately they induce an infinitary tree structure on $\Sigma_{N}$ under the set containments. In particular, the clusters at the bottom of the tree are called the {\it pure states} in the physics literature \cite{MPV87} --
they are essentially disjoint balls with the same radius  $\sqrt{2(1-q_P)}$ and the overlap between any two elements within the same pure state is about $q_P$, where $q_P=q_P(\beta)$  is the largest point in the support of $\mu_P.$  In addition, it is a well-known fact that if a metric space is ultrametric,  then in any ball (open or closed), every point is a center. 
In view of this and \eqref{eqn:ultrametricity}, a  natural way to construct asymptotic pure states is via the $\sqrt{2(1-q_P+\epsilon)}$ neighborhood of $\alpha\sim G_{N}$, that is,
\begin{align}\label{purestate}
\Sigma_{N}^\alpha :=\{\sigma\in\Sigma_N:\|\sigma-\alpha\|\leq \sqrt{2(1-q_P+\epsilon)}\} =\{\sigma \in \bSigma_N: R(\sigma, \alpha) > q_P-\epsilon\},
\end{align}
where $\|x\|:=\bigl(N^{-1}\sum_{i=1}^N|x_i|^2\bigr)^{1/2}$ for any $x\in \mathbb{R}^N$ and $\epsilon$ may be chosen to depend on $N$ as long as $\epsilon=\epsilon_{N}\downarrow 0$ slowly enough.
We define the conditional Gibbs measure $G_{N}^\alpha$ on the asymptotic pure state $\bSigma_{N}^\alpha$ by
\[
G_{N}^\alpha (\sigma) :=G_{N}(\bsigma|\bsigma \in \bSigma_{ N}^\alpha)=  \frac{G_{N}(\bsigma)\mathbbm 1_{\bSigma_{N}^\alpha}(\bsigma)}{G_{N}(\bSigma_{
		N}^\alpha)},
\]
where,by slightly abusing the notation, $G_{N}(A) :=\sum_{\sigma\in A} G_{N}(\sigma)$ is the total Gibbs measure for any subset $A\subseteq \bSigma_{N}$. The corresponding Gibbs average on $\Sigma_N^\alpha$ is denoted by $\gbs{
	\cdot}^\alpha$. Set the local magnetization within the asymptotic pure state $\Sigma^\alpha_N$ as
\begin{align*}
	\gbs{\sigma}^{\alpha}:=\bigl(\gbs{\sigma_1}^\alpha,\ldots,\gbs{\sigma_N}^\alpha\bigr).
\end{align*}
For any $1\leq i\leq N,$ let $\mathcal{I}_{N,p}(i)$ be the collection of all  $(i_1,\ldots,i_p)\in \mathcal{I}_{N,p}$ with $i_r=i$ for some $1\leq r\leq p.$ For any $1\leq i\leq N,$ define the cavity field at site $i$ as
\begin{align*}
	X_{N,\beta,i}(\sigma)&=\sum_{p\ge 2} \frac{\beta_p   }{N^{\frac{p-1}{2}}}\sum_{(i_1,\ldots,i_p)\in \mathcal{I}_{N,p}(i)} g_{i_1,\ldots,i_p}\prod_{s=1:i_s\neq i}^p\sigma_{i_s},\quad \forall \sigma\in\Sigma_N.
\end{align*}
Note that $X_{N,\beta,i}(\sigma)$ depends on all except the $i$-th spin.  Our main result is stated as follows:

\begin{theorem}[TAP equations] \label{thm:main}
	Assume that the model is generic. 
	Then the conditional local magnetization $\gbs{\sigma}^{\alpha}$ satisfies 
	\begin{align}
		\begin{split}
			\label{eqn:TAP}
			&\lim_{\epsilon \downarrow 0}\limsup_{N\to\infty}\E \Bigl\la\Bigl|\gbs{\sigma_{N}}^{\alpha}-\tanh \Bigl( 
			X_{N,\beta,N}(\gbs{\sigma}^{\alpha})+h -  \zeta''\bigl(\|\gbs{\sigma}^{\alpha}\|^2\bigr)\bigl(1-\|\gbs{\sigma}^{\alpha}\|^2\bigr)\gbs{\sigma_N}^{\alpha}\Bigr)\Bigr|^2\Bigr\ra= 0,
		\end{split}
	\end{align}
	where $\epsilon \downarrow 0$ along a sequence such that $q_P-\epsilon$ is always a point of continuity for $\mu_P$, i.e., $\mu_P(\{q_P-\epsilon\})=0.$
\end{theorem}

By symmetry, the expectation in \eqref{eqn:TAP} is the same for any spin site. Hence, as long as $N$ is large enough,  \eqref{eqn:TAP} implies that the local magnetization $\la \sigma\ra^\alpha$ satisfies the following asymptotic consistency equations on average over all spin sites $1\leq i\leq N,$
\begin{align}\label{TAP}
\gbs{\sigma_{i}}^{\alpha}\approx\tanh \Bigl( 
X_{N,\beta,i}(\gbs{\sigma}^{\alpha})+h -  \zeta''\bigl(\|\gbs{\sigma}^{\alpha}\|^2\bigr)\bigl(1-\|\gbs{\sigma}^{\alpha}\|^2\bigr)\gbs{\sigma_i}^{\alpha}\Bigr).
\end{align}
These self-consistent equations are the TAP equations. The term $\zeta''\bigl(\|\gbs{\sigma}^{\alpha}\|^2\bigr)\bigl(1-\|\gbs{\sigma}^{\alpha}\|^2\bigr)\gbs{\sigma_i}^{\alpha}$ is called the Onsager correction term, which distinguishes the disorder and non-disorder spin systems, where the latter, taking the Curie-Weiss model for example, does not involve an Onsager term in the mean-field equation for the spin magnetizations, see, e.g., Equation (II.2) in \cite{MPV87}.

The number of solutions to the TAP equations generally depends on the temperature. In the high temperature regime, i.e., $\zeta(1)$ is sufficiently small, it can be shown that the overlap $R(\sigma^1,\sigma^2)$ between two independent samples $\sigma^1,\sigma^2\sim G_{N}$ is concentrated around $q_P$ (see, e.g., \cite{Tal111}). Consequently, the pure state $\Sigma_N^\alpha$ associated to $\alpha\sim G_N$ is almost $\Sigma_N$ and $G_N(\Sigma_N^\alpha)$ is approximately one. As a result, we see that $\la \sigma\ra^\alpha\approx \la \sigma\ra$ and $\la \sigma\ra$ satisfies the TAP equations \eqref{TAP} in the generic mixed $p$-spin model at high temperature. This is similar to the one in the SK model established by Chatterjee \cite{Chatterjee10} and Talagrand \cite{Tal111}\footnote{Talagrand's result asserts that the $N$ TAP equations asymptotically hold simultaneously with high probability, while we establish the TAP equations in the average sense.}. 

In contrast to the high temperature case, it is expected that the TAP equations should have exponentially many solutions in $N$ in the low temperature regime. Although an argument for this remains missing in the literature, it can be argued that the TAP equations could have multiple solutions by making use of the ultrametricity \eqref{eqn:ultrametricity}. In fact, it is well-known that the Parisi measure $\mu_P$ possesses a nontrivial distribution.  Pick $q\in [0,q_P)$ in the support of $\mu_P.$ Note that in the generic mixed $p$-spin model, the {\it Ghirlanda-Guerra} identities are valid and they ensure that for any fixed $k\geq 1,$ with positive probability under $\e G_N^{\otimes k}$, there exist distinct $\alpha^1,\ldots,\alpha^k$ with $R(\alpha^\ell,\alpha^{\ell'})\approx q$ for all $1\leq \ell<\ell'\leq k$ (see, e.g., \cite[Theorem 2.20]{Pan13}). From the ultrametricity, it follows that $R(\sigma^\ell,\sigma^{\ell'})\approx q$ for any $\sigma^\ell\in \Sigma_{N}^{\alpha^\ell}$ and $\sigma^{\ell'}\in \Sigma_N^{\alpha^{\ell'}}$ for any $1\leq \ell<\ell'\leq k,$ which implies that $R(\la \sigma\ra^{\alpha^\ell},\la \sigma\ra^{\alpha^{\ell'}})\approx q$ for any $1\leq \ell<\ell'\leq k.$ As we will also see in Remark \ref{add:remark2} below that $R(\la \sigma\ra^{\alpha^\ell},\la \sigma\ra^{\alpha^\ell})$  is approximately $q_P$, we arrive at $\|\la \sigma\ra^{\alpha^\ell}-\la \sigma\ra^{\alpha^{\ell'}}\|\approx \sqrt{2(q_P-q)}>0.$ Hence,  $\la \sigma\ra^{\alpha^1},\ldots,\la \sigma\ra^{\alpha^k}$ are $k$ distinct approximate solutions to the TAP equations.

 Our approach for Theorem \ref{thm:main} is based on the cavity method, where the computation utilizes \eqref{eqn:ultrametricity}, the ultrametricity of the overlaps. This property allows us to partition the spin configuration space into random ``clusters'' $( C_{N,\alpha})_{\alpha \ge 1}$, i.e., the {pure states}, where the distribution of the weights of the pure states is characterized by the Ruelle Probability Cascades, see, e.g., \cite{Pan13,Tal112}. More importantly, within each pure state, the overlap of two spin configurations are asymptotically concentrated around $q_P$ \cite{J2017,panchenko2013parisi}. When approximating the pure states ${C}_{N,\alpha}$ via $\Sigma_N^\alpha,$ this concentration enables us to establish a central limit theorem for the cavity fields within a pure state and make our cavity computations feasible, ultimately leading to the TAP equations in \eqref{eqn:TAP}. 
Overall we believe that our approach has presented a general receipe in deriving the TAP equations for similar local magnetizations in related spin glass models.

\subsection{Related works}

{\bf The TAP Equations at High Temperature.} The mathematical establishment of the TAP equations appeared firstly in Talagrand's book \cite{Tal111} and Chatterjee \cite{Chatterjee10}, where they showed that the spin magnetization asymptotically satisfies the TAP equations in the SK model at sufficiently high temperature. Later Bolthausen \cite{Bolthausen14} proposed an iterative scheme and showed that his iteration would converge to one solution of the TAP equations, provided the temperature and external field parameters lie inside the {\it de Almeida-Thouless} phase transition line \cite{ATline}. In a joint paper, \cite{ChenTang2021}, the authors showed that Bolthausen's iteration scheme indeed converges to the spin magnetization whenever the overlap of two independently sampled spin configurations from the Gibbs measure is  concentrated locally uniformly in temperature. 
{Viewing the disorders as Brownian  motions and using tools from stochastic calculus,  Adhikari-Brennecke-von Soosten-Yau \cite{ABS21} established the TAP equations  for $\la \sigma_i \ra$ and $\la \sigma_i\sigma_j\ra - \la \sigma_i\ra \la \sigma_j\ra$ at sufficiently high temperature for the SK model and a version of the mixed $p$-spin models.}
\smallskip

{\color{black}\noindent {\bf The M\'ezard-Virasoro Equations.} Recently, Auffinger-Jagannath \cite{auffinger2019spin,AJ191} considered the generic mixed $p$-spin model and established a system of equations, discovered by M\'ezard-Virasoro \cite{mezard1985microstructure}, making a connection between the local magnetizations and the local fields. To be more precise, for each $N\geq 1,$ they constructed pure states $C_{N,a}\subseteq \Sigma_N$ for $a\geq 1$  and defined the local magnetizations as 
\begin{align}\label{AJpurestate}
\la \sigma_i\ra^a=\frac{\sum_{\sigma\in C_{N,a}}\sigma_iG_N(\sigma)}{G_N(C_{N,a})},\,\,\forall 1\leq i\leq N.
\end{align}
The M\'ezard-Virasoro equations read
\begin{align}\label{MVeq}
\la \sigma_N\ra^a\approx \tanh\Bigl(\bigl\la X_{N,\beta}(\sigma)\bigr\ra^a+h-(\zeta'(1)-\zeta'(q_P))\la \sigma_N\ra^a\Bigr).
\end{align}
Note that it is not always the case that  $\la X_{N,\beta}(\sigma)\ra^a\approx X_{N,\beta}(\la\sigma\ra^a)$. As a result, the M\'ezard-Virasoro equations is different from Theorem \ref{thm:main} in nature. See also a comparison of the cavity method used in \cite{auffinger2019spin} and in present work (Section \ref{CavityMethodComparison}).
}

\smallskip

{\noindent\bf The TAP Free Energy.} In \cite{CP18}, Chen-Panchenko established a variational principal for the free energy in terms of the classical TAP free energy and showed that the local magnetization $\la \sigma\ra^\alpha$ is asymptotically a {\it TAP state}, i.e., an optimizer of the TAP free energy.  Later in \cite{CPS18}, Chen-Panchenko-Subag  derived a general framework for the TAP free energy in the mixed $p$-spin model and from the first order optimality, it was argued that every TAP state must satisfy the generalized TAP equations. Combining the result in \cite{CP18} and the discussion of \cite[Remark 6]{CPS18} together, it is then expected that the classical TAP equations, \eqref{TAP}, should be satisfied by our local magnetization $\la \sigma\ra^\alpha$ --  Theorem \ref{thm:main} provides a justification to this prediction.

In relation to the TAP equations, the TAP free energies have also received great attention in recent years. By utilizing Parisi's ansatz, the works of Chen-Panchenko \cite{CP18} and Chen-Panchenko-Subag \cite{CPS18,CPS19} established the TAP free energy in the Ising mixed $p$-spin model as well as its generalization to the zero temperature setting. In the spherical case, Subag established the TAP free energy in the mixed $p$-spin model \cite{Subag2018} and in the multi-species model \cite{subag2021tap,subag2021tap2} by introducing the {\it multisamplable overlap} property. Independently, by means of a geometric microcanonical method, the TAP free energy involving the {\it Plefka condition} was implemented in the spherical SK model by Belius-Kistler \cite{belius2019tap} and an upper bound for the free energy in terms of the TAP free energy was also obtained in the mixed $p$-spin model with general spins by Belius \cite{belius2022high}.

\subsection{Further results related to the TAP equations}
The TAP equations have many important consequences in the study of mean-field spin glass models and related applications. First of all, based on his iteration, Bolthausen \cite{Bolthausen19} performed a conditional second moment method to derive the replica symmetry formula for the limiting free energy in the SK model at very high temperature.  In a follow-up work,  Brennecke-Yau \cite{BrenneckeYau21} provided a simplified argument for Bolthausen's approach and extended the replica symmetry formula to a larger regime. In addition, the TAP equations have played a key role in some optimization problems in spin glasses and statistical inference problems. Most importantly, they naturally give rise to the so-called Approximate Message Passing (AMP) algorithms based on Bolthausen's iteration scheme \cite{Bolthausen14}; several generalizations of the AMP algorithms can be found in Bayati-Montanari \cite{bayati2011dynamics} and Javanmard-Montanari \cite{javanmard2013state}. By using the AMP algorithms, Montanari \cite{montanari2021optimization} constructed a polynomial-time random algorithms to produce a near ground state for the SK Hamiltonian under the assumption that the Parisi measure is full replica symmetry breaking. The same construction was also carried out in the mixed $p$-spin model by El Alaoui-Montanari-Sellke \cite{el2021optimization}. In the context of Bayesian inferences, various AMPs driven by the TAP equations have also been popularly used, see, e.g., \cite{kabashima2016phase,montanari2015non,montanari2021estimation,zdeborova2016statistical}.

\subsection{Organization of the paper}

In Section \ref{proofsketch}, we provide a sketch of our proof  beginning with the settlement of some standard notations.
Section  \ref{sec:cavity1} will be devoted to establishing a cavity equation for the local magnetization $\la \sigma\ra^\alpha$ based on a univariate central limit theorem for the cavity field, see Theorem \ref{lemma4-CLT} below. In Section \ref{sec:cavitytoTAP1}, we continue to perform some cavity computations for $X_{N,\beta,N}(\la \sigma\ra^\alpha).$ In particular, noting that this quantity is an infinite series, we shall establish some quantitative bounds for its truncation and {\color{black}moments}.  Finally, the proof of Theorem \ref{thm:main} will be presented in Section \ref{sec:cavitytoTAP2} by translating the cavity equation in Section \ref{sec:cavity1} into the TAP equation. This part of the argument will rely on a multivariate central limit theorem for multiple cavity fields, a central ingredient throughout the entire paper.\\

\medskip 

{\bf \noindent Acknowledgments.} W.-K. Chen's research is partly supported by NSF grant (DMS-1752184) and the Simons fellowship (\#1027727). S. Tang's research is  partly supported by the Simons Collaboration Grant (\#712728)  and the NSF LEAPS-MPS Award (DMS-2137614).
Both authors thank A. Auffinger for explaining their work \cite{auffinger2019spin} to us and M. Sellke for pointed out a few typos. They also thank anonymous referees for carefully reading the manuscript and providing valuable comments.

\section{Proof Sketch}\label{proofsketch}

Our approach adapts a similar route as \cite[Theorem 1.7.7]{Tal111}, where the TAP equations in the SK model were established at very high temperature relying on the fact that the overlap $R(\sigma^1,\sigma^2)$ between two replicas $\sigma^1,\sigma^2\sim G_N$ is concentrated around a deterministic constant. However, this concentration is no longer valid in the low temperature regime. Our proof utilizes,  instead, a conditional concentration property of the overlap, deduced from the ultrametricity \eqref{eqn:ultrametricity}, that the overlap between two independently sampled spin configurations from the conditional Gibbs measure on an approximate pure state is  concentrated around $q_P.$ In this section, we elaborate some key steps to summarize the main ideas of our arguments. We believe that our approach is potentially applicable to establish the TAP equations in {\color{black} some other mean-field spin glass models, such as the multi-species model \cite{barra2015multi} and the perceptron model \cite[Chapter 2]{Tal111}.}  First, we  settle down our standard notations.

\subsection{Notations}\label{proofsketch:sub1}

\begin{notation}\rm Recall that $H_N$, $G_{N}$, $G_N^\alpha$, $Z_N,$ $\la \cdot\ra,$ and $\la \cdot\ra^{\alpha}$  
 are all dependent on the temperature parameter $\beta=(\beta_p)_{p\geq 2}$ and the external field $h$. For the rest of this paper, we will always keep $h$ fixed and express these quantities as $H_{N,\beta}$, $G_{N,\beta},$ $G_{N,\beta}^\alpha$, $Z_{N,\beta},$ $\la \cdot\ra_{N,\beta},$ and $\la \cdot\ra_{N,\beta}^\alpha$,  respectively, to emphasize their dependence on both $N$ and $\beta.$ 
\end{notation}

\begin{notation}
	\rm In this paper, we shall always work with  the cavity field associated to the last spin, $X_{N,\beta,N}.$ Due to the symmetry in the spin sites in  the Hamiltonian ${H}_{N,\beta}$, this cavity field can be simplified as
	$$
	X_{N,\beta,N}(\sigma)=\sum_{p\geq 2}\frac{\beta_p\sqrt{p!}}{N^{\frac{p-2}{2}}}\sum_{1\leq i_1<i_2<\cdots<i_{p-1}\leq N-1}g_{i_1,\ldots,i_{p-1},N}\sigma_{i_1}\cdots\sigma_{i_{p-1}}.
	$$ 
	 Since  it is indeed a function that does not depend on the $N$-th  spin coordinate, for notational simplicity, for either  $x\in [-1,1]^{N-1}$ or $x\in [-1,1]^N$, we will simply denote it by $X_{N,\beta}(x)$, that is,
	\begin{align}
		\begin{split}
			\label{eqn:defXn}
			X_{N,\beta}(x)
			&:=\sum_{p\geq2}\frac{\beta_p\sqrt{p!}}{N^{\frac{p-1}{2}}}\sum_{1\leq i_1<\cdots <i_{p-1}\leq N-1}g_{i_1,\ldots,i_{p-1},N}x_{i_1}\cdots x_{i_{p-1}}.
		\end{split}
	\end{align}
\end{notation}

\begin{notation} \rm For any $\alpha\sim G_{N,\beta}$, the notations $\rho$ and $\alpha_N$ stand for the first $N-1$ spins and the last spin of $\alpha$, respectively.
	When there is no ambiguity, we will generally use $\tau,\tau^1,\tau^2,\ldots \in \Sigma_{N-1}$ to denote the spin configurations that  are sampled  independently from $G_{N-1,\beta'}$ or $G_{N-1,\beta'}^\rho$ for any $\beta'=(\beta_p')_{p\geq 2}$ satisfying that $\sum_{p\geq 2}2^p{\beta_p'}^2<\infty.$
\end{notation}

\begin{notation}\rm
	For any two vectors $x,y\in \mathbb{R}^M$ for some $M\geq 1,$ we define $R(x,y)$ as the averaged inner product between $x$ and $y$ and $\|x\|$ as the averaged $\ell_2$-norm of $x$ respectively by $$R(x,y)=\frac{1}{M}\sum_{i=1}^Mx_iy_i\quad\mbox{and}\quad \|x\|=R(x,x)^{1/2}.$$ 
	For any two sequences $(a_N)_{N\geq 1}$ and $(b_N)_{N\geq 1}$ of real numbers or random variables, we say that $a_N=b_N+o_N(1)$ (or $a_N\leq b_N+o_N(1)$) if $|a_N-b_N|\leq c_N$ (or $a_N\leq b_N+c_N$) for all $N\geq 1$, where $(c_N)_{N\geq 1}$ is some deterministic positive sequence that converges to $0$ as $N\to\infty$. 
\end{notation}

\subsection{Proof sketch of Theorem \ref{thm:main}}

First of all, by making use of the notations in Subsection \ref{proofsketch:sub1}, we rewrite \eqref{eqn:TAP} as
\begin{align*}
	\lim_{\epsilon\downarrow 0}\limsup_{N\to\infty}\E\Bigl\la\Big[\gbs{\sigma_N}_{N,\beta}^\alpha-\tanh &\Big( 
	X_{N,\beta}(\gbs{\sigma}_{N,\beta}^\alpha)+h \\
	\notag
	&-  \zeta''(\|\gbs{\sigma}_{N,\beta}^\alpha\|^2)(1-\|\gbs{\sigma}_{N,\beta}^\alpha\|^2)\gbs{\sigma_N}_{N,\beta}^\alpha\Big)\Big]^2\Bigr\ra_{N,\beta}= 0.
\end{align*}
On the other hand, from the ultrametricity \eqref{eqn:ultrametricity}, it can be deduced (see Remark \ref{add:remark2}) that
\begin{align*}
	\lim_{N\to\infty}\E\bigl\la\big|\|\la \sigma\ra_{N,\beta}^\alpha\|^2-q_P\bigr|^2\bigr\ra_{N,\beta}=0.
\end{align*}
As a result, we only need to show that
\begin{align}
	\label{eqn:TAP:eq1}
	\lim_{\epsilon\downarrow 0}\limsup_{N\to\infty}\E \Bigl\la\Big[\gbs{\sigma_N}_{N,\beta}^\alpha-\tanh \Big( 
	X_{N,\beta}(\gbs{\sigma}_{N,\beta}^\alpha)+h -  \zeta''(q_P)(1-q_P)\gbs{\sigma_N}_{N,\beta}^\alpha\Big)\Big]^2\Bigr\ra_{N,\beta}= 0.
\end{align}
There are three key steps in our argument:

\medskip

{\noindent \bf Step 1: Cavity computation.} In Section 
\ref{sec:cavity1}, we decompose the Hamiltonian \eqref{eqn:hamiltonian} into two components. The first involves only the interactions among the first $N-1$ coordinates and the second gathers the interactions with the $N$-th spin. More precisely, if we fix the last spin as the cavity coordinate, then for $\sigma=(\tau,\sigma_N)\in\Sigma_{N-1}\times\Sigma_1,$ we can rewrite
\begin{align}\label{add:eq---3}
	{H}_{N,\beta}(\sigma)={H}_{N-1,\beta'}(\tau)+\sigma_N\bigl(X_{N,\beta}(\tau)+h\bigr),
\end{align}
where $\beta'=(\beta_p')_{p\geq 2}$ is a modification of $\beta$ defined as $\beta_p'=\beta_p((N-1)/N)^{(p-1)/2}$ for each $p\ge 2$. In this way, we can express $\gbs{\sigma_N}_{N,\beta}^\alpha$ as a fraction of averages with respect to the  Gibbs measure associated to the reduced Hamiltonian ${H}_{N-1,\beta'}$ in terms of two constraints, $\mathbbm{1}_{\Sigma_N^\alpha}\bigl((\tau,1)\bigr)$ and $\mathbbm{1}_{\Sigma_N^\alpha} \bigl((\tau,-1)\bigr)$. We show that when writing $\alpha=(\rho ,\alpha_N)$, both of them are approximately $\mathbbm{1}_{\Sigma_{N-1}^{\rho }}(\tau)$ due to the fact that $R(\tau, \rho) \approx R((\tau, \pm 1), \alpha)$. This will result in 
\begin{align}\label{add:eq1}
	\gbs{\sigma_{N}}_{N,\beta}^\alpha \approx \frac{ \gbs{\sinh(X_{N,\beta}(\btau)+h)}_{N-1,\beta'}^{\rho } }{\gbs{\cosh(X_{N,\beta}(\btau)+h)}_{ N-1,\beta'}^{\rho }}.
\end{align}

{\noindent \bf Step 2: Central Limit Theorem of cavity fields.}
We notice that for any $\rho \in \Sigma_{N-1}$ fixed, the disorders, $g_{i_1,\ldots, i_{p-1},N}$, in $X_{N,\beta}$ are independent of the conditional Gibbs measure $G_{N-1,\beta'}^{\rho }$ and furthermore, for independent samples $\tau^1,\tau^2\sim G_{N-1,\beta'}^{\rho }$, the overlap $R(\tau^1,\tau^2)$ is concentrated around $q_P$, due to ultrametricity \eqref{eqn:ultrametricity}. These allow us to show (see Theorem \ref{lemma4-CLT}) that the first term in the following representation
\[
X_{N,\beta}(\btau) = \bigl(X_{N,\beta}(\btau)- \gbs{X_{N,\beta}(\btau)}_{N-1,\beta'}^{\rho } \bigr) + \gbs{X_{N,\beta}(\btau)}_{N-1,\beta'}^{\rho }
\]
is Gaussian distributed with some variance, approximately $
\zeta'(1) - \zeta'(q_P)$, 
independent of $\tau$. 
Consequently, we can represent $$X_{N,\beta}(\btau) \approx z \sqrt{ \zeta'(1) - \zeta'(q_P) }  + \gbs{X_{N,\beta}(\btau)}_{N-1,\beta'}^{\rho },\quad \text{ for }z\sim N(0,1). $$
It follows that
\begin{align*}
	\gbs{\sinh(X_{N,\beta}(\btau)+h)}_{N-1,\beta'}^{\rho } &\approx \E_z \sinh (z  \sqrt{ \zeta'(1) - \zeta'(q_P) } + \gbs{X_{N,\beta}(\btau)}_{N-1,\beta'}^{\rho }+h) \\
	&= e^{( \zeta'(1) - \zeta'(q_P) )/2}\sinh ( \gbs{X_{N,\beta}(\btau)}_{N-1,\beta'}^{\rho }+h),
\end{align*}
and 
\begin{align*}
	\gbs{\cosh(X_{N,\beta}(\btau)+h)}_{N-1,\beta'}^{\rho } &\approx \E_z \cosh ( \sqrt{ \zeta'(1) - \zeta'(q_P) }  z + \gbs{X_{N,\beta}(\btau)}_{N-1,\beta'}^{\rho }+h) \\
	&\approx e^{ ( \zeta'(1) - \zeta'(q_P) )/2}\cosh ( \gbs{X_{N,\beta}(\btau)}_{N-1,\beta'}^{\rho }+h),
\end{align*}
which, combining with \eqref{add:eq1}, gives \begin{align}\label{add:eq:Step2}
\gbs{\sigma_N}_{N,\beta}^\alpha \approx \tanh( \gbs{X_{N,\beta}(\btau)}_{N-1,\beta'}^{\rho }+h).
\end{align}In establishing this Central Limit Theorem, there is a subtle issue that occurs from the fact that $\alpha \in \Sigma_N$ is not fixed (and neither is $\rho \in \Sigma_{N-1}$) but sampled from $G_{N,\beta}$. Since the Gaussian disorders $g_{i_1,\ldots, i_{p-1},N}$ appearing in $X_{N,\beta}$ are also involved in the Gibbs measure $G_{N,\beta}$, this introduces dependence between $X_{N,\beta}(\tau)$ and $G_{N,\beta}$. To handle this, we perform a change of measure for $\rho $, the first $(N-1)$ coordinates of $\alpha$ and control the associated Radon-Nykodym derivative, which allows us to prove a version of Central Limit Theorem for the Gibbs measure conditioned on the asymptotic pure states. This is done in Section \ref{sec:cavity2}, where the ultrametricity of the Gibbs measure plays a key role, ensuring that the overlap between two independently sampled spin configurations from the conditional Gibbs measure $G_{N-1,\beta'}^{\rho}$ on the asymptotic pure state $\Sigma_{N-1}^{\rho}$ is approximately concentrated around $q_P.$ This is analogous to the concentration of the overlap $R(\sigma^1,\sigma^2)$ in the SK model at high temperature.

\medskip

{\noindent \bf Step 3: Producing the Onsager correction term.} In Section \ref{sec:cavitytoTAP2}, we relate $\gbs{X_{N,\beta}(\tau)}_{N-1,\beta'}^{\rho }$ with $X_{N,\beta}(\gbs{\sigma}_{N,\beta}^{\alpha})$ and show that their difference is exactly the Onsager correction term $\zeta''(q_P)(1-q_P)\gbs{\sigma_N}_{N,\beta}^\alpha$. Due to the infinite sum over $p\ge 2$ and the product $\tau_{i_1}\cdots \tau_{i_{p-1}}$ in $X_{N,\beta}(\tau)$, handling the mixed-$p$ spin model is much more involved than the SK model. Here, we will first truncate the infinite sum over $p\ge 2$ and approximate  $\gbs{X_{N,\beta}(\btau)}_{N-1,\beta'}^{\rho }$ and $X_{N,\beta}(\gbs{\sigma}_{N,\beta}^{\alpha})$, respectively, by finite sums  of sufficiently many terms; this requires controlling the moments of both quantities. With such truncation, we are able to estimate the difference term-by-term for each $p\ge 2$ fixed. More precisely, by definition \eqref{eqn:defXn}, each term in $X_{N,\beta}(\gbs{\sigma}_{N,\beta}^{\alpha})$ takes the form 
\begin{align}
	\label{eqn:p3case}
	\frac{\beta_p \sqrt{p!} }{N^{(p-1)/2}}\sum_{1\le i_1<\cdots<i_{p-1} \le N-1} g_{i_1,\ldots,i_{p-1},N}\gbs{\sigma_{i_1}}_{N,\beta}^{\alpha}\cdots \gbs{\sigma_{i_{p-1}}}_{N,\beta}^{\alpha}.
\end{align}
We will follow an argument similar to that in our first step to write each conditioned local magnetization in  \eqref{eqn:p3case} as
\[
\gbs{\sigma_{i }}_{N,\beta}^\alpha \approx \frac{\gbs{\tau_{i}\cosh(X_{N,\beta}(\tau)+h)}_{N-1,\beta'}^{\rho }}{\gbs{\cosh(X_{N,\beta}(\tau)+h)}_{N-1,\beta'}^{\rho }},
\]
and then use replicas to represent  \eqref{eqn:p3case} as
\begin{align}\label{add:eq2}
	\frac{\bgbs{\bigl(\sum_{1\le i_1<\ldots<i_{p-1} \le N-1} \frac{\beta_p \sqrt{p!}}{N^{(p-1)/2}}g_{i_1,\ldots,i_{p-1},N}\tau_{i_1}^1\cdots\tau_{i_{p-1}}^{p-1}\bigr) \,\prod_{\ell=1}^{p-1}\cosh(X_{N,\beta}(\tau^\ell)+h) }_{N-1,\beta'}^{\rho }}{\bgbs{\prod_{\ell=1}^{p-1}\cosh(X_{N,\beta}(\tau^\ell)+h)}_{N-1.\beta'}^{\rho }}.
\end{align}
To proceed, a multivariate Central Limit Theorem (see Theorem \ref{thm:mvclt}) is needed to handle the asymptotic behavior of the jointly Gaussian-distributed random variables in \eqref{add:eq2}, namely,
\begin{align}
	\label{add:eq3}
	\sum_{1\le i_1< \cdots < i_{p-1} \le N-1} \frac{\beta_p \sqrt{p!} }{N^{(p-1)/2}}g_{i_1\cdots i_{p-1},N}\tau_{i_1}^1\cdots \tau_{i_{p-1}}^{p-1},\ \  X_{N,\beta}(\tau^1), \ \ldots, \ X_{N,\beta}(\tau^{p-1}),
\end{align}
which will allow us to handle the numerator and denominator in \eqref{add:eq2} separately as in Step 2 so that \eqref{add:eq2} is essentially  $\gbs{X_{N,\beta}(\btau)}_{N-1,\beta'}^{\rho }$ plus an Onsager correction term corresponding to the pure $p$-spin interaction. We emphasize that this step is not needed for deriving the TAP equation in the SK model at high temperature, e.g., \cite[Theorem 1.7.7]{Tal111}, because in this case, one only needs to deal with the $p=2$ case and the collection of Gaussian random variables in \eqref{add:eq3} reduces to a single one $X_{N,\beta}(\tau^1)$, consequently. A univariate central limit theorem (e.g., Theorem \ref{lemma4-CLT}) would suffice.

{\color{black}\subsection{Cavity method: A comparison}\label{CavityMethodComparison}

As mentioned before, \cite{auffinger2019spin,AJ191} established the M\'ezard-Virosoro equations. Their argument required $\mu_P(\{q_P\})>0$, which ensured the validity of the pure state decomposition, $(C_{N,a})_{a\geq 1}$, with decreasing Gibbs weights, see \cite{J2017,talagrand2010construction}. Based on this, they defined the local magnetizations $\la \sigma\ra^a$ through \eqref{AJpurestate} and performed a cavity argument for this quantity as follows. In view of \eqref{add:eq---3}, denoting by $G_N'$ the Gibbs measure associated to  $({H}_{N-1,\beta'}(\tau))_{\tau\in \Sigma_{N-1}}$ and by $(C_{N-1,a}')_{a\geq 1}$ the pure state decomposition according to $G_N'$, they decomposed their pure states as 
\begin{align}
    \label{add:eq---6}
    C_{N,a}\approx C_{N-1,\pi(a)}'\times\{-1,1\}
\end{align} under the Gibbs measure $G_N$. The function $\pi$ is a random permutation of the natural numbers $\mathbb{N}$, depending on all disorders in ${H}_N$, such that the sequence $(C'_{N-1,\pi(a)}\times\{-1,1\})_{a\geq 1}$ would also have decreasing measures under $G_N$. Consequently, they wrote
\begin{align}\label{add:eq---4}
\la \sigma_N\ra^a\approx\int \sigma_N\nu_{N,a},
\end{align}
where $\nu_{N,a}$ was the probability measure defined as
$$
\nu_{N,a}(\tau,\sigma_N)\propto e^{\sigma_NX_{N,\beta}(\tau)}G_{N-1}'(\tau|C_{N-1,\pi(a)}')$$
for $ (\tau,\sigma_N)\in \Sigma_N=\{-1,1\}^{N-1}\times\{-1,1\}$. 

Although the arguments for \eqref{add:eq---4} and Equation \eqref{add:eq1} in our Step 1 are both based on the cavity method, their derivations are fundamentally different. In \cite{auffinger2019spin}, the key step to justify \eqref{add:eq---4} relied on establishing \eqref{add:eq---6} under the Gibbs measure $G_N$, which utilized various properties of the pure state decomposition, including the fact that the Gibbs weights of the pure states asymptotically form a Poisson-Dirichlet process and the assumption $\mu_P(\{q_P\})>0$. In contrast, our constructions of the pure states and the local magnetizations are explicit so that it is easier to quantify the error estimates in the cavity computation; our Step 1 neither needs the pure state decomposition nor the assumption $\mu_P(\{q_P\})>0.$

We then proceed to Step 2 and conclude it with Equation \eqref{add:eq:Step2}. Note that Equation \eqref{add:eq:Step2} already deviates from the M\'ezard-Virasoro equations \eqref{MVeq}, where we average inside the hyperbolic tangent function with respect to the $(N-1)$-dimensional Gibbs measure, $\la X_{N,\beta}(\tau)\ra_{N-1,\beta'}^\rho$  instead of an $N$-dimensional Gibbs measure, $\la X_{N,\beta}(\sigma)\ra^a$ in \eqref{MVeq}. This is crucial, as it allows us to further express $\la X_{N,\beta}(\tau)\ra_{N-1, \beta'}^\rho$ as $ X_{N,\beta}(\la\sigma\ra^\alpha)$ plus the desired Onsager correction term in Step 3. 
Finally, we emphasize that in \cite{auffinger2019spin}, it is unclear how to rewrite the term $\la X_{N,\beta}(\sigma)\ra^a$ in the M\'ezard-Virasoro equations \eqref{MVeq} to get to the TAP equations, but we expect that it should be possible to show that the local magnetizations in \cite{auffinger2019spin} satisfy the same TAP equations \eqref{eqn:TAP} following our Steps 2 and 3.}

\section{\label{sec:cavity1}The Cavity equation for $\la \sigma_N\ra_{N,\beta}^\alpha$}

We derive the cavity equation for  $\gbs{\bsigma_N}_{ N, \beta}^\alpha$ in this section. As sketched in the previous section, for any $\sigma=(\tau,\sigma_N)\in \{\pm 1\}^{N-1}\times \{\pm 1\},$ {\color{black}by symmetry, we can decompose the Hamiltonian in \eqref{eqn:hamiltonian} as}
\begin{align*}
	 H_{N, \beta} (\bsigma) = &\sum_{p\ge 2} \frac{\beta_p\sqrt{p!}}{N^{\frac{p-1}{2}}}\sum_{1\le i_1 <\cdots < i_p \le N-1}g_{i_1,\ldots, i_p}\sigma_{i_1}\cdots \sigma_{i_p} + h \sum_{i=1}^{N-1} \sigma_i \\
	&+ \sigma_N \Bigl(\sum_{p\ge 2} \frac{\beta_p \sqrt{p!}  }{N^{\frac{p-1}{2}}}\sum_{1\le i_1 <\cdots< i_{p-1} \le N-1} g_{i_1,\ldots, i_{p-1},N}\sigma_{i_1}\cdots \sigma_{i_{p-1}} + h\Bigr)\\
	&= H_{N-1, \beta'}(\btau) + \sigma_N \big(X_{N, \beta}(\btau)+h\big),
\end{align*}
where $\beta' := (\beta_p')_{p\ge 2}$ with $\beta'_p = \beta_p ((N-1)/N)^{(p-1)/2}$ for each $p\ge 2$ and $X_{N,\beta}$ is defined by \eqref{eqn:defXn}. For any $\tau \in \Sigma_{N-1}$ fixed, $X_{N,\beta}(\tau)$ is a centered Gaussian random variable with variance 
\begin{align}
	\nonumber		\E X_{N,\beta}^2(\tau) &= \sum_{p\ge 2} \frac{\beta_p^2 p!}{N^{p-1}}\sum_{1\le i_1< \cdots <i_{p-1}\le N-1}(\tau_{i_1}\cdots \tau_{i_{p-1}})^2\\
 \nonumber&=\sum_{p\ge 2} \frac{\beta_p^2 p!}{N^{p-1}}\frac{(N-1)(N-2)\cdots (N-p+1)}{(p-1)!}\\
	\label{add:eq4}		&=\sum_{p\ge 2}\beta_p^2 p \prod_{l=1}^{p-1}\Bigl(1-\frac{l}{N}\Bigr)\nearrow\sum_{p\ge 2}\beta_p^2 p,\ \text{ as } N \to\infty.
\end{align}
We remark that the variance of $X_{N,\beta}(\tau)$ is bounded uniformly for any $N\ge 1$ and any $\tau \in \Sigma_{N-1}$  by $C_{\beta}$ defined at the beginning. The goal of this section is to establish the following theorem. 

\begin{theorem}[Cavity equation] \label{thm:cavity}Let $\alpha \sim G_{N,\beta}$ and write $\alpha = (\rho, \alpha_{N})$. Then, 
	\begin{align*}
		\lim_{\epsilon \downarrow 0}\limsup_{N\to\infty}\E \Bigl\la\Big(\gbs{\sigma_N}_{N, \beta}^{\alpha}-\tanh \big( X_{N,\beta}(\gbs{\tau}_{N-1, \beta'}^{\rho})+h\big)  \Big)^2\Bigr\ra_{N,\beta} = 0,
	\end{align*}
	where $\epsilon \downarrow 0$ along a sequence such that $q_P-\epsilon$ is always a point of continuity for $\mu_P$.
\end{theorem}

In other words, $\la \sigma_N\ra_{N,\beta}^\alpha$ is asymptotically a function of the vector $\la \tau\ra_{N-1,\beta'}^{\rho}$, the average with respect to the conditional Gibbs measure of an $(N-1)$ system. Here, the dimension of the system in $\la \cdot\ra_{N-1,\beta'}^{\rho}$ is less than that of $\la \cdot\ra_{N,\beta}^{\alpha}$ by 1. Furthermore, the last spin $\alpha_N$ is absent in the constrained Gibbs measure $\la \cdot\ra_{N-1,\beta'}^{\rho}$. We establish this theorem in three subsections.

\subsection{Removal of the cavity spin constraint I}

The conditional Gibbs average $\gbs{\bsigma_N}_{N, \beta}^{\alpha}$ can be written as 
\begin{align}
	\notag
	\gbs{\sigma_N}_{N, \beta}^{\alpha} &= \frac{\sum_{\bsigma \in \bSigma_{N}^{\alpha}} \sigma_N G_{N,\beta}(\bsigma)}{\sum_{\bsigma \in \bSigma_{N}^{\alpha}} G_{N,\beta}(\bsigma)}\\
	\label{eqn:cavity1}
	&=\frac{\bgbs{e^{X_{N,\beta}(\btau)+h}\mathbbm 1_{\bSigma_{N}^{\alpha}}\big((\btau, 1)\big) - e^{-X_{N,\beta}(\btau)-h}\mathbbm 1_{\bSigma_{N}^{ \alpha}}\big((\btau, -1)\big)\,}_{N-1, \beta'}}{\bgbs{e^{X_{N,\beta}(\btau)+h}\mathbbm 1_{\bSigma_{N}^{\alpha}}\big((\btau, 1)\big) + e^{-X_{N,\beta}(\btau)-h}\mathbbm 1_{\bSigma_{N}^{\balpha}}\big((\btau, -1)\big)\,}_{N-1, \beta'}},
\end{align}
and, for $\alpha=(\rho,\alpha_N)$, we have
\begin{align*}
	\{\btau\in \bSigma_{N-1}: (\btau, \pm1)\in \bSigma_{N}^{\balpha}\} =\Big\{ \btau \in \bSigma_{N-1}: R(\btau, \rho) \ge q_P - \epsilon + \frac{1}{N-1}(q_p - \epsilon \mp\alpha_N)\Big\},
\end{align*}
Thus, by letting
\begin{align*}
	A_{\ominus}^{\rho} &:=\Bigl\{\tau \in \bSigma_{N-1}: R(\rho, \tau) \ge q_P - \Bigl(\epsilon -\frac{3}{N-1} \Bigr)\Bigr\}, \\
	A_{\oplus}^{\rho} &:=\Bigl\{\tau \in \bSigma_{N-1}: R(\brho, \tau) \ge q_P - \Bigl(\epsilon +\frac{3}{N-1} \Bigr)\Bigr\}, 
\end{align*}
we can easily see that for all $\tau\in\Sigma_{N-1}$ and $\alpha=(\rho,\alpha_N)\in\Sigma_{N-1}\times\Sigma_1,$
\begin{align}
	\begin{split}		\label{eqn:nestedsets}
		&\mathbbm{1}_{A_{\ominus}^{\rho} } (\tau)\leq \mathbbm{1}_{\bSigma_{N-1}^{\rho}}(\tau)\leq \mathbbm{1}_{A_{\oplus}^{\rho}}(\tau),\\
		&	\mathbbm{1}_{A_{\ominus}^{\rho}}(\tau)\leq \mathbbm{1}_{\Sigma_N^\alpha}((\tau,-1)),\mathbbm{1}_{\Sigma_N^\alpha}((\tau,1))\leq \mathbbm{1}_{A_\oplus^{\rho}}(\tau).
	\end{split}
\end{align}
We will approximate the numerator and denominator of \eqref{eqn:cavity1} by substituting all indicators there by $\mathbbm 1_{\Sigma_{N-1}^{\rho}}(\tau)$. To this end, we first prepare three lemmas. \\

Denote by $g_{\cdot N}$ the collection of all Gaussian disorders that appear in $X_{N,\beta}$,
\[
g_{i_1, \ldots, i_{p-1}, N},\ \  \text{for }p=2,3,\ldots \text{and} \ 1\le i_1 < \cdots < i_{p-1} \le  N-1
\]
and by $\E_{g_{\cdot N}}$ the expectation (only) with respect to $g_{\cdot N}.$

\begin{lemma}
	\label{add:lem1} Let $\alpha = (\rho, \alpha_{N})$ be sampled according to $G_{N,\beta}$ and 
	$(\Gamma(\rho))_{\rho\in \Sigma_{N-1}}$ be a collection of nonnegative random variables, which are independent of the Gaussian disorders appearing in $X_{N,\beta}.$ Then we have
	\begin{align*}
		\E\bigl\la \Gamma(\rho)\bigr\ra_{N,\beta}&=\e\Bigl\la \Gamma(\tau)\frac{\cosh\left( X_{N,\beta}(\tau)+h\right)}{\gbs{\cosh \left(X_{N,\beta}(\tau)+h\right)}_{N-1, \beta'}}\Bigr\ra_{\!N-1,\beta'}\leq e^{2^{-1}\sum_{p\ge 2}\beta_p^2 p}\cosh(h)\E\bigl\la \Gamma(\tau)\bigr\ra_{N-1,\beta'}.
	\end{align*}
\end{lemma}

\begin{proof} Write
	\begin{align*}
		\E \bgbs{\Gamma(\rho)}_{N,\beta}&=\E \sum_{\tau \in \bSigma_{N-1}}\Gamma(\tau)\big[G_{N,\beta}((\tau,1)) + G_{N,\beta}((\tau,-1))\big] \\
		\notag &=\E \sum_{ \tau \in \bSigma_{N-1}}\Gamma(\tau) G_{N-1,\beta'}(\tau)\frac{G_{N,\beta}((\tau,1)) + G_{N,\beta}((\tau,-1))}{G_{N-1, \beta'}(\tau)},
	\end{align*}
	where the last fraction is the Radon-Nykodym derivative of the two measures $G_{N,\beta}((\cdot, 1))+G_{N,\beta}((\cdot, -1))$ and $G_{N-1,\beta'}(\cdot)$ on $\Sigma_{N-1}$. Computing the numerator and the denominator at $\tau$, we get
	\begin{align*}
		G_{N,\beta}((\tau,1)) + G_{N,\beta}((\tau,-1)) &= \frac{e^{ H_{N-1, \beta'}(\tau)}\cosh \big(X_{N,\beta}(\tau)+h\big)}{\sum_{\tau' \in \bSigma_{N-1}}e^{ H_{N-1, \beta'}(\tau')}\cosh \big(X_{N,\beta}(\tau')+h\big)},\\
		G_{N-1, \beta'}(\tau)& = \frac{e^{ H_{N-1, \beta'}(\tau)}}{\sum_{\tau' \in \bSigma_{N-1}}e^{ H_{N-1, \beta'}(\tau')}}.
	\end{align*}
Plugging these into the first display establishes the equality in our assertion.
	Finally, since $\cosh\geq 1,$ we can drop the term $\la \cosh(X_{N,\beta}(
 \tau)+h)\ra_{N-1,\beta'}$ and use \eqref{add:eq4} to get 
	\begin{align*}
		\E \bgbs{\Gamma(\rho)}_{N,\beta}&\leq \E\bigl\la \Gamma(\tau)\cosh\left( X_{N,\beta}(\tau)+h\right)\bigr\ra_{N-1,\beta'}\\
		&=\E\bigl\la \Gamma(\tau)\E_{g_{\cdot N}}\bigl[\cosh\left( X_{N,\beta}(\tau)+h\right)\bigr]\bigr\ra_{N-1,\beta'}\\
		&\leq e^{2^{-1}\sum_{p\ge 2}\beta_p^2 p}\cosh(h)\E\bigl\la \Gamma(\tau)\bigr\ra_{N-1,\beta'}.
	\end{align*}
\end{proof}

\begin{lemma}\label{add:lem2} For any $\epsilon$ such that $q_P-\epsilon$ is a point of continuity of $\mu_P,$ we have that
	\begin{align}\label{add:lem2:eq1}
		\limsup_{N\to\infty}\E\bgbs{\mathbbm 1_{\{\left|R(\btau^1, \tau^2)-(q_P-\epsilon)\right| < \frac{3}{N-1}\}}}_{N-1, \beta'}=0.
	\end{align}
	In addition, 
	\begin{align}\label{add:lem2:eq2}
		\lim_{\delta \downarrow 0}\limsup_{N\to \infty}\E \gbs{\mathbbm 1_{\{G_{N-1,\beta'}(A_{\ominus}^{ \tau})\ge \delta\}}}_{N-1,\beta'}=1,\,\,\forall \epsilon>0,
	\end{align}
	and
	\begin{align}\label{add:lem2:eq3}
		\lim_{\delta\downarrow 0}\limsup_{N\to\infty}\e\bigl\la \mathbbm 1_{ \{q_P+\delta\geq R(\tau^1,\tau^2)\geq \min(R(\tau^1,\tau^3),R(\tau^2,\tau^3))-\delta\}}\big\ra_{N-1,\beta'}=1.
	\end{align}
\end{lemma}

\begin{remark}\label{add:remark}
	These assertions hold if $\beta'$ is replaced by $\beta.$ Indeed, under this replacement, \eqref{add:lem2:eq2} was established in \cite[Lemma 3]{CP18}, while  \eqref{add:lem2:eq3} can be concluded from \cite{panchenko2013parisi} and the fact that $R(\tau^1,\tau^2)$ converges to $\mu_P$ weakly, see \cite{panchenko2008differentiability}. The latter fact can also be used to deduce \eqref{add:lem2:eq1}.
\end{remark}
\begin{proof} 
	First of all, we claim that $R(\tau^1, \tau^2)$ converges to $\mu_P$ under $\E G_{N-1,\beta'}^{\otimes 2}$ as $N\to\infty.$ To see this, note that $\beta \mapsto \E F_N(\beta)$ and $\beta\mapsto \E F_N(\beta')$ are convex functions and that there exists a constant $C>0$ independent of $N$ such that $|\E F_N(\beta)-\E F_N(\beta')|\leq C/N$ for all $N\geq 1.$ Hence, $\E F_{N-1}(\beta)$ and $\E F_{N-1}(\beta')$ converge to the same limit, say $F(\beta),$ which can be represented using the Parisi formula. One of the useful consequence of this representation is that $F$ is partially differentiable with respect to any $\beta_p$ and from Griffith's lemma\footnote{Girffith's lemma: Let $f_N$ be a sequence of differentiable convex functions defined on an open interval $I$. Assume that $f_N$ converges $f$ pointwise on $I$. If $f$ is differentiable at some $x\in I,$ then $\lim_{N\to\infty}f_N'(x)=f'(x).$},
	\begin{align*}
		\lim_{N\to\infty}\frac{\beta_p}{2}\bigl(1-\E\bigl\la R(\tau^1,\tau^2)^p\bigr\ra_{N-1,\beta'}\bigr)
		&=\lim_{N\to\infty}\partial_p\E F_{N-1}(\beta')\\
		&=\partial_{\beta_p}\E F(\beta)=\frac{\beta_p}{2}\Bigl(1-\int_0^1q^p\mu_P(dq)\Bigr),
	\end{align*}
	where the last equality holds due to \cite{panchenko2008differentiability}. Consequently, whenever $\beta_p\neq 0,$ we have
	\begin{align*}
		\lim_{N\to\infty}\E\bigl\la R(\tau^1,\tau^2)^p\bigr\ra_{N-1,\beta'}=\int_0^1q^p\mu_P(dq)
	\end{align*}
	and our claim follows since we assume that our model is generic.
	
	Now we turn to the proof of \eqref{add:lem2:eq1}. Let $\epsilon'$ and $\epsilon''$ be any positive numbers such that $\epsilon'<\epsilon<\epsilon''$. From these, as long as $N$ is large enough, we can bound
	\begin{align*}
		&\E\bgbs{\mathbbm 1_{\{\left|R (\tau^1, \tau^2)-(q_P-\epsilon)\right| < \frac{3}{N-1}\}}}_{N-1, \beta'}\\
		&=\E\bgbs{\mathbbm 1_{\{R(\tau^1, \tau^2)<q_P-\epsilon+ \frac{3}{N-1}\}}}_{N-1, \beta'}-\E\bgbs{\mathbbm 1_{\{R(\tau^1, \tau^2)\leq q_P-\epsilon- \frac{3}{N-1}\}}}_{N-1, \beta'}\\
		&\leq\E\bgbs{\mathbbm 1_{\{R(\tau^1, \tau^2)\leq q_P-\epsilon'\}}}_{N-1, \beta'}-\E\bgbs{\mathbbm 1_{\{R(\tau^1, \tau^2)< q_P-\epsilon''\}}}_{N-1, \beta'}.
	\end{align*}
	Consequently, from the weak convergence of $R(\tau^1, \tau^2)$ to $\mu_P$ under $\E G_{N-1,\beta'}^{\otimes 2},$ we see that
	\begin{align*}
		\limsup_{N\to\infty}\E\bgbs{\mathbbm 1_{\{\left|R (\tau^1, \tau^2)-(q_P-\epsilon)\right| < \frac{3}{N-1}\}}}_{N-1, \beta'}\leq \mu_P([0,q_P-\epsilon'])-\mu_P([0,q_P-\epsilon'')).
	\end{align*}
	Hence, from the continuity of $\mu_P$ at $q_P-\epsilon,$ after sending $\epsilon'\uparrow \epsilon$ and $\epsilon''\downarrow \epsilon,$
	\begin{align*}
		\limsup_{N\to\infty}\E\bgbs{\mathbbm 1_{\{\left|R(\tau^1, \tau^2)-(q_P-\epsilon)\right| < \frac{3}{N-1}\}}}_{N-1, \beta'}=0.
	\end{align*}
	As for \eqref{add:lem2:eq2} and \eqref{add:lem2:eq3}, we recall that it is already known in \cite[Lemma 3]{CP18} that 
	\begin{align*}
		\lim_{\delta \downarrow 0}\limsup_{N\to \infty}\E \gbs{\mathbbm 1_{\{G_{N-1,\beta}(A_{\ominus}^{ \tau})\ge \delta\}}}_{N-1,\beta}=1.
	\end{align*}
	Since the overlap $R(\tau^1,\tau^2)$ under $\e G_{N-1,\beta}^{\otimes 2}$ and $\e G_{N-1,\beta'}^{\otimes 2}$ converges to the same distribution $\mu_P$ and both of the measures $\e G_{N-1,\beta}^{\otimes \infty}$ and $\e G_{N-1,\beta'}^{\otimes \infty}$ satisfy the extended Ghirlanda-Guerra identities, the limiting distributions of the Gram matrix $(R(\tau^\ell,\tau^{\ell'}))_{1\leq \ell<\ell'}$ corresponding to $(\tau^\ell)_{\ell\geq 1}$  sampled from either $G_{N-1,\beta}$ or $G_{N-1,\beta'}$ are uniquely determined by the overlap distribution $\mu_P$ and are described by the Ruelle Probability Cascades parametrized by $\mu_P,$ see \cite{Pan13}. With this, the same argument in \cite[Lemma 3]{CP18} yields \eqref{add:lem2:eq2}. For \eqref{add:lem2:eq3}, note that the extended Ghirlanda-Guerra identities also imply that for independent $\tau^1,\tau^2,\tau^3\sim G_{N-1,\beta'},$ their overlaps must be asymptotically ultrametric due to \cite{panchenko2013parisi}. From our claim, we also see that $R(\tau^1,\tau^2)$ is asymptotically less than $q_P+\delta$. These together complete the proof of \eqref{add:lem2:eq3}.
\end{proof}

\begin{lemma}\label{lem:sandwich} For $\varepsilon = 0, \pm 1$ and any $\epsilon>0$ such that $q_P-\epsilon$ is a point of continuity of $\mu_P,$ we have
	\[
	\lim_{N\to\infty}\E  \lrgbs{ \big(\,\gbs{e^{\varepsilon X_{N,\beta}(\btau)}\mathbbm 1_{A_{\oplus}^{\rho}} (\btau)}_{N-1, \beta'} - \gbs{ e^{\varepsilon X_{N,\beta}(\btau)} \mathbbm 1_{A_{\ominus}^{\rho}}(\btau)}_{N-1, \beta'}\big)^2\, }_{N, \beta} = 0.
	\]
\end{lemma}
\begin{proof} 
	The left hand side is equal to
	\begin{align*}
		\E  \lrgbs{ \gbs{e^{\varepsilon X_{N,\beta}(\btau)}\mathbbm 1_{A_{\oplus}^{\rho }\setminus A_{\ominus}^{\rho }} (\btau)}_{N-1, \beta'}^2 }_{N, \beta} 
	\end{align*}
	and by Cauchy-Schwarz inequality, it is no more than
	\begin{align*}
		&\E  \bigl\la\gbs{e^{2\varepsilon X_{N,\beta}(\btau)}}_{N-1,\beta'}\gbs{\mathbbm 1_{A_{\oplus}^{\rho }\setminus A_{\ominus}^{\rho }}  (\btau)}_{N-1, \beta'}\bigr\ra_{N,\beta}\\
		&= \E  \gbs{e^{2\varepsilon X_{N,\beta}(\btau)}}_{N-1, \beta'}\bigl\la\gbs{\mathbbm 1_{A_{\oplus}^{\rho }\setminus A_{\ominus}^{\rho }} (\btau)}_{N-1, \beta'}\bigr\ra_{N, \beta}\\
		&\le \Bigl(\E  \gbs{e^{2\varepsilon X_{N,\beta}(\btau)}}_{N-1, \beta'}^{2} \E \bigl\la\gbs{\mathbbm 1_{A_{\oplus}^{\rho }\setminus A_{\ominus}^{\rho }}  (\btau) }_{N-1, \beta'}\bigr\ra_{N,\beta}^{2} \Bigr)^{1/2}\\
		&\le \Bigl(\E  \bigl[\bigl\langle \E_{g_{\cdot N }}  e^{4\varepsilon X_{N,\beta}(\btau)}\bigr\rangle_{N-1, \beta'}\bigr] \E \bigl\la\gbs{\mathbbm 1_{A_{\oplus}^{\rho }\setminus A_{\ominus}^{\rho }}  (\btau) }_{N-1, \beta'}\bigr\ra_{N,\beta} \Bigr)^{1/2}.
	\end{align*}
	Here, using the independence between $g_{\cdot N}$ and the disorder appearing in the Gibbs expectation $\la \cdot\ra_{N-1,\beta'}$, one can get from \eqref{add:eq4} that
	\begin{align*}
		\E_{g_{\cdot N}} e^{4\varepsilon X_{N,\beta}(\btau)}&=e^{8\varepsilon^{2}\text{Var}(X_{N,\beta}(\tau))}\le e^{8\sum_{p\ge 2} p\beta_p^2 } < \infty.
	\end{align*}
	Note that this upper bound holds for $\varepsilon=0$ and $\pm1$.
	Thus,
	\begin{align*}
		&\ \ \E  \lrgbs{\bgbs{e^{\varepsilon X_{N,\beta}(\btau)}\mathbbm 1_{A_{\oplus}^{\rho }}  (\btau)- e^{\varepsilon X_{N,\beta}(\btau)} \mathbbm 1_{A_{\ominus}^{\rho }}(\btau)}^{2}_{N-1, \beta'}}_{N,\beta}\\
		&\le e^{4\sum_{p\ge 2} p\beta_p^2} \Big\{ \E \bbgbs{\bgbs{\mathbbm 1_{A_{\oplus}^{\rho}\setminus A_{\ominus}^{\rho}} (\btau) }_{N-1, \beta'}}_{N,\beta}\Big\}^{1/2}\\
		&\leq e^{(4+1/4)\sum_{p\ge 2} p\beta_p^2} \cosh(h)^{1/2}\Big\{ \E \lrgbs{\bgbs{\mathbbm 1_{A_{\oplus}^{\tau^2}\setminus A_{\ominus}^{\tau^2}} (\btau^1) }_{N-1, \beta'}}_{N-1,\beta'}\Big\}^{1/2}\\
		&=e^{(4+1/4)\sum_{p\ge 2} p\beta_p^2} \cosh^{1/2}(h)\Big\{ \E \lrgbs{\mathbbm{1}_{\{(\tau^1,\tau^2):\,|R(\tau^1,\tau^2)-(q_P-\epsilon)|\leq 3/(N-1)\}}}_{N-1,\beta'}\Big\}^{1/2},
	\end{align*}
	where the second inequality used Lemma \ref{add:lem1} and $\tau^1$ and $\tau^2$ are sampled from the inner and outer Gibbs measures, respectively. Our proof is then completed by applying \eqref{add:lem2:eq1}.
\end{proof}

With the help of the above three lemmas, we can now replace all the indicators in \eqref{eqn:cavity1} with $\mathbbm{1}_{\Sigma_{N-1}^\rho}(\tau)$:
\begin{prop} \label{thm1} 
	For $\epsilon>0$ such that $q_P-\epsilon$ is a point of continuity of $\mu_P,$
	we have that
	\begin{align}
		\label{eqn:thm1}  
		\lim_{N\to\infty}\E \bbbgbs{\Bigl(\gbs{\sigma_N}_{N, \beta}^{\alpha}-\frac{ \gbs{\sinh(X_{N,\beta}(\btau)+h)}_{N-1,\beta'}^{\rho } }{\gbs{\cosh(X_{N,\beta}(\btau)+h)}_{ N-1,\beta'}^{\rho }}\Bigr)^2}_{N,\beta} =0.
	\end{align}
\end{prop}

\begin{proof}
	Define
	\begin{align*}
		A_{\balpha} &:=\frac{1}{2}\bgbs{e^{X_{N,\beta}(\btau)+h}\mathbbm 1_{\bSigma_{N}^{ \alpha}}\big((\btau, 1)\big) - e^{-X_{N,\beta}(\btau)-h}\mathbbm 1_{\bSigma_{N}^{\alpha}}\big((\btau, -1)\big)\,}_{N-1, \beta'},\\
		B_{\balpha} &:=\frac{1}{2}\bgbs{e^{X_{N,\beta}(\btau)+h}\mathbbm 1_{\bSigma_{N}^{ \alpha}}\big((\btau, 1)\big) + e^{-X_{N,\beta}(\btau)-h}\mathbbm 1_{\bSigma_{N}^{\alpha}}\big((\btau, -1)\big)\,}_{N-1, \beta'},\\
		A_\rho &:= \bgbs{\sinh(X_{N,\beta}(\btau)+h)\mathbbm 1_{\bSigma_{N-1}^{\rho }}(\btau)}_{N-1,\beta'},\\
		B_\rho&:= \bgbs{\cosh(X_{N,\beta}(\btau)+h)\mathbbm 1_{\bSigma_{N-1}^{\rho}}(\btau) }_{N-1,\beta'},
	\end{align*}
	and then we can express $$\gbs{\sigma_N}_{N, \beta}^{\alpha} =
	\frac{A_{\balpha}} {B_{\balpha}}
	\quad\mbox{and}\quad\frac{ \gbs{\sinh(X_{N,\beta}(\btau)+h)}_{ N-1,\beta'}^{\rho } }{\gbs{\cosh(X_{N,\beta}(\btau)+h)}_{ N-1,\beta'}^{\rho }} = \frac{A_{\rho}}{B_{\rho}}.$$  For any $\delta >0$, the expectation in \eqref{eqn:thm1} is no more than
	\begin{align*}
		\notag
		\E\lrgbs{\Bigl(\frac{A_{\balpha}} {B_{\balpha}} - \frac{A_{\rho}}{B_{\rho}}\Bigr)^2}_{N,\beta}
		&\le C \E\lrgbs{\Big(\frac{A_{\balpha}} {B_{\balpha}} -\frac{A_{\balpha}} {B_{\balpha}}\mathbbm 1_{\{G_{N-1,\beta'}(A_{\ominus}^{ \rho })\ge \delta\}}  \Big)^2}_{N,\beta}\\
		\notag
		&\quad + C \E\lrgbs{\Big(\frac{A_{\balpha}} {B_{\balpha}} -\frac{A_{\rho }}{B_{\rho }} \Big)^2 \mathbbm 1_{\{G_{N-1,\beta'}(A_{\ominus}^{ \rho })\ge \delta\}}}_{N,\beta}\\
		&\quad + C \E\lrgbs{\Big(\frac{A_{\rho }}{B_{\rho }} -\frac{A_{\rho }}{B_{\rho }}\mathbbm 1_{\{G_{N-1,\beta'}(A_{\ominus}^{\rho })\ge \delta\}} \Big)^2 }_{N,\beta}.
	\end{align*}
	Using that $|A_{\balpha}/B_{\balpha}|\le 1$,  $|A_{\rho }/B_{\rho }| \le 1$ and combining the first term and the third term, we get
	\begin{align}
		\notag \E&\lrgbs{\Bigl(\frac{A_{\balpha}} {B_{\balpha}} - \frac{A_{\rho }}{B_{\rho }}\Bigr)^2}_{N,\beta}\\
		\label{eqn:temp-ineq}
		\le &\,2C\E\bgbs{\mathbbm 1_{\{G_{N-1,\beta'}(A_{\ominus}^{ \rho })< \delta\}}}_{N,\beta} + C \E\lrgbs{\Bigl(\frac{A_{\balpha}}{B_{\balpha}} -\frac{A _{\rho }}{B _{\rho }} \Bigr)^2 \mathbbm 1_{\{G_{N-1,\beta'}(A_{\ominus}^{ \rho })\ge \delta\}}}_{N,\beta}.
	\end{align}
	Since $G_{N-1,\beta'}(A_{\ominus}^{\rho})$ is independent of $g_{\cdot N}$ for any $\rho\in \Sigma_{N-1},$ we can use \eqref{add:lem2:eq1} to bound
	\begin{align*}
		\E\bgbs{\mathbbm 1_{\{G_{N-1,\beta'}(A_{\ominus}^{ \rho })< \delta\}}}_{N,\beta} 
		&\le e^{\sum_{p\ge 2}\frac{\beta_p^2 p}{2}}\cosh(h)\, \E \gbs{\mathbbm 1_{\{G_{N-1,\beta'}(A_{\ominus}^{ \tau})< \delta\}}}_{N-1, \beta'}\to 0
	\end{align*}
	as $N\to\infty$ and then $\delta\downarrow 0,$ where the last limit used \eqref{add:lem2:eq2}.
	As for the second expectation in \eqref{eqn:temp-ineq}, we note that, on the event $\{G_{N-1,\beta'}(A_{\ominus}^{ \rho })\ge \delta\}$, 
	\begin{align*}
		\left|\frac{A_{\balpha}}{B_{\balpha}} -\frac{A _{\rho }}{B _{\rho }} \right|&\le \left\lvert\frac{A _{\rho }}{B _{\rho }} \right\rvert\frac{|B_{\balpha}-B _{\rho }|}{B_{\balpha}} + \frac{|A _{\rho }-A_{\balpha}|}{B_{\balpha}}\\
		&\le \frac{|B_{\balpha}-B _{\rho }|}{G_{N-1,\beta'}(A_{\ominus}^{\rho })} + \frac{|A _{\rho }-A_{\balpha}|}{G_{N-1,\beta'}(A_{\ominus}^{\rho })}\le \delta^{-1}[|B_{\balpha}-B _{\rho }|+ |A _{\rho }-A_{\balpha}|],
	\end{align*}
	where we used that $|A _{\rho }/B _{\rho }|\le 1$ and that from \eqref{eqn:nestedsets},
	\begin{align*}
		B_{\balpha},B _{\rho }& \ge \bgbs{\cosh\big( X_{N,\beta}(\btau)+h\big)\mathbbm 1_{A_{\ominus}^{\rho }}(\btau)}_{N-1, \beta'}\ge \bgbs{\mathbbm 1_{A_{\ominus}^{\rho }}(\btau)}_{N-1, \beta'}=G_{N-1,\beta'}(A_{\ominus}^{\rho }).
	\end{align*}
	Thus, for any $\delta >0$,
	\begin{align*}
		\E\lrgbs{\Bigl(\frac{A_{\balpha}}{B_{\balpha}} -\frac{A _{\rho }}{B _{\rho }} \Bigr)^2 \mathbbm 1_{\{G_{N-1,\beta'}(A_{\ominus}^{\rho })\ge \delta\}}}_{N,\beta}\le C \delta^{-2}\Bigl(  \E \bgbs{(A_{\balpha}-A _{\rho })^2}_{N,\beta} +  \E \bgbs{(B_{\balpha}-B _{\rho })^2}_{N,\beta}\Bigr).
	\end{align*}
	Here, from Lemma \ref{lem:sandwich} and the inequality in \eqref{eqn:nestedsets}, we have 
	\begin{align*}
		\E \bgbs{(A_{\balpha}-A _{\rho })^2}_{N,\beta} &\le C'\E \lrgbs{\big(\,\bgbs{e^{ X_{N,\beta}(\btau)+h}\mathbbm 1_{A_{\oplus}^{\rho }} (\btau)}_{N-1, \beta'} - \bgbs{ e^{ X_{N,\beta}(\btau)+h} \mathbbm 1_{A_{\ominus}^{\rho }}(\btau)}_{N-1, \beta'}\big)^2}_{N,\beta}\\
		&\quad + C'\E \lrgbs{\big(\,\bgbs{e^{- X_{N,\beta}(\btau)-h}\mathbbm 1_{A_{\oplus}^{\rho }} (\btau)}_{N-1, \beta'} - \bgbs{ e^{- X_{N,\beta}(\btau)-h} \mathbbm 1_{A_{\ominus}^{\rho }}(\btau)}_{N-1, \beta'}\big)^2}_{N,\beta} \\
		&\quad \to 0 \quad \text{as }N \to \infty.
	\end{align*}
	Similarly, $ \E \bgbs{(B_{\balpha}-B _{\rho })^2}_{N,\beta}\to 0$. 
	These imply that the second term in \eqref{eqn:temp-ineq} also vanishes for any $\delta>0$, completing our proof.
\end{proof}
\subsection{Central limit theorem for the cavity field \label{sec:cavity2}}
The goal of this section is to derive a central limit theorem for the centered cavity field defined as
\begin{align}	
	\nonumber \dot X_{N,\beta}^{\rho}( \tau) &:= X_{N,\beta}(\tau)-\bigl\la X_{N,\beta}(\tau)\bigr\ra_{N-1,\beta'}^{\rho}\\
	\label{eqn:defYnbeta}&=\sum_{p\ge 2}\frac{\beta_p\sqrt{p!}}{N^{\frac{p-1}{2}}} \sum_{1\le i_1< \cdots<i_{p-1}\le N-1} g_{i_1\ldots i_{p-1},N}\bigl(\tau_{i_1}\cdots \tau_{i_{p-1}}-\gbs{\tau_{i_1}\cdots\tau_{i_{p-1}}}_{N-1, \beta'}^{\rho}\bigr)
\end{align}
for any $\rho,\tau\in \Sigma_{N-1}.$

\begin{theorem}[Central Limit Theorem] 
	\label{lemma4-CLT}
	Assume that $U$ is an infinitely differentiable function on $\mathbb R$ satisfying
	\[
	\sup_{0\le \gamma \le M }\E|U^{(d)}(\gamma z )|^k  < \infty,\ \  \forall d, k\in \mathbb Z,\  0\le d, k, M < \infty.  
	\] 
	Then, for any integer $r\ge 1$, we have
	\begin{align}
		\label{eqn:lemma4}
		\lim_{\epsilon \downarrow 0}
		\limsup_{N\to\infty}\E \lrgbs{ \Big[\bgbs{U\big(
				\dot X_{N,\beta}^{\rho}( \tau)
				\big)}^{\rho}_{N-1, \beta'} - \E U\big(\xi \sqrt{\zeta'(1)-\zeta'(q_P)} \big)\Big]^{2r}}_{N,\beta}=0,
	\end{align}
	where $\xi\sim N(0,1)$ is independent of all other random variables.
\end{theorem}

\begin{remark}
	As we will see later, this theorem is special case of the multivariate central limit theorem stated in Theorem \ref{thm:mvclt} below. Here we provide a standalone proof for Theorem \ref{lemma4-CLT} to illustrate the main steps as a warm-up for the proof of Theorem \ref{thm:mvclt}. 
\end{remark}

	 The proof of Theorem \ref{lemma4-CLT} relies on the concentration of the overlap within the pure state established in the following lemma.

\begin{lemma}[Concentration of the overlap]\label{concentration}
	We have that
	\begin{align*}
		\lim_{\epsilon\downarrow 0}\limsup_{N\to\infty}\e \bigl\la \bigl\la|R(\tau^1,\tau^2)-q_P|\bigr\ra_{N-1,\beta'}^\rho\bigr\ra_{N-1,\beta'}&=0,\\
		\lim_{\epsilon\downarrow 0}\limsup_{N\to\infty}\e \bigl\la \bigl\la|R(\tau,\rho)-q_P|\bigr\ra_{N-1,\beta'}^\rho\bigr\ra_{N-1,\beta'}&=0,	
	\end{align*}
where in both equations, the $\rho$'s are sampled from the outer Gibbs measure and $\tau,\tau^1,\tau^2$ are sampled from the inner one.
\end{lemma}

\begin{proof}
	For any $\delta > 0$, letting $\mathcal{E}_{\delta,\rho }= \{G_{N-1,\beta'}(\Sigma_{N-1}^{\rho }) \ge \delta\}$, we have
	\begin{align}
		\notag &\E \bgbs{ \bgbs{\bigl|
				R(\tau^1,\tau^2)-q_P
				|}_{N-1, \beta'}^{\rho } }_{N-1, \beta'}\\
		\notag
		&=\E \bgbs{\bgbs{\big|
				R(\tau^1,\tau^2)-q_P
				\big|}_{N-1, \beta'}^{\rho} \mathbbm 1_{\mathcal{E}_{\delta, \rho }^c} }_{N-1, \beta'}  \\
		\notag	&
		+  \E \lrgbs{\frac{\big|R(\tau^1,\tau^2)-q_P\big| \mathbbm 1_{\Sigma_{N-1}^{\rho }}(\tau^1)\mathbbm 1_{\Sigma_{N-1}^{\rho }}(\tau^2)}{G^2_{N-1,\beta'}(\Sigma_{N-1}^{\rho })}\mathbbm 1_{\mathcal{E}_{\delta, \rho }}}_{N-1, \beta'}\\
		\label{eqn:temp2} &\le 2 \E \bgbs{\mathbbm 1_{\mathcal{E}_{\delta,\rho }^c} }_{N-1, \beta'}
		+\frac{1}{\delta^2} \E \bgbs{\big| 
			R(\tau^1,\tau^2)-q_P\big|\mathbbm 1_{\Sigma_{N-1}^{\rho }}(\tau^1)\mathbbm 1_{\Sigma_{N-1}^{\rho }}(\tau^2)}_{N-1, \beta'},
	\end{align}
	where in the second expectations on the right-hand side, $\rho,\tau^1,\tau^2$ are understood as i.i.d. copies from $G_{N-1,\beta'}.$
	From \eqref{add:lem2:eq2}, the first term in \eqref{eqn:temp2} vanishes as $N\to\infty$ and then $\delta\downarrow 0.$ Next, since both $R(\rho,\tau^1)$ and $R(\rho,\tau^2)$ are at least $q_P-\epsilon$, it follows from \eqref{add:lem2:eq3} that $|R(\tau^1,\tau^2)-q_P|\leq 2\epsilon$ with probability nearly one under $\e G_{N-1,\beta'}^{\otimes 2}.$ Hence, the second term in \eqref{eqn:temp2} also vanishes as $\epsilon\downarrow 0$, and this completes the proof of the first assertion. As for the second one, we have
	\begin{align*}
		&\E \bgbs{ \bgbs{\bigl|
				R(\rho,\tau)-q_P
				|}_{N-1, \beta'}^{\rho } }_{N-1, \beta'}\le 2 \E \bgbs{\mathbbm 1_{\mathcal{E}_{\delta,\rho }^c} }_{N-1, \beta'}
		+\frac{1}{\delta} \E \bgbs{\big| 
			R(\rho,\tau)-q_P\big|\mathbbm 1_{\Sigma_{N-1}^{\rho }}(\tau)}_{N-1, \beta'}.
	\end{align*}
	Again the first term vanishes in the limit due to \eqref{add:lem2:eq2}.
	From \eqref{add:lem2:eq3} and the constraint $\Sigma_{N-1}^{\rho}$, we see that $|R(\rho,\tau)-q_P|\leq 2\epsilon$ with probability nearly one under $\e G_{N-1,\beta'}^{\otimes 2}.$  Hence the second term also vanishes, completing our proof.
\end{proof}

\begin{remark}
	\label{add:remark2}\rm Following the same proof, it can also be shown that $$\lim_{\epsilon\downarrow 0}\limsup_{N\to\infty}\e\bigl\la\bigl|\|\la \sigma\ra_{N,\beta}^\alpha\|^2-q_P\bigr|^2 \bigr\ra_{N,\beta}=\lim_{\epsilon\downarrow 0}\limsup_{N\to\infty}\e\bigl\la\bigl|\bigl\la R(\sigma^1,\sigma^2)\bigr\ra_{N,\beta}^\alpha-q_P\bigr|^2 \bigr\ra_{N,\beta}=0.$$
\end{remark}
\begin{proof}[\bf Proof of Theorem \ref{lemma4-CLT}] 
	We prove the central limit theorem via an interpolation argument.
	Fix $r\ge 1$. For every integer $1\leq m\leq 2r$ and any $\rho,\tau^m\in\Sigma_{N-1}$, consider the interpolation
	\begin{align*}
		x_{\rho}^m(t) &:= \sqrt{t} S_{\rho}^m + \sqrt{1-t}\sqrt{\zeta'(1)-\zeta'(q_P)}\xi^m,\,\,0\leq t\leq 1,
	\end{align*}
	where the $\xi^m$'s are i.i.d copies of $\xi$ that are also independent of everything else, and $S_{\rho}^m:=  \dot X_{N,\beta}^{\rho}(\tau^m).$ Set
	\begin{align*}
		V(x) &:= U(x) -\E U(\xi\sqrt{\zeta'(1)-\zeta'(q_P)}),\,\,x\in \mathbb{R}.
	\end{align*}
	Recall that for any $\alpha\in\Sigma_N,$ $\rho\in\Sigma_{N-1}$ represents the first $N-1$ spins. For $0\leq t\leq 1,$ define
	\begin{align*}
		\phi_N(t) &= \E \lrgbs{\bbbgbs{\prod_{m=1}^{2r}V(x_{\rho}^m(t))}_{N-1, \beta'}^{\rho}}_{N,\beta}.
	\end{align*}
	Evidently the left-hand side of \eqref{eqn:lemma4} equals $\phi_N(1)$. 
	
	In order to control $\phi_N(1)$, we introduce
	\[
	\psi_N(t):= \E \lrgbs{\cosh(X_{N,\beta}(\rho)+h)\bbbgbs{\prod_{m=1}^{2r}V(x_{\rho}^m(t))}_{N-1, \beta'}^{\rho}}_{N-1,\beta'}.
	\]
where $\rho$ and $(\tau^1,\ldots,\tau^{2r})$ are sampled from the outer and inner Gibbs measures, $\la \cdot\ra_{N-1,\beta'}$ and $\la \cdot\ra_{N-1,\beta'}^\rho$, respectively.
	Observe that $\psi_N(0)=0$. On the other hand, note that
when $t=1,$ $x_\rho^m(1)=\dot X_{N,\beta}^\rho(\tau^m)$ and since $\tau^1,\ldots, \tau^{2r}$ are i.i.d. sampled from $\la \cdot\ra_{N-1,\beta'}^\rho$, it follows that $$\la V(x_\rho^1(1))\ra_{N-1,\beta'}^\rho=\la V(x_\rho^2(1))\ra_{N-1,\beta'}^\rho=\cdots=\la V(x_\rho^{2r}(1))\ra_{N-1,\beta'}^\rho.$$ Consequently,
$$
\Bigl\la\prod_{m=1}^{2r}V(x_\rho^m(1))\Bigr\ra_{N-1,\beta'}^\rho=\Bigl(\la V(x_\rho^1(1))\ra_{N-1,\beta'}^\rho\Bigr)^{2r},
$$
which is a nonnegative.
From this, we can adopt  a change of measure argument as the one in the proof of Lemma \ref{add:lem1} to bound
\begin{align}
\nonumber     \phi_N(1)&=\e\lrgbs{\bbbgbs{ \prod_{m=1}^{2r}V(x_\rho^m(1))}_{N-1,\beta'}^\rho}_{N,\beta}\\ \nonumber&=\e\lrgbs{\left(\bbbgbs{V(x_\rho(1))}_{N-1,\beta'}^\rho\right)^{2r}}_{N,\beta}\\
\nonumber     &\leq \e\lrgbs{\cosh(X_{N,\beta}(\rho)+h)\left(\bbbgbs{V(x_\rho(1))}_{N-1,\beta'}^\rho\right)^{2r}}_{N-1,\beta'}\\
 \label{add:eq---2}    &=\e\lrgbs{\cosh(X_{N,\beta}(\rho)+h)\bbbgbs{\prod_{m=1}^{2r}V(\dot{X}_{N,\beta}^\rho(\tau^m))}_{N-1,\beta'}^\rho}_{N-1,\beta'}=\psi_N(1).
 \end{align}
Therefore, we have
	\begin{align*}
		\phi_N(1)\leq \psi_N(1) = \int_0^1 \psi_N'(t) dt\le \int_0^1 |\psi_N'(t)| dt,
	\end{align*}
	and it suffices to estimate $|\psi'_N(t)|$. 
	To this end, we differentiate $\psi_N$ and use symmetry in $m$ to get
	\begin{align*}
		\psi_N'(t)&=r\e\bgbs{\bgbs{\Phi_{N,t}(\rho,\tau^1,\ldots,\tau^{2r})}_{N-1,\beta'}^\rho}_{N-1,\beta'}\\
		&=r\e\bgbs{\bgbs{\e_{g_{\cdot N,\xi}}\Phi_{N,t}(\rho,\tau^1,\ldots,\tau^{2r})}_{N-1,\beta'}^\rho}_{N-1,\beta'},	
	\end{align*}
where
\begin{align*}
	\Phi_{N,t}(\rho,\tau^1,\ldots,\tau^{2r})&:=\cosh(X_{N,\beta}(\rho)+h)V'(x_\rho^1(t))\bigg(\frac{S_\rho^1}{\sqrt{t}} -\frac{\sqrt{\zeta'(1)-\zeta'(q_P)}}{\sqrt{1-t}}\xi^1\bigg)\prod_{\substack{s=2}}^{2r} V(x_\rho^{s}(t))
\end{align*}
and $\e_{g_{\cdot N},\xi}$ is the expectation with respect to $g_{\cdot N}$ and $\xi^1.$
	Using Gaussian integration by parts with respect to $g_{\cdot N}$ and $\xi^1$ yields that
	\begin{align*}
		&\e_{g_{\cdot N,\xi}}\Phi_{N,t}(\rho,\tau^1,\ldots,\tau^{2r})\\
		&=\E_{g_{\cdot N},\xi}\Big[\!\cosh(X_{N,\beta}(\rho )+h)V''(x_\rho^1(t))\prod_{m=2}^{2r}V(x_\rho^m(t))\Big]\Big(\E_{g_{\cdot N}} \big[S_\rho^1 S_\rho^1\big]-(\zeta'(1)-\zeta'(q_P))\Big)\\
		&\qquad \ +(2r-1)\E_{g_{\cdot N},\xi}\Big[\!\cosh(X_{N,\beta}(\rho )+h) V'(x_\rho^1(t))V'(x_\rho^2(t))\prod_{m=3}^{2r}V(x_\rho^m(t))\Big]\E_{g_{\cdot N}} \big[S_\rho^1S_\rho^2\big]\\
		& \qquad \ + \frac{1}{\sqrt{t}}\E_{g_{\cdot N}, \xi}\Big[\!\sinh(X_{N,\beta}(\rho )+h) V'(x_\rho^1(t))\prod_{m=2}^{2r}V(x_\rho^m(t))\Big]\E_{g_{\cdot N}} \big[ S_\rho^1X_{N,\beta}(\rho )\big],
	\end{align*}
	where the term $2r-1$ arises from the symmetry in $m$.
	Note that $X_{N,\beta}(\rho)+h$ is a mean-$h$ Gaussian random variable, whose variance is uniformly bounded for all $N$ and all $\rho\in \Sigma_{N-1}$, which ensures, for all $k\ge 0$, 
	\begin{align*}
		\max_{\rho\in \Sigma_{N-1}}	\E_{g_{\cdot N}}\big[\sinh^{2k} (X_{N,\beta}(\rho)+h)\big] & < \infty,\\
		\max_{\rho\in\Sigma_{N-1}}	\E_{g_{\cdot N}}\big[\cosh^{2k} (X_{N,\beta}(\rho)+h)\big]& < \infty.
	\end{align*}
	Similarly, $x_\rho^m(t)$ is a mean-zero Gaussian random variable, whose variance is uniformly bounded for all $N$, all $\rho, \tau^m\in \Sigma_{N-1}$, and all $0\le t\le 1$:
	\[
	0\le \text{Var}_{g_{\cdot N,\xi}}(x_\rho^m(t)) = t\text{Var}_{g_{\cdot N}}(S_\rho^m) + (1-t)(\zeta'(1)-\zeta'(q_P))\le 4 \zeta'(1) + (\zeta'(1)-\zeta'(q_P)) <\infty.
	\]
	With the given assumption on $U$, we then have that for any $k, d\ge 0$ and $1\leq m\leq 2r,$
	\begin{align*}
		\sup_{\substack{0\leq t\leq 1,\\\rho,\tau^m\in \Sigma_{N-1}}}\E_{g_{\cdot N},\xi}\big[V^{(d)}(x_\rho^m(t))^{2k}\big]&< \infty. 
	\end{align*}
	Thus, using H\"older's inequality, the derivative $\psi_N'(t)$ is bounded by 
	\begin{align}
		\label{eqn:CLT-3terms:eq1}
		|\psi'_N(t)| &\le C\E \bgbs{\bgbs{\,\big|\E_{g_{\cdot N}}\big[  S_\rho^1  S_\rho^1-(\zeta'(1)-\zeta'(q_P))\big]\big|\, }^{\rho }_{N-1,\beta'}}_{N-1, \beta'} \\
		\label{eqn:CLT-3terms:eq2}&+C\E\bgbs{ \bgbs{\,\big|\E_{g_{\cdot N}}\big[  S_\rho^1  S_\rho^2\big]\big|\, }^{\rho }_{N-1,\beta'}}_{N-1, \beta'} \\
		\label{eqn:CLT-3terms:eq3}
		&+\frac{C}{\sqrt{t}}\E\bgbs{\bgbs{\big|\E_{g_{\cdot N}}\big[  S_\rho^1X_{N, \beta}(\rho )\big]\big|}^{\rho }_{ N-1,\beta'}}_{N-1,\beta'}
	\end{align}
for some constant $C$ independent of $N$ and $t.$
	If we can show that all these expectations vanish as $N\to\infty$ and then $\epsilon \downarrow 0,$ our proof is complete since $1/\sqrt{t}$ is integrable on $(0,1].$ 
	For the remaining of the proof, we handle the three major expectations in the above display as follows.
	
	\smallskip
	
	{\noindent \bf Estimation of \eqref{eqn:CLT-3terms:eq1}:} Write
	\begin{align}
		\notag
		&\ \ \E_{g_{\cdot N}} \big[  S_\rho^1  S_\rho^1-(\zeta'(1)-\zeta'(q_P))\big]\\
		\notag
		&= 
		\sum_{p\ge 2}\frac{\beta_p^2 p!}{N^{p-1}}\sum_{1\le i_1< \cdots < i_{p-1}\le N-1}\bigl(\tau_{i_1}^1\cdots \tau_{i_{p-1}}^1-\gbs{\tau_{i_1}\cdots \tau_{i_{p-1}}}_{N-1,\beta'}^{\rho }\bigr)^2
		- (\zeta'(1)-\zeta'(q_P))\\
		\notag
		&=\sum_{p\ge 2}\beta_p^2 p\bigg[\prod_{l=1}^{p-1}\Bigl(\frac{N-l}{N}\Bigr)-1\bigg] +\bigg[\sum_{p\ge 2}\frac{\beta^2_p p}{N^{p-1}}\!\sum_{\substack{i_1,\ldots, i_{p-1}=1\\\text{distinct}}}^{N-1}\gbs{\tau^1_{i_1}\cdots \tau^1_{i_{p-1}}\tau^2_{i_1}\cdots\tau^2_{i_{p-1}}}_{N-1,\beta'}^{\rho }-\zeta'(q_P)\bigg]\\
		\label{eqn:CLT-3terms-1}
		&\ \ - 2\bigg[\sum_{p\ge 2}\frac{\beta^2_p p}{N^{p-1}}\!\sum_{\substack{i_1,\ldots, i_{p-1}=1\\\text{distinct}}}^{N-1}\tau_{i_1}^1\cdots \tau_{i_{p-1}}^1\gbs{\tau_{i_1}\cdots\tau_{i_{p-1}}}_{N-1,\beta'}^{\rho }-\zeta'(q_P)\bigg].
	\end{align}
	Note that each sum in \eqref{eqn:CLT-3terms-1} is bounded in absolute value by $C_\beta$, introduced at the beginning of Section \ref{intro}. By the dominated convergence theorem, the first sum in \eqref{eqn:CLT-3terms-1} goes to 0 as $N\to\infty$.
	The second sum in \eqref{eqn:CLT-3terms-1} is equal to
	\begin{align*}
		&\sum_{p\ge 2}\left(\frac{N-1}{N}\right)^{p-1}\beta^2_p p\,\bgbs{R^{p-1}(\tau^1, \tau^2)}_{N-1,\beta'}^{\rho }-\zeta'(q_P) +o_N(1)=\bgbs{\zeta'(R_{12})}_{N-1,\beta'}^{\rho } - \zeta'(q_P)+ o_N(1),
	\end{align*}
	where for any $\tau^1,\tau^2\in \Sigma_{N-1},$ we abbreviate $R_{12} = R(\tau^1, \tau^2)$. For the last sum in \eqref{eqn:CLT-3terms-1}, we 
	take the Gibbs average under $\gbs{\cdot}_{N-1, \beta'}^{\rho }$ and use Jensen's inequality to get
	\begin{align*}
		&\ \ \bbgbs{\Big|
			\sum_{p\ge 2}\frac{\beta^2_p p}{N^{p-1}}\sum_{\substack{i_1,\ldots, i_{p-1}=1\\\text{distinct}}}^{N-1}\tau_{i_1}^1\cdots \tau_{i_{p-1}}^1\gbs{\tau_{i_1}^2\cdots\tau_{i_{p-1}}^2}_{N-1,\beta'}^{\rho }-\zeta'(q_P)
			\Big|}^{\rho }_{N-1, \beta'}\\
		&\le \bbbgbs{\Big|
			\sum_{p\ge 2}\frac{\beta^2_p p}{N^{p-1}}\sum_{\substack{i_1,\ldots, i_{p-1}=1\\\text{distinct}}}^{N-1}\tau_{i_1}^1\cdots \tau_{i_{p-1}}^1\tau_{i_1}^2\cdots\tau_{i_{p-1}}^2-\zeta'(q_P)
			\Big|}_{N-1, \beta'}^{\rho }\\
		&=\bgbs{|\,
			\zeta'(R_{12})-\zeta'(q_P)
			|}_{N-1, \beta'}^{\rho } + o_N(1).
	\end{align*}
	Thus, 
	\begin{align}
		\notag
		&\ \ \E \lrgbs{\bgbs{\,\big|\E_{g_{\cdot N }}\big[  S_\rho^1  S_\rho^1-(\zeta'(1)-\zeta'(q_P))\big]\big|\, }_{N-1,\beta'}^{\rho }}_{N-1, \beta'}\\
		\label{eqn:clt-first-term-upbd}
		&\le 3 \E \lrgbs{ \bgbs{\bigl|
				\zeta'(R_{12})-\zeta'(q_P)
				\bigr|}_{N-1, \beta'}^{\rho } }_{N-1, \beta'} + o_N(1).
	\end{align}
	Therefore, from Lemma \ref{concentration} and the fact that for all $x,y\in [0,1]$,
	\begin{align}
		\label{eqn:continuity-of-zetaprime}
		|\zeta'(x)-\zeta'(y)|\le \zeta''(1)|x-y|,
	\end{align}
	the right hand side of \eqref{eqn:clt-first-term-upbd} also vanishes in the limit $N\to\infty$ and then $\epsilon\downarrow 0.$
	
	\smallskip
	
	{\noindent \bf Estimation of \eqref{eqn:CLT-3terms:eq2}:}  With a similar argument as the above control, we see that
	\begin{align*}
		&\bgbs{\big|\E_{g_{\cdot N }}\big[  S_\rho^1  S_\rho^2\big]\big|}^{\rho }_{N-1, \beta'}
		\le 4\bgbs{|\,
			\zeta'(R_{12})-\zeta'(q_P)
			|}^{\rho }_{N-1, \beta'} + o_N(1)
	\end{align*}
and Lemma \ref{concentration} and \eqref{eqn:continuity-of-zetaprime} ensure that \eqref{eqn:CLT-3terms:eq2} vanishes in the limit.
	
	\smallskip
	
	{\noindent \bf Estimation of \eqref{eqn:CLT-3terms:eq3}:} 
	Note that for any $\rho\in\Sigma_{N-1},$
	\begin{align}
		\notag
		\bigl|\E_{g_{\cdot N }}\big[  S_\rho^1X_{N, \beta}(\rho )\big]\bigr|& = \biggl|\sum_{p\ge 2}\frac{\beta_p^2 p}{N^{p-1}}\sum_{\substack{i_1,\ldots, i_{p-1}=1\\ \text{distinct}}}^{N-1}\bigl(\tau_{i_1}^1\cdots\tau_{i_{p-1}}^1- \gbs{\tau_{i_1}^1\cdots\tau_{i_{p-1}}^1}^{\rho }_{N-1, \beta'}\bigr)\rho _{i_1}\cdots\rho _{i_{p-1}}\bigg|
		\notag
	\end{align}
	Again, we can use Jensen's inequality to obtain
	\begin{align*}
		&\e\bigl\la	\bigl\la\bigl|\E_{g_{\cdot N }}\big[  S_\rho^1X_{N, \beta}(\rho )\big]\bigr| \bigr\ra_{N-1,\beta'}^\rho\big\ra_{N-1,\beta'}
		\leq
		\e\bigl\la\bigl\la \bigl|\zeta'(R(\tau^1,\rho))-\zeta'(R(\tau^1,\rho))\bigr|\bigr\ra_{N-1,\beta'}^{\rho}\big\ra_{N-1,\beta'}+o_N(1).
	\end{align*}
	From Lemma \ref{concentration},
	this bound vanishes in the limit.
\end{proof}

\subsection{Proof of Theorem \ref{thm:cavity}}

We begin by establishing the following proposition.

\begin{prop} \label{thm3} 
	\begin{align*}
		\lim_{\epsilon \downarrow 0}\limsup_{N\to\infty}\E \lrgbs{\big[\gbs{\sigma_N}_{N, \beta}^{\alpha}-\tanh \big( \gbs{X_{N,\beta}(\tau)}_{N-1, \beta'}^{\rho}+h\big)  \big]^2}_{N,\beta} = 0,
	\end{align*}
	where $\epsilon \downarrow 0$ along a sequence such that $q_P-\epsilon$ is always a point of continuity for $\mu_P$.
\end{prop}

\begin{proof}
	By Proposition \ref{thm1}, it is sufficient to show that
	\begin{align}\label{add:eq---1}
		\lim_{\epsilon \downarrow 0}\limsup_{N\to\infty}\E \bbbgbs{\bigg[\frac{ \gbs{\sinh(X_{N,\beta}(\btau)+h)}_{N-1,\beta'}^{\rho } }{\gbs{\cosh(X_{N,\beta}(\btau)+h)}_{N-1,\beta'}^{\rho }}-\tanh \big( 
			\gbs{X_{N,\beta}(\tau)}_{N-1, \beta'}^{\rho }+h\big)  \bigg]^2}_{N,\beta} = 0.
	\end{align}
	To this end, we claim that  for $\varepsilon\in\{-1,1\},$
	\begin{align*}
		\lim_{\epsilon\downarrow 0}\limsup_{N\to\infty} \E\lrgbs{ \Big[\bgbs{e^{\varepsilon
					X_{N,\beta}(\tau)
			}}^{\rho }_{ N-1, \beta'} -  e^{\varepsilon \la X_{N,\beta}(\tau)\ra^{\rho }_{N-1,\beta'}+ \frac{1}{2}(\zeta'(1)-\zeta'(q_P))
			}\Big]^2}_{N,\beta}=0.
	\end{align*}
 If this is valid, then multiplying $e^{\varepsilon h}$ to this equation leads to
 \begin{align*}
		\lim_{\epsilon\downarrow 0}\limsup_{N\to\infty} \E\lrgbs{ \Big[\bgbs{e^{\varepsilon
					(X_{N,\beta}(\tau)
			+h)}}^{\rho }_{ N-1, \beta'} -  e^{\varepsilon (\la X_{N,\beta}(\tau)\ra^{\rho }_{N-1,\beta'}+h)+ \frac{1}{2}(\zeta'(1)-\zeta'(q_P))
			}\Big]^2}_{N,\beta}=0.
	\end{align*}
Noting that
	\[
	\tanh(x)=\frac{\sinh(x)}{\cosh(x)}=\frac{e^x-e^{-x}}{e^x+e^{-x}},\ \ |\sinh(x)|\le \cosh(x), \ \text{and } \cosh(x)\ge 1,
	\]
 we can use the previous limit to recover \eqref{add:eq---1}; along the way, since $(\zeta'(1)-\zeta'(q_P))/2$ is a constant term, it will be cancelled and does not appear in \eqref{add:eq---1}.

	The proof of our claim proceeds as follows. Note that Theorem \ref{lemma4-CLT} with $U(x)=e^{\varepsilon x}$
	for $\varepsilon \in \{+1, -1\}$ readily implies that
	\begin{align}
		\label{eqn:approx-eterm}
		\lim_{\epsilon \downarrow 0}
		\limsup_{N\to\infty}\E \lrgbs{ \Big[\bgbs{e^{\varepsilon
					\dot X_{N,\beta}^\rho(\tau)
			}}_{N-1, \beta'}^{\rho } - \exp\Big\{\frac{1}{2}\big[\zeta'(1)-\zeta'(q_P)\big]
			\Big\}\Big]^4}_{N,\beta}=0.
	\end{align}
	For any $\rho \in \Sigma_{N-1}$, we can decompose
	$
	X_{N,\beta}(\tau) = \dot X_{N,\beta}^\rho(\tau) + \gbs{X_{N,\beta}(\tau)}^{\rho }_{N-1,\beta'}
	$ to get
	\begin{align*}
		&\ \ \E\lrgbs{ \Big[\bgbs{e^{\varepsilon
					X_{N,\beta}(\tau)
			}}^{\rho }_{ N-1, \beta'} -  \exp\Big\{\varepsilon \bgbs{X_{N,\beta}(\tau)}^{\rho }_{N-1,\beta'}+ \frac{1}{2}\big[\zeta'(1)-\zeta'(q_P)\big]
			\Big\}\Big]^2}_{N,\beta}\\
		&=\E \lrgbs{ e^{2\varepsilon \gbs{X_{N,\beta}(\tau)}^{\rho }_{N-1,\beta'}} \Big[\bgbs{e^{\varepsilon
					\dot X_{N,\beta}( \tau)
			}}^{\rho }_{N-1, \beta'} -  \exp\Big\{ \frac{1}{2}\big[\zeta'(1)-\zeta'(q_P)\big]
			\Big\}\Big]^2}_{N,\beta}\\
		&\le \bigg\{\E \bgbs{ e^{4\varepsilon \gbs{X_{N,\beta}(\tau)}^{\rho }_{ N-1,\beta'}}}_{N,\beta}\E \lrgbs{\Big[\bgbs{e^{\varepsilon
					\dot X_{N,\beta}( \tau)
			}}^{\rho }_{N-1, \beta'} -  \exp\Big\{ \frac{1}{2}\big[\zeta'(1)-\zeta'(q_P)\big]
		\Big\}\Big]^4}_{N,\beta}\bigg\}^{1/2}.
	\end{align*}
	Here the second term vanishes as ensured by \eqref{eqn:approx-eterm}. By a change of measure as in the proof of 
	Lemma \ref{add:lem1}, the first expectation is bounded by
	\begin{align*}
		&\E \bigg[\sum_{\rho \in \Sigma_{N-1}} G_{N-1,\beta'}(\rho )\E_{g_{\cdot N }}\Bigl[\cosh(X_{N,\beta}(\rho)+h) e^{4\varepsilon \gbs{X_{N,\beta}(\tau)}^{\rho }_{N-1,\beta'}}\Bigr]\bigg]\\
		&\le \E \sum_{\rho \in \Sigma_{N-1}} G_{N-1,\beta'}(\rho )\left\{\E_{g_{\cdot N }}[\cosh^2(X_{N,\beta}(\rho)+h)]\,\E_{g_{\cdot N }} [e^{8\varepsilon \gbs{X_{N,\beta}(\tau)}^{\rho }_{ N-1,\beta'}}]\right\}^{1/2} \leq C
	\end{align*}
	for some constant $C>0$ independent of $N$ and $\epsilon.$ This completes the proof of our claim.
\end{proof}

We now turn to the proof of Theorem \ref{thm:cavity}.  From Proposition \ref{thm3} and noting that $|\tanh(x)-\tanh(x')|^{2}\leq 2|x-x'|$, it remains to show that
\begin{align}
	\label{eqn:last-step}
	\lim_{\epsilon\downarrow 0}\limsup_{N\to\infty}\E \lrgbs{\Big|X_{N,\beta}(\gbs{\tau}_{N-1,\beta'}^{\rho }) -  \bgbs{X_{N,\beta}(\tau)}_{N-1, \beta'}^{\rho } \Big|}_{N,\beta}= 0.
\end{align}
As in the proof of Lemma \ref{add:lem1}, we apply a change of measure to deal with the dependence of $\gbs{\cdot}_{N,\beta}$ on the disorders $g_{\cdot N }$,
\begin{align*}
	&\ \ \E \lrgbs{\Big |X_{N,\beta}\bigl(\bgbs {\tau}_{N-1,\beta'}^{\rho }\bigr) -  \bgbs {X_{N,\beta}(\tau)}_{N-1, \beta'}^{\rho } \Big |}_{N,\beta}\\
	&\le \E \lrgbs{\cosh(X_{N,\beta}(\rho )+h)\Big |X_{N,\beta}\bigl(\bgbs {\tau}_{N-1,\beta'}^{\rho }\bigr) -  \bgbs {X_{N,\beta}(\tau)}_{N-1, \beta'}^{\rho } \Big |}_{N-1,\beta'}\\
	&\le \E \bigg[\lrgbs{\Big [\E_{g_{\cdot N  }}\cosh^{2}(X_{N,\beta}(\rho )+h)\Big ]^{1/2}\Big \{\E_{g_{\cdot N  }}\Big [X_{N,\beta}\bigl(\bgbs {\tau}_{N-1,\beta'}^{\rho }\bigr) -  \bgbs {X_{N,\beta}(\tau)}_{N-1, \beta'}^{\rho } \Big ]^{2}\Big \}^{1/2}}_{N-1,\beta'}\bigg].
\end{align*}
Here, uniformly in $\rho\in\Sigma_{N-1},$
\begin{align*}
	\E_{g_{\cdot N  }}\cosh^{2}(X_{N,\beta}(\rho )+h)\leq e^{2C_\beta}\cosh^2 (h)
\end{align*}
and 
\begin{align*}
	&\E_{g_{\cdot N  }}\Big [X_{N,\beta}\bigl(\bgbs {\tau}_{N-1,\beta'}^{\rho }\bigr) -  \bgbs {X_{N,\beta}(\tau)}_{N-1, \beta'}^{\rho } \Big ]^{2}\\
	&\leq \sum_{p\ge 2}\frac{\beta_{p}^{2}p}{N^{p-1}}\sum_{\substack{i_{1},\ldots,i_{p-1}\\\text{distinct}}}\bigg(\bgbs {\tau_{i_{1}}}_{N-1,\beta'}^{\rho }\cdots \bgbs {\tau_{i_{p-1}}}_{N-1,\beta'}^{\rho } -\bgbs {\tau_{i_{1}}\cdots\tau_{i_{p-1}}}_{N-1,\beta'}^{\rho } \bigg)^{2}\\
	&=\zeta'\bigl(\bgbs {R_{1,2}}_{N-1,\beta'}^{\rho }\bigr)+\bgbs {\zeta'(R_{1,2})}_{N-1,\beta'}^{\rho } - 2\sum_{p\ge 2}\beta_{p}^{2}p\bgbs {R_{1,p}R_{2,p}\cdots R_{p-1,p}}_{N-1,\beta'}^{\rho }+o_{N}(1).
\end{align*}
where we write $R_{\ell,\ell'}:=R(\tau^\ell,\tau^{\ell'}).$ From Lemma \ref{concentration} and \eqref{eqn:continuity-of-zetaprime}, we readily see that
\begin{align*}
	\lim_{\epsilon\downarrow 0}	\limsup_{N\to\infty}\e\bigl\la \bigl|\zeta'\bigl(\bgbs {R_{1,2}}_{N-1,\beta'}^{\rho }\bigr)-\zeta'(q_P)\bigr|\bigr\ra_{N-1,\beta'}=0,\\
	\lim_{\epsilon\downarrow 0}	\limsup_{N\to\infty}\e\bigl\la \bigl|\bgbs {\zeta'\bigl(R_{1,2})}_{N-1,\beta'}^{\rho }-\zeta'(q_P)\bigr|\bigr\ra_{N-1,\beta'}=0.
\end{align*}
Finally, note that
\begin{align*}
	\bigl|R_{1,p}R_{2,p}\cdots R_{p-1,p}-q_P^{p-1}\bigr|&\leq \sum_{l=1}^{p-1}|R_{l,p}-q_P|.
\end{align*}
It follows from Lemma \ref{concentration} and the symmetry in $l$ that
\begin{align*}
	&\e\Bigl\la \Bigl|\sum_{p\ge 2}\beta_{p}^{2}p\bigl(\bgbs {R_{1,p}R_{2,p}\cdots R_{p-1,p}}_{N-1,\beta'}^{\rho }-q_P^{p-1}\bigr)\Bigr|\Bigr\ra_{N-1,\beta'}\\
	&\leq \bigg(\sum_{p\ge 2}\beta_{p}^{2}p^2\bigg)\e\bigl\la \bigl\la\bigl|R_{1,2}-q_P\bigr|\bigr\ra_{N-1,\beta'}^\rho\bigr\ra_{N-1,\beta'}\to 0,
\end{align*}
as we send $N\to\infty$ and $\epsilon \downarrow 0$. Combining these together validates \eqref{eqn:last-step} and completes our proof.

\section{\label{sec:cavitytoTAP1}Cavity computation for $X_{N,\beta}(\la \sigma\ra_{N,\beta}^\alpha)$}
To derive the TAP equation from the cavity equation, the next step is to represent  $X_{N,\beta}(\gbs{\tau}_{N-1, \beta'}^{\rho})$ in Theorem \ref{thm:cavity} as the difference between $X_{N,\beta}(\gbs{\sigma}^{\alpha}_{N, \beta})$ and the Onsager correction term $\zeta''(q_P)(1-q_P)\la \sigma\ra_{N,\beta}^\alpha$.
To this end, in this section, we will first perform a cavity computation for $X_{N,\beta}(\gbs{\sigma}^{\alpha}_{N, \beta})$, like we did in Proposition \ref{thm1} for $\la \sigma_N\ra_{N,\beta}^\alpha.$

First of all, write
\begin{align*}
	X_{N,\beta}(\gbs{\sigma}_{N,\beta}^\alpha) = \sum_{p\ge 2}\frac{\beta_p \sqrt{ p!}}{N^{\frac{p-1}{2}}}\sum_{1\le i_1<\cdots< i_{p-1}\le N-1}g_{i_1,\ldots, i_{p-1},N}\gbs{\sigma_{i_1}}_{N,\beta}^{\alpha}\cdots \gbs{\sigma_{i_{p-1}}}_{N,\beta}^\alpha.
\end{align*}
For each local magnetization on the right hand side, we can follow the same procedure in Section \ref{sec:cavity1} and write $\sigma$ as $\sigma = (\tau, \sigma_N)$ to express
\[
\gbs{\sigma_j}_{N,\beta}^\alpha = \frac{\bgbs{\tau_j e^{X_{N,\beta}(\tau)+h}\mathbbm 1_{\Sigma_N^{\alpha}}((\tau,1))+{\tau_j}e^{-X_{N,\beta}(\tau)-h}\mathbbm 1_{\Sigma_N^{\alpha}}((\tau,-1))}_{N-1,\beta'}}{\bgbs{e^{X_{N,\beta}(\tau)+h}\mathbbm 1_{\Sigma_N^{\alpha}}((\tau,1))+e^{-X_{N,\beta}(\tau)-h}\mathbbm 1_{\Sigma_N^{\alpha}}((\tau,-1))}_{N-1,\beta'}}
\]
for $1\leq j\leq N-1.$
Due to the nested structure \eqref{eqn:nestedsets} and Lemma \ref{lem:sandwich}, one is tempted to approximate all indicators above by $\mathbbm 1_{\Sigma_{N-1}^{\rho}}(\tau)$ for $\alpha=(\rho, \alpha_N)$ and write that
\begin{align}\label{add:eq12}
\gbs{\sigma_j}_{N,\beta}^\alpha \approx s^{\rho}_j:= \frac{\bgbs{\tau_j\cosh(X_{N,\beta}(\tau)+h)}_{N-1,\beta'}^{\rho}}{\bgbs{\cosh(X_{N,\beta}(\tau)+h)}_{N-1,\beta'}^{\rho}}.
\end{align}
for each $j = 1,2,\ldots, N-1$. Write $s^{\rho} = (s_1^{\rho}, \ldots, s_{N-1}^{\rho})$. The main goal in this section is to establish this approximation and prove the following proposition.
\begin{prop}\label{thm:gbs-cavity-approx}
For any $\epsilon$ such that $q_P-\epsilon$ is a point of continuity of $\mu_P,$ we have that
	\begin{align}
		\label{eqn:gbs-cavity-approx}
		\lim_{N\to\infty}\E \bgbs{\big[X_{N,\beta}(\gbs{\sigma}_{N,\beta}^\alpha)-X_{N,\beta}(s^{\rho})\big]^2}_{N,\beta} = 0.
	\end{align}
\end{prop}
We will prove \eqref{eqn:gbs-cavity-approx} in two steps. Firstly, we show that $X_{N,\beta}(\gbs{\sigma}_{N,\beta}^{\alpha})$ and $X_{N,\beta}(s^{\rho})$ will not differ too much if we truncate the infinite sum over $p\ge 2$ to a finite sum over $2\le p \le p_0$, provided that $p_0$ is sufficiently large. This will require some moment and truncation controls on $X_{N,\beta}(\gbs{\sigma}_{N,\beta}^{\alpha})$ and $X_{N,\beta}(s^{\rho})$ that we establish in the next subsection.

\subsection{Moment and truncation controls of $X_{N,\beta}(\gbs{\sigma}_{N,\beta}^{\alpha})$ and $X_{N,\beta}(s^{\rho})$}

Write  $X_{N,\beta}(\gbs{\sigma}_{N}^{\alpha}) = \sum_{p\ge 2}A^{\alpha}_{p}$ and $X_{N,\beta}(s^{\rho}) = \sum_{p\ge 2}B^{\rho}_{p}$, where 
\begin{align}
	\label{eqn:defAp}
	A^{\alpha}_{p} &:= \frac{\beta_p \sqrt{ p!}}{N^{\frac{p-1}{2}}}\sum_{1\le i_1<\cdots <i_{p-1}\le N-1}g_{i_1,\ldots, i_{p-1},N}\gbs{\sigma_{i_1}}_{N,\beta}^{\alpha}\cdots \gbs{\sigma_{i_{p-1}}}_{N,\beta}^\alpha,\\
	\label{eqn:defBp}
	B^{\rho}_{p} &:= \frac{\beta_p \sqrt{ p!}}{N^{\frac{p-1}{2}}}\sum_{1\le i_1<\cdots <i_{p-1}\le N-1}g_{i_1,\ldots, i_{p-1},N}s_{i_1}^{\rho}\cdots s_{i_{p-1}}^{\rho}.
\end{align}
The following proposition provides a control on the truncation of the above two random series. 

\begin{prop}
	\label{Coro:tailofXn}
	For any $\delta >0$, there exists $p_0>0$ such that for all $N\ge 1$ and small enough $\epsilon>0,$
	\begin{align*}
		\E \bbbgbs{\Big[ \sum_{p>p_0}A_p^\alpha\Big]^{4}}_{N, \beta} &\le \delta\quad\mbox{and}\quad	\E \bbbgbs{\Big[\sum_{p>p_0}B_p^{\rho}\Big]^{4}}_{N, \beta} \le \delta.
	\end{align*}
\end{prop}

Recall the definitions of $A_{\ominus}^{\rho}$ and $G_{N-1,\beta'}(A_{\ominus}^{\rho})$ from Section \ref{sec:cavity1}. As we have seen from Lemma \ref{add:lem2} and Remark \ref{add:remark}, the Gibbs probability of the event $G_{N-1,\beta'}(A_{\ominus}^{\rho})<\delta$ is essentially negligible, the next proposition further shows that the second moments of $X_{N,\beta}(\gbs{\sigma}_{N,\beta}^{\alpha})$ and $X_{N,\beta}(s^{\rho})$ on this event remain negligible.

\begin{prop} \label{Coro:Xn-Negligible}For any $\epsilon>0,$ we have that
	\begin{align}
		\label{eqn:approxXn-sigma-smallW}
		\lim_{\delta \downarrow 0}\limsup_{N\to\infty}\E \bgbs{\big[X_{N,\beta}(\gbs{\sigma}_{N,\beta}^{\alpha})\mathbbm 1_{\{G_{N-1,\beta'}(A_{\ominus}^{\rho})<\delta\}}\big]^2}_{N, \beta} &=0, \\
		\label{eqn:approxXn-s-smallW}
		\lim_{\delta \downarrow 0}\limsup_{N\to\infty}\E \bgbs{\big[X_{N,\beta}(s^{\rho})\mathbbm 1_{\{G_{N-1,\beta'}(A_{\ominus}^{\rho})<\delta\}}\big]^2}_{N, \beta} &=0. 
	\end{align}
\end{prop}

The proofs of  Propositions \ref{Coro:tailofXn} and \ref{Coro:Xn-Negligible} are deferred to Appendix \ref{app1}.

\subsection{Removal of the cavity spin constraint II}

Fix $2\le p\le p_0$ and recall the definitions of $A_p^\alpha$ and $B_p^{\rho}$ in \eqref{eqn:defAp} and \eqref{eqn:defBp}. Using replicas, we may rewrite $A_p^\alpha$ and $B_p^{\rho}$ as follows.
For $\boldsymbol{\tau}=(\tau^1,\ldots,\tau^{p-1})\in \Sigma_{N-1}^{p-1}$, define
\begin{align}
	\label{eqn:defnZNp}
	Z_{N,p}(\boldsymbol{\tau})= \frac{\sqrt{p!}}{N^{(p-1)/2}}\sum_{1\le i_1<\cdots <i_{p-1}\le N-1}g_{i_1,\ldots, i_{p-1},N}\tau^{1}_{i_1}\cdots \tau^{p-1}_{i_{p-1}}.
\end{align}
Define
\begin{align*}
	C_p^\alpha &= \bbbgbs{Z_{N,p}(\boldsymbol{\tau})\prod_{l=1}^{p-1}\frac{ e^{X_{N,\beta}(\tau^l)+h}\mathbbm 1_{\Sigma_N^{\alpha}}((\tau^l,1))+e^{-X_{N,\beta}(\tau^l)-h}\mathbbm 1_{\Sigma_N^{\alpha}}((\tau^l,-1))}{2}}_{N-1,\beta'},\\
	 C_p^{\rho}&=\bbbgbs{Z_{N,p}(\boldsymbol{\tau})\prod_{l=1}^{p-1} \cosh(X_{N,\beta}(\tau^l)+h)\mathbbm 1_{\Sigma_{N-1}^{\rho}}(\tau^l)}_{N-1,\beta'},\\
	D_p^{\alpha}&=\bbbgbs{\prod_{l=1}^{p-1}\frac{ e^{X_{N,\beta}(\tau^l)+h}\mathbbm 1_{\Sigma_N^{\alpha}}((\tau^l,1))+e^{-X_{N,\beta}(\tau^l)-h}\mathbbm 1_{\Sigma_N^{\alpha}}((\tau^l,-1))}{2}}_{N-1,\beta'},\\
	D_p^{\rho}&=\bbbgbs{\prod_{l=1}^{p-1}\cosh(X_{N,\beta}(\tau^l)+h)\mathbbm 1_{\Sigma_{N-1}^{\rho}}(\tau^l)}_{N-1,\beta'}.
\end{align*}
Then we can write
\begin{align}
	\label{eqn:rewriteApBp}
	A_p^\alpha 
	&=\beta_p\frac{C_p^\alpha}{D_p^\alpha}\quad\mbox{and}\quad
	B_p^{\rho} =\beta_p\frac{C_p^{\rho}}{D_p^{\rho}}.
\end{align}
Denote 
\[
\eta_N= \E  \bgbs{ \bgbs{\mathbbm 1_{A_{\oplus}^{\rho}\setminus A_{\ominus}^{\rho}} (\btau)}_{N-1, \beta'}  }_{N-1, \beta'}.
\]
As we have seen at the end of the proof of Lemma \ref{lem:sandwich}, $\eta_N\to 0$ as $N\to \infty$.
Our next Lemma bounds the errors when  approximating $C_p^\alpha$ with $C^{\rho}_{p}$ and $D_p^\alpha$ with $D_p^{\rho}$, which are the main ingredients when we estimate the difference between $A_p^\alpha$ and $B_p^{\rho}$. The proof of the lemma is deferred to Appendix \ref{app2}.

\begin{lemma} \label{lemma:CpDp-diff}For each $r\ge 1$ and $p\ge 2$, there exist positive constants $K(r,p,\beta,h)$ and $K'(r,p,\beta,h)$
such that for all $N\geq 1,$
	\begin{align*}
		\E \bgbs{(C_p^\alpha-C_p^{\rho})^{2r}}_{N,\beta} &\le \eta_N K(r,p,\beta, h) ,  \\
		\E \bgbs{(D_p^\alpha-D_p^{\rho})^{2r}}_{N,\beta} &\le  \eta_N K'(r,p,\beta, h) .
	\end{align*}
\end{lemma}

\subsection{Proof of Proposition \ref{thm:gbs-cavity-approx}}
First,  write
	\begin{align*}
		&\ \ \ X_{N,\beta}(\gbs{\sigma}_{N,\beta}^\alpha)-X_{N,\beta}(s^{\rho})\\
		&=\Big(X_{N,\beta}(\gbs{\sigma}_{N,\beta}^\alpha)-X_{N,\beta}(s^{\rho})\Big) \mathbbm 1_{\{G_{N-1,\beta'}(A_{\ominus}^{\rho})\ge \delta \}} + \Big(X_{N,\beta}(\gbs{\sigma}_{N,\beta}^\alpha)-X_{N,\beta}(s^{\rho})\Big) \mathbbm 1_{\{G_{N-1,\beta'}(A_{\ominus}^{\rho})< \delta \}}\\
		&=\Big(\sum_{2\le p\le p_0}\big(A_p^{\alpha} - B_p^{\rho}\big)\Big)\mathbbm 1_{\{G_{N-1,\beta'}(A_{\ominus}^{\rho})\ge \delta \}} + \Big(\sum_{p> p_0}\big(A_p^{\alpha} - B_p^{\rho}\big)\Big)\mathbbm 1_{\{G_{N-1,\beta'}(A_{\ominus}^{\rho})\ge  \delta \}} \\
		& \qquad + X_{N,\beta}(\gbs{\sigma}_{N,\beta}^\alpha)\mathbbm 1_{\{G_{N-1,\beta'}(A_{\ominus}^{\rho})< \delta \}} - X_{N,\beta}(s^{\rho})\mathbbm 1_{\{G_{N-1,\beta'}(A_{\ominus}^{\rho})< \delta \}}.
	\end{align*}
	It follows that 
	\begin{align*}
		&\ \E \bgbs{\big(X_{N,\beta}(\gbs{\sigma}_{N,\beta}^\alpha)-X_{N,\beta}(s^{\rho})\big)^2}_{N,\beta} \\
		&\le C \E \bbbgbs{\Big(\sum_{2\le p\le p_0}\big(A_p^{\alpha} - B_p^{\rho}\big)\Big)^2\mathbbm 1_{\{G_{N-1,\beta'}(A_{\ominus}^{\rho})\ge \delta \}}}_{N,\beta} \\
		&\ \ \ + C \E \bbbgbs{\Big(\sum_{p>p_0}A_p^{\alpha} \Big)^2}_{N,\beta} + C \E \bbbgbs{\Big(\sum_{p>p_0}B_p^{\rho} \Big)^2}_{N,\beta} \\
		&\ \ \  + C \E \bgbs{X^2_{N,\beta}(\gbs{\sigma}_{N,\beta}^\alpha)\mathbbm 1_{\{G_{N-1,\beta'}(A_{\ominus}^{\rho})< \delta \}} }_{N,\beta} +  C \E \bgbs{X^2_{N,\beta}(s^{\rho})\mathbbm 1_{\{G_{N-1,\beta'}(A_{\ominus}^{\rho})< \delta \}} }_{N,\beta}.
	\end{align*}
	By Propositions \ref{Coro:tailofXn} and \ref{Coro:Xn-Negligible}, the last four terms can be made arbitrary small by choosing $\delta >0$ small enough and $p_0$ large enough. Thus, it suffices to show that for any $0 < \delta < 1$ and $p_0 \ge 2$, the first term vanishes as $N\to \infty$, i.e., 
	\begin{align}
		\label{eqn:Ap-Bp-main-term}
		\lim_{N\to\infty}\E \bbbgbs{\bigg[\sum_{2\le p\le p_0}\big(A_p^{\alpha} - B_p^{\rho}\big)\bigg]^2\mathbbm 1_{\{G_{N-1,\beta'}(A_{\ominus}^{\rho})\ge \delta \}}}_{N,\beta} =0. 
	\end{align}
	To do this,  notice that from \eqref{eqn:rewriteApBp},
	\begin{align*}
		\Big(\sum_{2\le p\le p_0}\big(A_p^{\alpha} - B_p^{\rho}\big)\Big)^2 &= \bigg(\sum_{2\le p\le p_0}\beta_p\Big(\frac{C_p^\alpha}{D_p^\alpha}- \frac{ C_p^{\rho}}{D_p^{\rho}}\Big)\bigg)^2\\
		& \le \bigg(\sum_{2\le p\le p_0} \beta_p^2\bigg)\bigg( \sum_{2\le p\le p_0}\Big(\frac{C_p^\alpha}{D_p^\alpha}- \frac{ C_p^{\rho}}{ D_p^{\rho}}\Big)^2\bigg)\\
		&\le 2C_\beta \sum_{2\le p\le p_0}\bigg(\Big(\frac{C_p^{\rho}}{D_p^{\rho}}\Big)^2 \Big(\frac{D_p^\alpha- D_p^{\rho}}{D_p^\alpha}\Big)^2 + \Big(\frac{C_p^\alpha- C_p^{\rho}}{D_p^\alpha}\Big)^2\bigg).
	\end{align*}
    Using \eqref{eqn:nestedsets} and the fact $\cosh x \ge 1$, we have
	$
	D_p^\alpha, D_p^{\rho}\ge \delta^p
	$
	on the event $\{G_{N-1,\beta'}(A_{\ominus}^{\rho})\ge \delta \}$ and it follows that
	\begin{align}
		\notag
		&\ \ \ \E \bbbgbs{\bigg[\sum_{2\le p\le p_0}\big(A_p^{\alpha} - B_p^{\rho}\big)\bigg]^2\mathbbm 1_{\{G_{N-1,\beta'}(A_{\ominus}^{\rho})\ge \delta \}}}_{N,\beta} \\
		\notag
		&\le \frac{2C_\beta}{\delta^{2p_0}}\sum_{2\le p\le p_0}\E \lrgbs{\Big[\Big(\frac{C_p^{\rho}}{\delta^{p_0}}\Big)^2 \big(D_p^\alpha- D_p^{\rho}\big)^2 + \big(C_p^\alpha-C_p^{\rho}\big)^2\Big]}_{N,\beta}\\
		\label{eqn:tempboundAp-Bp}
		&\le \frac{2C_\beta}{\delta^{2p_0}}\sum_{2\le p\le p_0}\Bigl(\frac{1}{\delta^{2p_0}}\sqrt{\E \bgbs{\big(C_p^{\rho}\big)^4}_{N,\beta}\E\bgbs{ \big(D_p^\alpha-D_p^{\rho}\big)^4 }_{N,\beta}} + \E\bgbs{\big(C_p^\alpha-C_p^{\rho}\big)^2}_{N,\beta}\Bigr).
	\end{align}
	By Lemma \ref{lemma:CpDp-diff}, we have for all $2\le p \le p_0$, 
	\begin{align*}
		\E\bgbs{ \big(D_p^\rho-D_p^{\rho}\big)^4 }_{N,\beta} &\le K  \eta_N,\\
		\E\bgbs{\big(C_p^\alpha-C_p^{\rho}\big)^2}_{N,\beta} &\le K \sqrt{\eta_N},
	\end{align*}
for all $N\geq 1,$ where $K$ is a constant depending only on $p_0,\beta,h.$
	Also, a slight modification, by dropping the indicators in \eqref{add:eq6}, we can bound 
	$
	\E \gbs{\big(C_p^{\rho}\big)^4}_{N,\beta}\le K'
	$ for $2\le p\le p_0$, where $K'$ is a constant depending on $p,\beta,h.$
	Plugging these into \eqref{eqn:tempboundAp-Bp} and then sending $N\to \infty$ yields \eqref{eqn:Ap-Bp-main-term} and thus completes the proof of \eqref{eqn:gbs-cavity-approx}.

\section{\label{sec:cavitytoTAP2}From cavity  to TAP equations}

We will complete the proof of Theorem \ref{thm:main} in this section. At the current stage, we have seen that $X_{N,\beta}(\la \sigma\ra_{N,\beta}^\alpha)$ asymptotically equals $X_{N,\beta}(s^{\rho}) = \sum_{p\ge 2}B^{\rho}_{p}$ following from Proposition \ref{thm:gbs-cavity-approx}, where $B_p^\rho$ is defined in \eqref{eqn:defBp}. Our next step is to handle this sum, via an argument similar to the proof of Proposition \ref{thm3}, by adapting a multivariate central limit theorem for the Gaussian fields, $X_{N, \beta}(\tau^1),\ldots,X_{N, \beta}(\tau^{p-1})$, and $Z_{N,p}(\boldsymbol{\tau})$ appearing in $C_p^{\rho}$ and $D_p^\rho$, see \eqref{eqn:rewriteApBp}. We establish this limit theorem in the first subsection.

\subsection{A multivariate central limit theorem for the cavity fields}

Let $\boldsymbol\xi=(\xi_0,\ldots, \xi_{p-1})^T$ be a centered Gaussian (column) vector of length $p$ with covariance matrix
\[
\E(\boldsymbol\xi \boldsymbol\xi^T)=
\begin{bmatrix}
p (1-q_P^{p-1}) &  \beta_ppq_P^{p-2}(1-q_P) &\beta_p pq_P^{p-2}(1-q_P) & \cdots &  \beta_p pq_P^{p-2}(1-q_P)\vspace{2mm}\\
	\beta_p pq_P^{p-2}(1-q_P) & \zeta'(1)-\zeta'(q_P) & 0 & \cdots &0\\
	\beta_p pq_P^{p-2}(1-q_P) & 0 & \zeta'(1)-\zeta'(q_P) &\cdots & 0\\
	\vdots & \vdots &\vdots & \ddots & \vdots\\
	\beta_p pq_P^{p-2}(1-q_P) & 0 & 0 &\cdots & \zeta'(1)-\zeta'(q_P)
\end{bmatrix},
\]
that is, $\xi_1,\ldots, \xi_{p-1}$ are i.i.d. $N(0, \zeta'(1)-\zeta'(q_P))$ and $\xi_0 \sim N(0, p(1-q_P^{p-1}))$ that has covariance $\beta_ppq^{p-2}(1-q)$ with each of $\xi_1,\ldots, \xi_{p-1}$. Note that this multivariate Gaussian distribution is well-defined. To see this, consider i.i.d. standard normal random variables $z_0, z_1, \ldots, z_{p-1}$, and $\boldsymbol \xi$ can be constructed via linear combinations of $z_0,z_1,\ldots, z_{p-1}$ as follows:
\begin{align*}
	\xi_i &= \sqrt{\zeta'(1)-\zeta'(q_P)} z_i,\quad i=1,\ldots, p-1,\\
	\xi_0 &= \frac{\beta_ppq_P^{p-2}(1-q_P)}{\sqrt{\zeta'(1)-\zeta'(q_P)}}\sum_{i=1}^{p-1}z_i + z_0 \sqrt{p (1-q_P^{p-1}) - \frac{\beta_p^2 p^2(p-1)q_P^{2p-4}(1-q_P)^2}{\zeta'(1)-\zeta'(q_P)}},
\end{align*}
where the quantity in the last square root is always nonnegative because
\begin{align*}
	 &p (1-q_P^{p-1}) (\zeta'(1)-\zeta'(q_P))- \beta_p^2 p^2(p-1)q_P^{2p-4}(1-q_P)^2\\
	 &\geq p (1-q_P^{p-1})\cdot \beta_p^2p (1-q_P^{p-1})-\beta_p^2 p^2(p-1)q_P^{2p-4}(1-q_P)^2\\
	 &=\beta_p^2p^2\bigl((1-q_P^{p-1})^{2}-(p-1)q_P^{2p-4}(1-q_P)^2\bigr)\\
	 &=\beta_p^2p^2(1-q_P)^2\bigl((1+q_P+\cdots+q_{P}^{p-2})^2-(p-1)q_P^{2p-4}\bigr)\geq 0.
\end{align*}
For $\rho\in \Sigma_{N-1}$, recall from \eqref{eqn:defYnbeta} that $$\dot X_{N,\beta}^\rho(\tau^1),\ldots,\dot X_{N,\beta}^\rho(\tau^{p-1}),\quad\forall \tau^1,\ldots,\tau^{p-1} \in \Sigma_{N-1}$$ 
are defined as
\begin{align}
	\label{eqn:defYn-dot-tau}    	\dot X_{N,\beta}^\rho({\tau}^l)& 
	=\sum_{p\ge 2}\frac{\beta_p \sqrt{p!}}{N^{(p-1)/2}}\sum_{1\le i_1<\cdots < i_{p-1} \le N-1}g_{i_1,\ldots, i_{p-1},N} \bigl( \tau_{i_1}^l\cdots \tau^l_{i_{p-1}} -\bgbs{ \tau_{i_1}^l\cdots \tau^{l}_{i_{p-1}}}_{N-1,\beta'}^{\rho}\bigr).
\end{align}
Recall $Z_{N,p}$ from \eqref{eqn:defnZNp}. We set the centralized $Z_{N,p}(\boldsymbol{\tau})$ as
\begin{align}
	\label{eqn:defZn-dot-tau}
	\dot Z_{N, p}^\rho({\boldsymbol\tau})&=\frac{\sqrt{p!}}{N^{(p-1)/2}}\sum_{1\le i_1<\cdots < i_{p-1} \le N-1}g_{i_1,\ldots, i_{p-1},N} \bigl(\tau_{i_1}^1\cdots \tau_{i_{p-1}}^{p-1} -\bgbs{ \tau_{i_1}^1\cdots \tau^{p-1}_{i_{p-1}}}_{N-1,\beta'}^{\rho}\bigr).
\end{align}
Now we are ready to state the central limit theorem. 
\begin{theorem}[Multivariate Central Limit Theorem]\label{thm:mvclt} Fix $p\ge 2$. Let  $F_{0},\ldots, F_{p-1}$ be  infinitely differentiable functions on $\mathbb R$ such that for $z\sim N(0,1)$, any $M>0$ and integers $m,k\geq 0$,
	\[
	\sup_{0\le \gamma \le M }\max\bigl(\E|F_{0}^{(m)}(\gamma z )|^k, \ldots, \E|F_{p-1}^{(m)}(\gamma z )|^k \bigr) < \infty.
	\]
	Then, for any integer $r\ge 1$ fixed, we have
	\begin{align}
		\begin{split}	\label{eqn:mv-clt}
		\lim_{\epsilon \downarrow 0}
		\limsup_{N\to\infty}\,&
		\E \Big\langle \Big[\bgbs{U\big(\dot Z_{N, p}^\rho(\boldsymbol{\tau}), \dot X_{N,\beta}^\rho( \tau^1),\ldots,  \dot X_{N,\beta}^\rho(\tau^{p-1})
			\big)}^{\rho}_{N-1, \beta'} 
	  - \E U(\boldsymbol \xi)\Big]^{2r}\Big\rangle_{N,\beta}=0,
	  	\end{split}
	\end{align}
	where  $U(x_0, \ldots, x_{p-1}) := F_{0}(x_0)F_{1}(x_1)\cdots F_{p-1}(x_{p-1})$ for $x=(x_0,\ldots, x_{p-1})\in \mathbb{R}^p$.
\end{theorem}
\begin{proof}
	The proof utilizes a similar idea as that of Theorem \ref{lemma4-CLT}. The main difference is that here we need to deal with the $2r$-th moment of a multivariate function depending on $(p-1)$ spin replicas. Thus, we will further consider $2r$ independent copies of these $(p-1)$ spins,  $(\tau^1, \tau^2, \ldots, \tau^{p-1})$, and denote them by 
	\begin{align*}
		\boldsymbol \tau^m &= (\tau^{m, 1}, \ldots, \tau^{m, p-1}), \quad m=1,2,\ldots, 2r,
	\end{align*}
where $(\tau^{m,l})_{1\leq m\leq 2r, 1\leq p\leq p-1}$ are independently sampled from $G_{N-1,\beta'}^{\rho}$. For $1\le m\le 2r$, define
	\begin{align*}
		{\boldsymbol S}_\rho^{m}&=(S^{m,0}_\rho,S^{m,1}_\rho,\ldots, S^{m,p-1}_{\rho})=\bigl(\dot Z_{N,p}^\rho(\boldsymbol\tau^{m}), \dot X_{N,\beta}^\rho({\tau}^{m,1}), \ldots, \dot X_{N,\beta}^\rho( {\tau}^{m,p-1})\bigr).
	\end{align*}
Let $\boldsymbol\xi^m=(\xi_0^m,\xi_1^m,\ldots, \xi_{p-1}^m)$ be i.i.d. copies of $\boldsymbol\xi$ and be independent of everything else.
	For  $0\le t\le 1$, consider the interpolation
	\[
	\phi_N(t):=\E \lrgbs{\bbbgbs{\prod_{m=1}^{2r}V\big(\boldsymbol x_\rho^m(t)\big)}_{N-1, \beta'}^{\rho}}_{N,\beta}
	\]
	where
	\begin{align*}
		\boldsymbol x_\rho^m(t) &= \bigl(x_\rho^{m,0}(t),\ldots,x_\rho^{m,p-1}(t)\bigr):=\sqrt{t} {\boldsymbol S}_\rho^{m} + \sqrt{1-t}\,\xi^m
	\end{align*}
and $V:\mathbb R^p \to \mathbb R$ is defined as 
\begin{align*}
		V(x_0,\ldots, x_{p-1}) = U(x_0,\ldots, x_{p-1})  - \E U(\xi_0,\ldots, \xi_{p-1}).
\end{align*}
Note that the left-hand side of \eqref{eqn:mv-clt} equals $\phi_N(1)$ and $\phi_N(0)=0.$
	Following the same argument as \eqref{add:eq---2}, we have $\phi_N(1)\leq \psi_N(1)$, where
	\begin{align*}
		\psi_N(t) :=& \E \lrgbs{\cosh(X_{N,\beta}(\rho)+h)\bbbgbs{\prod_{m=1}^{2r}V\big(\boldsymbol x_\rho^m(t)\big)}_{\!N-1, \beta'}^{\rho}}_{\!\raisemath{4pt}{\!N-1,\beta'}}.
	\end{align*}
    It suffices to show that $\lim_{\epsilon\downarrow 0}\limsup_{N\to\infty}\psi_N(1)=0.$ To justify this limit, observe that $\psi_N(0)=0$ and we thus aim at showing
    \begin{align}\label{add:eq9}
    \lim_{\epsilon\downarrow 0}\limsup_{N\to\infty}\int_0^1|\psi_N'(t)|dt=0.
    \end{align}
    
    First of all, by chain rule, $\psi_N'(t)$ is equal to
	\begin{align*}
	\sum_{m=1}^{2r}\sum_{l=0}^{p-1}\E \lrgbs{\cosh(X_{N,\beta}(\rho)+h)\bbbgbs{\frac{\partial V}{\partial x_l}(\boldsymbol x_\rho^{m}(t))\frac{1}{2}\Big(\frac{{ S}_\rho^{m,l}}{\sqrt{t}}-\frac{\xi_l^m}{\sqrt{1-t}}\Big)\prod_{\substack{m'=1\\m'\ne m}}^{2r}V\big(\boldsymbol x_\rho^{m'}(t)\big)}_{\!N-1, \beta'}^{\rho}}_{\!\raisemath{13pt}{\!N-1, \beta'}}.
\end{align*}
Note that the Gibbs measures $G_{N-1,\beta'}$ and $G_{N-1, \beta'}^{\rho}$ do not depend on $g_{\cdot N }$, and thus, as before, we can compute this derivative by first taking the expectation (inside the Gibbs averages $\gbs{\cdot}_{N-1,\beta'}$ and $\gbs{\cdot}_{N-1,\beta'}^{\rho}$) only with respect to $g_{\cdot N}$ and $(\xi_l^{m})_{0\le l\leq p-1, 1\le m\le 2r}$ and then using the Gaussian integration by parts.  With these, for fixed $m$ and $l$, the relevant terms in the above display become
\begin{align*}
	&\ \ \E_{g_{\cdot N },\boldsymbol \xi}\bigg[\cosh(X_{N,\beta}(\rho)+h)\frac{\partial V}{\partial x_l}(\boldsymbol x_\rho^m(t)) \frac{1}{2}\Big(\frac{{S}_\rho^{m,l}}{\sqrt{t}}-\frac{\xi_l^m}{\sqrt{1-t}}\Big)\prod_{\substack{m'=1\\m'\ne m}}^{2r}V\big(\boldsymbol x_\rho^{m'}(t)\big)\bigg]= \text{(I) + (II) + (III)},
\end{align*}
where
\begin{align*}
{\rm (I)}	&:= \E_{g_{\cdot N },\boldsymbol \xi}\bigg[\sinh(X_{N,\beta}(\rho)+h)\frac{\partial V}{\partial x_l}(\boldsymbol x_\rho^m(t))\prod_{\substack{m'=1\\m'\ne m}}^{2r}V\big(\boldsymbol x_\rho^m(t)\big)\bigg]\cdot\frac{1}{2\sqrt{t}} \E_{g_{\cdot N }} \big[{S}_\rho^{m,l}X_{N,\beta}(\rho)\big],\\
{\rm (II)}&:=\ \sum_{\substack{l'=0}}^{p-1} \E_{g_{\cdot N },\boldsymbol \xi}\bigg[\cosh(X_{N,\beta}(\rho)+h)\frac{\partial^2 V}{\partial x_l\partial x_{l'}}(\boldsymbol x_\rho^m(t))\prod_{\substack{m'=1\\m'\ne m}}^{2r}V\big(\boldsymbol x_\rho^{m'}(t)\big)\bigg]\\
&\quad\qquad\qquad\cdot\frac{1}{2} \E_{g_{\cdot N },\boldsymbol \xi} \Big[\Bigl(\frac{ S_{\rho}^{m,l}}{\sqrt{t}}-\frac{\xi_l^m}{\sqrt{1-t}}\Bigr)x_{\rho}^{m,l'}(t)\Big],
\end{align*}
and
\begin{align*}
{\rm (III)}	&:=\sum_{\substack{m'=1\\m'\ne m}}^{2r}\sum_{\substack{l'=0}}^{p-1} \E_{g_{\cdot N },\boldsymbol \xi}\bigg[\cosh(X_{N,\beta}(\rho)+h)\frac{\partial V}{\partial x_l}(\boldsymbol x_\rho^m(t))\frac{\partial V}{\partial x_{l'}}(\boldsymbol x_\rho^{m'}(t))\prod_{\substack{m''=1\\m''\ne m,m'}}^{2r}V\big(\boldsymbol x_\rho^{m''}(t)\big)\bigg]\\
	&\qquad\qquad\qquad\qquad \quad \cdot  \frac{1}{2} \E_{g_{\cdot N },\boldsymbol \xi} \Big[\Bigl(\frac{ S_\rho^{m,l}}{\sqrt{t}}-\frac{\xi_l^m}{\sqrt{1-t}}\Bigr)x_\rho^{m',l'}(t)\Big].
\end{align*}
Here $\e_{g_{\cdot N},\boldsymbol{\xi}}$ is the expectation with respect to $g_{\cdot N}$ and $(\xi_l^m)_{0\leq l\leq p-1,1\leq m\leq 2r}.$
To handle these terms, we first note that for any $(\tau^{m,l})_{1\le m\le 2r, 1\le l\le p-1}$ sampled from $G_{N-1,\beta'}$, coordinates of ${\boldsymbol S}_\rho^m$ are mean-zero Gaussian random variables with uniformly bounded variance for any $N\ge 1$, and so are the coordinates of $\boldsymbol x_\rho^m(t)$ for any $0\le t\le 1$. Consequently, from the given assumptions on the functions $F_{0},\ldots, F_{p-1}$ and the H\"older inequality, the first expectations in (I), (II), and (III) are bounded by an absolute constant independent of $N$, $t$, and $(\tau^{m,l})_{1\le m\le 2r, 1\le l\le p-1}$. Furthermore, in (II) and (III), using Gaussian integration by parts implies that for all $0<t<1,$ $1\leq m,m'\leq 2r$, and $0\leq l,l'\leq p-1,$
\begin{align*}
	\E_{g_{\cdot N },\boldsymbol \xi} \Big[\Bigl(\frac{S_\rho^{m,l}}{\sqrt{t}}-\frac{\xi_l^m}{\sqrt{1-t}}\Bigr)x_{\rho}^{m',l'}(t)\Big]
	&=\E_{g_{\cdot N }} \big[S_{\rho}^{m,l}S_{\rho}^{m',l'}\big] - \E\big[\xi_l^m\xi_{l'}^{m'}\big].
\end{align*}
From these, we conclude that there exists a positive constant $K_0$ independent of $N$ and $t$ such that
\begin{align}
	\label{eqn:2exp-1}
	|\psi'_N(t)| &\le \frac{K_0}{\sqrt{t}}\sum_{m=1}^{2r}\sum_{l=0}^{p-1}\E\bbgbs{\bgbs{\big|\E_{g_{\cdot N }}\big[ S_{\rho}^{m,l}X_{N,\beta}(\rho)\big]\big|}^{\rho}_{ N-1,\beta'}}_{N-1,\beta'}\\
	\label{eqn:2exp-2}
	&\ \  + K_0\sum_{m,m'=1}^{2r}\sum_{l,l'=0}^{p-1}\E\bbgbs{\bgbs{\big| \E_{g_{\cdot N }} \big[S_{\rho}^{m,l}S_{\rho}^{m',l'}\big] - \E \big[\xi_l^m\xi_{l'}^{m'}\big] \big|}^{\rho}_{ N-1,\beta'}}_{N-1,\beta'}.
\end{align}
Note that $t^{-1/2}$ is integrable on $(0,1].$ It remains to show that the expectations in \eqref{eqn:2exp-1} and \eqref{eqn:2exp-2} vanish in the limit as $N\to \infty$ and then $\epsilon \to 0$, which implies in \eqref{add:eq9}.

\smallskip

{\noindent \bf Estimation of \eqref{eqn:2exp-1}:} We claim that for any $1\leq m\leq 2r$ and $0\leq l\leq p-1,$
 \begin{align}\label{add:eq8}
	\lim_{\epsilon\downarrow 0}	\limsup_{N\to\infty}\E\bbgbs{\bgbs{\big|\E_{g_{\cdot N }}\big[ S_{\rho}^{m,l}X_{N,\beta}(\rho)\big]\big|}^{\rho}_{ N-1,\beta'}}_{N-1,\beta'}=0.
 \end{align}
If $l\ne 0$, then recalling $S_\rho^{m,l}=\dot X_{N,\beta}( \tau^{m,l})$,  the same  argument as that used in the estimation of \eqref{eqn:CLT-3terms:eq3} in the proof of Theorem \ref{lemma4-CLT} yields \eqref{add:eq8}.
As for $l=0$, from $S_\rho^{m,0}=\dot Z_{N,p}^\rho(\boldsymbol\tau),$ we compute
\begin{align*}
	&\E_{g_{\cdot N }}\big[S_\rho^{m,0} X_{N,\beta}(\rho)\big] \\
	&=\frac{\beta_p p}{N^{p-1}}  \sum_{\substack{i_1,\ldots i_{p-1} = 1\\\text{distinct}}}^{N-1}\bigl( \tau_{i_1}^{m,1}\cdots \tau_{i_{p-1}}^{m,p-1} -\bgbs{ \tau_{i_1}^{m,1}\cdots \tau^{m, p-1}_{i_{p-1}}}_{N-1,\beta'}^{\rho}\bigr)\rho_{i_1}\cdots\rho_{i_{p-1}}\\
	&=\beta_pp \Bigl(R(\tau^{m,1},\rho)\cdots R(\tau^{m,p-1},\rho) -\bgbs{R(\tau^{m,1},\rho)\cdots R(\tau^{m,p-1},\rho)}_{N-1,\beta'}^{\rho}\Bigr) + o_N(1).
\end{align*}
Since $\tau^{m,1},\ldots,\tau^{m,p-1}\sim G_{N-1, \beta'}^{\rho}$, we always have
$R(\tau^{m,1}, \rho),\ldots,R(\tau^{m,p-1},\rho) \ge q_P-\epsilon$. This combining with \eqref{add:lem2:eq3} implies that
\[\limsup_{N\to\infty} \E\bbgbs{\bgbs{\big|\E_{g_{\cdot N }}\big[S_\rho^{m,0}X_{N,\beta}(\rho)\big]\big|}^{\rho}_{ N-1,\beta'}}_{N-1,\beta'} \le \beta_p p \bigl((q_P+\epsilon)^{p-1}-(q_P-\epsilon)^{p-1}\bigr) \le 2\beta_p p(p-1)\epsilon.
\]
Sending $\epsilon\downarrow 0$ yields \eqref{add:eq8}.

\smallskip

{\noindent \bf Estimation of \eqref{eqn:2exp-2}:} We claim that for any $1\leq m,m'\leq 2r$ and $0\leq l,l'\leq p-1,$ 
\begin{align}\label{add:eq7}
	\lim_{\epsilon\downarrow 0}\limsup_{N\to\infty}\E\bgbs{\bgbs{\big| \E_{g_{\cdot N }} \big[S_{\rho}^{m,l}S_{\rho}^{m',l'}\big] - \E \big[\xi_l^m\xi_{l'}^{m'}\big] \big|}^{\rho}_{ N-1,\beta'}}_{N-1,\beta'}=0.
\end{align}
First consider $l,l'\neq 0$. Note that, in this case,
\begin{align*}
S_\rho^{m,l}=\dot X_{N,\beta}^\rho( \tau^{m,l})\quad\mbox{and}\quad S_\rho^{m',l'}=\dot X_{N,\beta}^\rho( \tau^{m',l'}).
\end{align*}
In view of the proof of Theorem \ref{lemma4-CLT}, when $m=m'$ and $l=l'$, the control of $\E_{g_{\cdot N }} \big[S_{\rho}^{m,l}S_{\rho}^{m',l'}\big] $ has been implemented in the {\bf Estimation of \eqref{eqn:CLT-3terms:eq1}} ; if $m\neq m'$, or $l\neq l'$, or both,  this expectation can also be controlled by the {\bf Estimation of \eqref{eqn:CLT-3terms:eq2}}. As a conclusion, we readily see that \eqref{add:eq7} holds under these two cases. It remains to consider the scenarios when at least one of the $l, l'$ is equal to 0. Without loss of generality, we assume $l=0$ and divide our discussion into two cases:

\smallskip

{\noindent \bf Case 1: $m\neq m'.$} Note that $\E\big[\xi_0^m\xi_0^{m'}\big]=0$. 
\begin{itemize}
	\item[(1a)] When $l'=0$,
	\begin{align*}
		\big|\E_{g_{\cdot N }} [S_\rho^{m,0} S_\rho^{m',0}]\big|&= \big\lvert\E_{g_{\cdot N }} [\dot Z_{N,p}^\rho ({\boldsymbol \tau^m})\dot Z_{N,p}^\rho({\boldsymbol \tau^{m'}})]\big\rvert\\
		&\le  p \Big|\prod_{\ell=1}^{p-1}R(\tau^{m,\ell},\tau^{m',\ell})-q_P^{p-1}\Big| + p \Big|\prod_{\ell=1}^{p-1}\bgbs{R(\tau^{m,\ell},\tau^{m',\ell})}_{N-1,\beta'}^{\rho}-q_P^{p-1}\Big|\\
		&\ \ \  +  p \Bigg|\sum_{\substack{i_1,\ldots, i_{p-1}=1\\ \text{distinct}}}^{N-1} \tau^{m,1}_{i_1}\cdots\tau^{m,p-1}_{i_{p-1}}\bgbs{\tau^{m',1}_{i_1}\cdots\tau^{m',p-1}_{i_{p-1}}}_{N-1,\beta'}^{\rho}-q_P^{p-1}\Bigg|\\
		&\ \ \ + p \Bigg|\sum_{\substack{i_1,\ldots, i_{p-1}=1\\ \text{distinct}}}^{N-1} \tau^{m',1}_{i_1}\cdots\tau^{m',p-1}_{k_{p-1}}\bgbs{\tau^{m,1}_{i_1}\cdots\tau^{m,p-1}_{i_{p-1}}}_{N-1,\beta'}^{\rho}-q_P^{p-1}\Bigg| + o_N(1).
	\end{align*}
	Taking Gibbs average with respect to $G_{N-1,\beta'}^{\rho}$ and using Jensen's inequality,  we get
	\begin{align*}
		\bgbs{\big|\E_{g_{\cdot N }} [S_\rho^{m,0} S_\rho^{m',0}]\big|}_{N-1,\beta'}^{\rho}\le 4 p \bbbgbs{\bigg|\prod_{\ell=1}^{p-1}R(\tau^{1,\ell},\tau^{2,\ell})-q_P^{p-1}\bigg|}_{N-1,\beta'}^{\rho} + o_N(1).
	\end{align*}
\item[(1b)] If $l'\ne 0$, then 
	\begin{align*}
		\big|\E_{g_{\cdot N }} [S_\rho^{m,0} S_\rho^{m',l'}]\big|&= \big\lvert\E_{g_{\cdot N }} [\dot Z_{N,p}^\rho({\boldsymbol \tau})\dot X_{N,\beta}^\rho({ \tau}^{m',l'})]\big\rvert\\
		&\le \beta_p p \Big|\prod_{\ell=1}^{p-1}R(\tau^{m,\ell},\tau^{m',l'})-q_P^{p-1}\Big| + \beta_p p \Big|\prod_{\ell=1}^{p-1}\bgbs{R(\tau^{m,\ell},\tau^{m',l'})}_{N-1,\beta'}^{\rho}-q_P^{p-1}\Big|\\
		&\ \ + \beta_p p \Bigg|\sum_{\substack{i_1,\ldots, i_{p-1}=1\\ \text{distinct}}}^{N-1} \tau^{m,1}_{i_1}\cdots\tau^{m,p-1}_{i_{p-1}}\bgbs{\tau^{m',l'}_{i_1}\cdots\tau^{m',l'}_{i_{p-1}}}_{N-1,\beta'}^{\rho}-q_P^{p-1}\Bigg|\\
		&\ \ + \beta_p p \Bigg|\sum_{\substack{i_1,\ldots, i_{p-1}=1\\ \text{distinct}}}^{N-1} \tau^{m',l'}_{i_1}\cdots\tau^{m',l'}_{i_{p-1}}\bgbs{\tau^{m,1}_{i_1}\cdots\tau^{m,p-1}_{i_{p-1}}}_{N-1,\beta'}^{\rho}-q_P^{p-1}\Bigg| + o_N(1)
	\end{align*}
	and thus,
	\begin{align*}
		\bgbs{	\big|\E_{g_{\cdot N }} [S_\rho^{m,0} S_\rho^{m',l'}]\big|}_{N-1,\beta'}^{\rho}\le 4 \beta_p p \bbbgbs{\bigg|\prod_{\ell=1}^{p-1}R(\tau^{1,\ell},\tau^{2,l'})-q_P^{p-1}\bigg|}_{N-1,\beta'}^{\rho} + o_N(1).
	\end{align*}
\end{itemize}

{\noindent \bf Case 2: $m=m'$.}

\begin{itemize}
	\item[(2a)] If $l'=0$, then $\E [\xi_0^{m}\xi_0^{m}] = p(1-q_P^{p-1})$, and 
	\begin{align*}
		\bgbs{\big|\E_{g_{\cdot N }} [S_\rho^{m,0}S_\rho^{m,0}]-\E[\xi_0^{m}\xi_0^{m}] \big| }_{N-1,\beta'}^{\rho}&= \lrgbs{\Big|\E_{g_{\cdot N }} \bigl[\dot Z_{N,p}^\rho({\boldsymbol \tau})^2\bigr]- p(1-q_P^{p-1})\Big|}_{N-1,\beta'}^{\rho}\\
		&\le 3p \bbbgbs{\Big|\prod_{\ell=1}^{p-1}R(\tau^{1,\ell},\tau^{2,\ell})-q_P^{p-1}\Big|}_{N-1,\beta'}^{\rho}+o_N(1).
	\end{align*}
\item[(2b)] If $l'\ne 0$, then $\E[\xi_0^{m}\xi_{l'}^{m}] = \beta_p pq_P^{p-2}(1-q_P)$, and 
	\begin{align*}
		\bgbs{\big|\E_{g_{\cdot N }} [S_\rho^{m,0}S_{\rho}^{m,l'}]-\E[\xi_0^{m}\xi_{l'}^{m}]  \big| }_{N-1,\beta'}^{\rho}&= \lrgbs{\Big|\E_{g_{\cdot N }} \bigl[\dot Z_{N,p}^\rho(\boldsymbol \tau)\dot X_{N,\beta}^\rho(\tau^{m,l'})\bigr]- \beta_p pq_P^{p-2}(1-q_P)\Big|}_{N-1,\beta'}^{\rho}\\
		&\le \beta_p p \bbbgbs{\Big|\prod_{\substack{\ell=1\\\ell\ne l'}}^{p-1}R(\tau^{m,\ell},\tau^{m,l'})-q_P^{p-2}\Big| }_{N-1,\beta'}^{\rho}\\
		&\ \ + 3\beta_p p \bbbgbs{\Big|\prod_{\ell=1}^{p-1}R(\tau^{1,\ell},\tau^{2,l'})-q_P^{p-1}\Big|}_{N-1,\beta'}^{\rho}+o_N(1).
	\end{align*}
\end{itemize}	
	From (1a), (1b), (2a), and (2b), we can use the concentration of the overlap in Lemma \ref{concentration} to conclude that \eqref{add:eq7} holds as long as one of the $l,l'$ equals zero. This completes our proof.
\end{proof}

\subsection{Derivation of the Onsager correction term}

Recall $s^\rho$ from \eqref{add:eq12}. We proceed to show that $X_{N,\beta}(s^{\rho})$ is asymptotically equal to the sum of the cavity field $X_{N,\beta}(\gbs{\tau}_{N-1,\beta'}^{\rho})$ and the Onsager term $\zeta''(q_P)(1-q_{P})\gbs{\sigma_N}_{N,\beta}^{\alpha}$ by using the multivariate central limit theorem, Theorem \ref{thm:mvclt}. This is the first place in our derivation that gives rise to  the Onsager correction term. 

\begin{prop} \label{thm:Xns-to-Xntau}We have
	\begin{align*}
		\lim_{\epsilon\downarrow 0}\limsup_{N\to\infty} \E \lrgbs{\Big[X_{N,\beta}(s^{\rho})-X_{N,\beta}(\gbs{\tau}_{N-1,\beta'}^{\rho}) -\zeta''(q_P)(1-q_{P})\gbs{\sigma_N}_{N,\beta}^{\alpha}\Big]^{2}}_{N,\beta} =0.
	\end{align*}
\end{prop}

We establish the proof of Proposition \ref{thm:Xns-to-Xntau} in this subsection. Recall that $X_{N,\beta}(s^{\rho}) = \sum_{p\ge 2}B_p^{\rho}$, where $B_p^{\rho}$ is defined in \eqref{eqn:defBp}.
Write 
\begin{align*}
	&X_{N,\beta}(\gbs{\tau}_{N-1,\beta'}^{\rho}) +\zeta''(q_P)(1-q_{P}) \gbs{\sigma_{N}}_{N,\beta}^{\alpha}=\sum_{p\ge 2} \Big(E_p^{\rho} +  \beta_p^2 p(p-1)q_P^{p-2}(1-q_{P})\gbs{\sigma_N}_{N,\beta}^{\alpha}\Big),
\end{align*}
where
\begin{align*}
	E_p^\rho:=\frac{\beta_p\sqrt{p!}}{N^{(p-1)/2}}\sum_{1\le i_1<\cdots<i_{p-1}\le N-1}g_{i_1\ldots i_{p-1}N}\gbs{\tau_{i_1}}_{N-1,\beta'}^{\rho}\cdots\gbs{\tau_{i_{p-1}}}_{N-1,\beta'}^{\rho}.
\end{align*}
Thus, to prove Proposition \ref{thm:Xns-to-Xntau}, we will show that each $B_p^{\rho}$ is approximately equal to $E_p^{\rho} +  \beta_p^2 p(p-1)q_P^{p-2}(1-q_{P})\gbs{\sigma_N}_{N,\beta}^{\alpha}$, for each $p\ge 2$ fixed.  To begin with, we firstly show that the infinite sum in $X_{N,\beta}(\gbs{\tau}_{N-1,\beta'}^{\rho})$ can be approximated by a finite sum of $E_p^\rho$, similar to what we have done for  $X_{N,\beta}(\gbs{\sigma}_{N,\beta}^{\alpha})$ in Proposition \ref{Coro:tailofXn}.

\begin{proof}[\bf Proof of Proposition \ref{thm:Xns-to-Xntau}] Recall the expression $B_p^\rho$, $\dot Z_{N,p}^\rho$, and $\dot X_{N,\beta}^\rho$ from\eqref{eqn:rewriteApBp}, \eqref{eqn:defZn-dot-tau}, and \eqref{eqn:defYn-dot-tau}, respectively.  Noticing that
	\begin{align*}
		\beta_pZ_{N,p}(\boldsymbol{\tau}) &= \beta_p\dot Z_{N,p}^\rho( {\boldsymbol \tau}) + E_{p}^{\rho},
	\end{align*}
	we can write
	\begin{align}
		\notag
		B_{p}^{\rho} &=\beta_p\frac{C_{p}^{\rho} }{D_{p}^{\rho}} = \beta_p\frac{\bgbs{Z_{N,p}(\boldsymbol{\tau})\prod_{l=1}^{p-1} \cosh(X_{N,\beta}(\tau^l)+h)}_{N-1,\beta'}^{\rho}}{\bgbs{\prod_{l=1}^{p-1} \cosh(X_{N,\beta}(\tau^l)+h)}_{N-1,\beta'}^{\rho}}\\
		\label{eqn:BpwithEp-1}
		&=E_{p}^{\rho}+\beta_p \frac{\bgbs{\dot Z_{N,p}^\rho({\boldsymbol \tau})\prod_{l=1}^{p-1} \cosh(X_{N,\beta}(\tau^l)+h)}_{N-1,\beta'}^{\rho}}{\bgbs{\prod_{l=1}^{p-1} \cosh(X_{N,\beta}(\tau^l)+h)}_{N-1,\beta'}^{\rho}}.
	\end{align}
	For the numerator in the second term of \eqref{eqn:BpwithEp-1}, we can use the relations
	\begin{align*}
		X_{N,\beta}(\tau^{l}) &= \dot X_{N,\beta}^\rho(\tau^{l}) +  \bgbs{X_{N,\beta}(\tau^{l})}_{N-1,\beta'}^{\rho}\quad\mbox{and}\quad\cosh x = \frac{1}{2}\sum_{\varepsilon=\pm 1} e^{\varepsilon x}
	\end{align*}
	to rewrite it as
	\begin{align*}
		&\ \ \bbgbs{\dot Z_{N,p}^\rho({\boldsymbol \tau})\prod_{l=1}^{p-1}\cosh\bigl(\dot X_{N,\beta}^\rho( \tau^{l}) +  \bgbs{X_{N,\beta}(\tau^{l})}_{N-1,\beta'}^{\rho}+h\bigr)}_{N-1,\beta'}^{\rho}\\
		&=\frac{1}{2^{p-1}}\sum_{\boldsymbol \varepsilon}\exp\Big(\sum_{l=1}^{p-1}\varepsilon_{l}\big(\bgbs{X_{N,\beta}(\tau^{l})}_{N-1,\beta'}^{\rho}+h\big)\Big)\bbgbs{\dot Z_{N,p}^\rho({\boldsymbol \tau})\prod_{l=1}^{p-1}e^{\varepsilon_{l}\dot X_{N,\beta}^\rho( \tau^{l})}}_{N-1,\beta'}^{\rho},
	\end{align*}
	where $\boldsymbol \varepsilon = (\varepsilon_{1},\ldots, \varepsilon_{p-1})\in\Sigma_{p-1}$. Applying Theorem \ref{thm:mvclt} to the last term with the choices of functions $F_{0}(x)=x$ and $F_{l}(x) = e^{\varepsilon_{l} x}$ for $l=1,2,\ldots, p-1$, we get, for each $\boldsymbol\varepsilon \in \Sigma_{p-1}$,
	\[
	\lim_{\epsilon\downarrow 0}\limsup_{N\to\infty}\E \lrgbs{\!\bigg[\!\bbgbs{\dot Z_{N,p}^\rho({\boldsymbol \tau})\prod_{l=1}^{p-1}e^{\varepsilon_{l}\dot X_{N,\beta}^\rho(\tau^{l})}\!}_{N-1,\beta'}^{\rho} \!-\! \beta_{p}pq_{P}^{p-2}(1-q_{P})e^{\frac{p-1}{2}(\zeta'(1)-\zeta'(q_{P}))}\sum_{l=1}^{p-1}\varepsilon_{l}\bigg]^{4}}_{N,\beta} \!= 0,
	\]
	where the second term comes from
	\begin{align*}
		\E \big(F_0(\xi_0)\cdots F_{p-1}(\xi_{p-1}) \Big) &=\E\Big(\xi_{0}\prod_{l=1}^{p-1}e^{\varepsilon_{l}\xi_{l}}\Big) \\
		&= \sum_{l=1}^{p-1}\varepsilon_{l}\E(\xi_{0}\xi_{l}) \E(e^{\varepsilon_{1}\xi_{1}})\cdots \E(e^{\varepsilon_{p-1}\xi_{p-1}})\\
		&=\beta_{p}pq_{P}^{p-2}(1-q_{P})e^{\frac{p-1}{2}(\zeta'(1)-\zeta'(q_{P}))}\sum_{l=1}^{p-1}\varepsilon_{l}.
	\end{align*}
	Denote
	\begin{align*}
		Q_{N,p}^{\rho}&:=\frac{1}{2^{p-1}}\sum_{\boldsymbol \varepsilon \in \Sigma_{p-1}}\exp\Big(\sum_{l=1}^{p-1}\varepsilon_{l}\big(\bgbs{X_{N,\beta}(\tau^{l})}_{N-1,\beta'}^{\rho}+h\big)\Big)\E\Big(\xi_{0}\prod_{l=1}^{p-1}e^{\varepsilon_{l}\xi_{l}}\Big) \\
		&=\beta_{p}p(p-1)q_{P}^{p-2}(1-q_{P})e^{\frac{p-1}{2}(\zeta'(1)-\zeta'(1))}\\
		&\qquad\qquad\qquad\cdot \sinh\bigl(\bgbs{X_{N,\beta}(\tau)}_{N-1,\beta'}^{\rho}+h\bigr)\cosh^{p-2}\bigl(\bgbs{X_{N,\beta}(\tau)}_{N-1,\beta'}^{\rho}+h\bigr).
	\end{align*}
	It then follows from the Cauchy-Schwarz inequality that 
	\begin{align*}
		&\Bigg\{\E \Big\langle \Big[ \bbgbs{\dot Z_{N,p}^\rho({\boldsymbol \tau})\prod_{l=1}^{p-1} \cosh\bigl(X_{N,\beta}(\tau^l)+h\bigr)}_{N-1,\beta'}^{\rho}-Q_{N,p}^{\rho}\Big]^{2}\Big\rangle_{N,\beta}\Bigg\}^{2}\\
		&\le\frac{1}{2^{p-1}}\sum_{\boldsymbol \varepsilon \in \Sigma_{p-1}} \E \bbgbs{\exp\Big(4\sum_{l=1}^{p-1}\varepsilon_{l}\big(\bgbs{X_{N,\beta}(\tau^{l})}_{N-1,\beta'}^{\rho}+h\big)\Big)}_{N,\beta}\\
		&\cdot \frac{1}{2^{p-1}}\sum_{\boldsymbol \varepsilon \in \Sigma_{p-1}}\E \bbbgbs{\Big[\bbgbs{\dot Z_{N,p}({\boldsymbol \tau})\prod_{l=1}^{p-1}e^{\varepsilon_{l}X_{N,\beta}^\rho( \tau^{l})}}_{N-1,\beta'}^{\rho} - \beta_{p}pq_{P}^{p-2}(1-q_{P})e^{\frac{p-1}{2}(\zeta'(1)-\zeta'(q_{P}))}\sum_{l=1}^{p-1}\varepsilon_{l}\Big]^{4}}_{N,\beta},
	\end{align*}
	where on the right-hand side, the last expectation vanishes as $N\to \infty$ and then $\epsilon \downarrow 0$. The first expectation equals
	\begin{align*}
		& \E \bbbgbs{\prod_{l=1}^{p-1}\cosh^{4}\bigl(\bgbs{X_{N,\beta}(\tau^l)}_{N-1,\beta'}^{\rho}+h\bigr)}_{N,\beta} \\
		\le\, & \E  \bbbgbs{\cosh(X_{N,\beta}(\rho)+h)\prod_{l=1}^{p-1}\cosh^{4}\bigl(\bgbs{X_{N,\beta}(\tau^l)}_{N-1,\beta'}^{\rho}+h\bigr)}_{N-1,\beta'},
	\end{align*}
where again we adapted a change of measure for $G_{N,\beta}$ as in Lemma \ref{add:lem1}.
	Applying H\"{o}lder's inequality to decompose the product inside $\gbs{\cdot}_{N-1, \beta'}$ and taking the expectation with respect to $g_{\cdot N }$ first, we see this term stays bounded for all $N\ge 1$ and small enough $\epsilon>0$; this is because $X_{N,\beta}(\tau)$ has a variance uniformly bounded by $C_\beta$ for all $N,\beta$ and $\tau \in \Sigma_{N-1}$.  We thus conclude that 
	\begin{align}
		\label{eqn:B-E-numerator}
		\lim_{\epsilon\downarrow 0}\limsup_{N\to\infty}\E \Big\langle \Big[ \bbgbs{\dot Z_{N,p}^\rho({\boldsymbol \tau})\prod_{l=1}^{p-1} \cosh\bigl(X_{N,\beta}(\tau^l)+h\bigr)}_{N-1,\beta'}^{\rho}-Q_{N,p}^{\rho}\Big]^{2}\Big\rangle_{N,\beta}=0,
	\end{align}
	which takes care of the numerator in the second term of \eqref{eqn:BpwithEp-1}. 
	For the denominator there, applying Theorem \ref{thm:mvclt} to the last term with $F_{0}(x)=1$ and $F_{l}(x) = e^{\varepsilon_{l} x}$ for $l=1,2,\ldots, p-1$ and mimicking the computation for the numerator, we will get
	\begin{align}
		\label{eqn:B-E-denominator}
		\lim_{\epsilon\downarrow 0}\limsup_{N\to\infty}\E \bbbgbs{ \Big[ \bbgbs{\prod_{l=1}^{p-1} \cosh\bigl(X_{N,\beta}(\tau^l)+h\bigr)}_{N-1,\beta'}^{\rho}-\tilde Q_{N,p}^{\rho}\Big]^{2}}_{N,\beta}=0,
 \end{align}
	where 
	\begin{align*}
		\tilde Q_{N,p}^{\rho} &:= \frac{1}{2^{p-1}}\sum_{\boldsymbol \varepsilon \in \Sigma_{p-1}}\exp\Big\{\sum_{l=1}^{p-1}\varepsilon_{l}\big(\bgbs{X_{N,\beta}(\tau^{l})}_{N-1,\beta'}^{\rho}+h\big)\Big\}\E\Big(\prod_{l=1}^{p-1}e^{\varepsilon_{l}\xi_{l}}\Big) \\
		&=e^{\frac{p-1}{2}(\zeta'(1)-\zeta'(q_{P}))} \cosh^{p-1}\bigl(\bgbs{X_{N,\beta}(\tau)}_{N-1,\beta'}^{\rho}+h\bigr).
	\end{align*}
	Furthermore, by Proposition \ref{thm3}, 
	\begin{align}
		\begin{split}	\label{eqn:Qnp-ratio-approx}
		&\lim_{\epsilon\downarrow 0}\limsup_{N\to\infty}\e\bbbgbs{ \Big|\frac{Q_{N,p}^{\rho}}{\tilde Q_{N,p}^{\rho}} -\beta_{p}p(p-1)q_{P}^{p-2}(1-q_{P}) \gbs{\sigma_{N}}_{N,\beta}^{\alpha}\Bigr|^2}_{N,\beta}\\
		&=\beta_{p}p(p-1)q_{P}^{p-2}(1-q_{P})\\
		&\qquad\qquad\cdot \lim_{\epsilon\downarrow 0}\limsup_{N\to\infty}\e\bbgbs{ \big|	\tanh\bigl(\bgbs{X_{N,\beta}(\tau)}_{N-1,\beta'}^{\rho}+h\bigr) - \gbs{\sigma_{N}}_{N,\beta}^{\alpha}\bigr|^2}_{N,\beta}=0.
		\end{split}
	\end{align}
Therefore,
	\begin{align*}
		&\E\bgbs{\big[B_{p}^{\rho}-E_{p}^{\rho}-\beta_{p}^{2}p(p-1)q_{P}^{p-2}(1-q_{P}) \gbs{\sigma_{N}}_{N,\beta}^{\alpha}\big]^{2}}_{N,\beta}\\
		&=\beta_p^2\E\bbbgbs{\Big[\frac{\bgbs{\dot Z_{N,p}({\boldsymbol \tau})\prod_{l=1}^{p-1} \cosh(X_{N,\beta}(\tau^l)+h)}_{N-1,\beta'}^{\rho}}{\bgbs{\prod_{l=1}^{p-1} \cosh(X_{N,\beta}(\tau^l)+h)}_{N-1,\beta'}^{\rho}}-\beta_{p}p(p-1)q_{P}^{p-2}(1-q_{P}) \gbs{\sigma_{N}}_{N,\beta}^{\alpha}\Big]^{2}}_{N,\beta}\\
		&\le C\beta_p^2 \E\bbbgbs{\Big[\frac{\bgbs{\dot Z_{N,p}({\boldsymbol \tau})\prod_{l=1}^{p-1} \cosh(X_{N,\beta}(\tau^l)+h)}_{N-1,\beta'}^{\rho}}{\bgbs{\prod_{l=1}^{p-1} \cosh(X_{N,\beta}(\tau^l)+h)}_{N-1,\beta'}^{\rho}}-\frac{Q_{N,p}^{\rho}}{\tilde Q_{N,p}^{\rho}}\Big]^{2}}_{N,\beta}\\
		&\qquad  + C\beta_p^2\E\bbbgbs{\Big[\frac{Q_{N,p}^{\rho}}{\tilde Q_{N,p}^{\rho}}-\beta_{p}p(p-1)q_{P}^{p-2}(1-q_{P}) \gbs{\sigma_{N}}_{N,\beta}^{\alpha}\Big]^{2}}_{N,\beta}
	\end{align*}
	The second term vanishes as $N\to\infty$ and $\epsilon \downarrow 0$ due to \eqref{eqn:Qnp-ratio-approx}. From \eqref{eqn:B-E-numerator}, \eqref{eqn:B-E-denominator}, and the fact that $\cosh(X_{N,\beta}(\tau)+h)\geq 1$ and $\tilde{Q}_{N,\beta}^\rho\geq e^{(p-1)(\zeta'(1)-\zeta'(q_{P}))/2},$ the first term also vanishes. 
	As a conclusion, this proves for each $p\ge 2$, 
	\begin{align}
		\label{eqn:BpEp-final}
		\lim_{\epsilon \downarrow 0}\limsup_{N\to\infty}\E\bgbs{\big[B_{p}^{\rho}-E_{p}^{\rho}-\beta_{p}^{2}p(p-1)q_{P}^{p-2}(1-q_{P}) \gbs{\sigma_{N}}_{N,\beta}^{\alpha}\big]^{2}}_{N,\beta}=0.
	\end{align}
	Finally, write
	\begin{align*}
		&\ \ X_{N,\beta}(s^{\rho})-X_{N,\beta}\bigl(\gbs{\tau}_{N-1,\beta'}^{\rho}\bigr) -\zeta''(q_P)(1-q_P)\gbs{\sigma_N}_{N,\beta}^{\alpha}\\
		&=\sum_{2\le p\le p_{0}}\Bigl(B_{p}^{\rho} -  E_{p}^{\rho}-\beta_{p}^{2}p(p-1)q_{P}^{p-2}(1-q_{P}) \gbs{\sigma_{N}}_{N,\beta}^{\alpha}\Big)\\
		&\qquad + \sum_{p> p_{0}} B_{p}^{\rho}  -  \sum_{p> p_{0}} \Big( E_{p}^{\rho} +\beta_{p}^{2}p(p-1)q_{P}^{p-2}(1-q_{P}) \gbs{\sigma_{N}}_{N,\beta}^{\alpha}\Big),
	\end{align*}
	where the first sum vanishes as $N\to \infty$ and $\epsilon \downarrow 0$ for any $p_{0}$ due to \eqref{eqn:BpEp-final}, and the last two sums can be made arbitrarily small by choosing $p_{0}$ sufficiently large, due to Proposition \ref{Coro:tailofXn} and Lemma \ref{prop:Xn2bd-2} below (proof deferred to Appendix \ref{app3}), respectively. This completes our proof.
\end{proof}

\begin{lemma}  \label{prop:Xn2bd-2}  For any $\delta >0$, there exists $p_0>0$ such that for all $N\ge 1$ and any $\epsilon>0,$
		\begin{align*}
			\E \bbbgbs{\bigg[\sum_{p> p_0}E_p^{\rho}\bigg]^{4}}_{N, \beta} &\le \delta.
		\end{align*}
\end{lemma}

\subsection{Proof of Theorem \ref{thm:main}}
From Theorem \ref{thm:cavity}, Proposition \ref{thm:Xns-to-Xntau}, and the fact that $|\tanh(x)-\tanh(x')|^2\leq 2|x-x'|$ for all $x,x'\in \mathbb{R},$ we readily have 
	\begin{align*}
		\lim_{\epsilon\downarrow 0}\limsup_{N\to\infty}\e\bigl\la\big|	\gbs{\sigma_N}_{N, \beta}^{\alpha}-\tanh \big( X_{N,\beta}(s^\rho)+h- \zeta''(q_P)(1-q_{P})\gbs{\sigma_N}_{N,\beta}^{\alpha}\big)\big|^2\bigr\ra_{N,\beta}=0,
	\end{align*}
	where $\epsilon \downarrow 0$ along a sequence such that $q_P-\epsilon$ is always a point of continuity for $\mu_P$.
	Finally, from Proposition \ref{thm:gbs-cavity-approx}, \eqref{eqn:TAP:eq1} holds and we are done.

\newpage

\begin{appendices}

\section{Proofs of  Propositions \ref{Coro:tailofXn} and \ref{Coro:Xn-Negligible}}\label{app1}

The proofs of  Propositions \ref{Coro:tailofXn} and \ref{Coro:Xn-Negligible}
are based on the following lemma:
\begin{lemma}\label{lemma:Xn2bd} 
		There exists a constant $K= K(\beta, h)>0$ such that for all $N\ge 1$ and small $\epsilon>0,$
		\begin{align*}
\sum_{p\ge 2}\big(\E \bgbs{|A_{p}^\alpha|^{4}}_{N,\beta}\big)^{\frac{1}{4}}\leq K\quad\mbox{and}\quad \sum_{p\ge 2}\big(\E \bgbs{|B_{p}^{\rho}|^4}_{N,\beta}\big)^{\frac{1}{4}}\leq K.
\end{align*}
\end{lemma}

\begin{proof} For notation simplicity, we suppress the superscript $\alpha$ and write $A_p^{\alpha}$ as $A_p$. We handle the series of $A_p$ first.
Note that for each $p$, we only need to consider the case, $N\ge p$, otherwise $A_p=0$ by the definition \eqref{eqn:defAp}. Write
	\begin{align}
		\notag
		\E \bgbs{A_{p}^{4}}_{N,\beta} &= \E \bbbgbs{\Big(\frac{\beta_{p}^{2}p!}{N^{p-1}}\Big)^2 \Big(\sum_{1\le i_1<\cdots<i_{p-1}\le N-1} g_{i_1,\ldots, i_{p-1},N}\gbs{\sigma_{i_1}}_{N,\beta}^{\alpha}\cdots \gbs{\sigma_{i_{p-1}}}_{N,\beta}^\alpha\Big)^{4}}_{N,\beta}\\
		\label{eqn:Ap4}
		&=\Big(\frac{\beta_{p}^{2}p!}{N^{p-1}}\Big)^2\sum_{\alpha \in \Sigma_{N}}\sum_{\mathbf i_1,\ldots, \mathbf i_{4}}\E \Big[ G_{N,\beta}(\alpha)\prod_{l=1}^{4}g_{\mathbf i_l,N} \prod_{k=1}^{p-1}\gbs{\sigma_{i_{l,k}}}_{N,\beta}^{\alpha}\Big],
	\end{align}
	where  the second sum is over all $\mathbf i_l = (i_{l,1},\cdots, i_{l, p-1})$, for $l=1,2,3,4$ that are $(p-1)$-tuples with strictly increasing coordinates from $\{1,2, \ldots, N-1\}^{p-1}$. We note that there are $\binom{N-1}{p-1}$ choices for each $\mathbf i_l$. Write $g_{i_{l,1},\ldots, i_{l, p-1}, N}$ as $g_{\mathbf i_l, N}$ and let $\Delta_p:=\beta_{p}\sqrt{p!}{N^{-(p-1)/2}}.$ We have
	\begin{align*}
		\Delta_p^2 \binom{N-1}{p-1}&=\frac{\beta_p^2p!}{N^{p-1}}\,\frac{(N-1)(N-2)\cdots (N-(p-1))}{(p-1)!}\\
		&=\beta_p^2p\frac{N-1}{N}\cdot \frac{N-2}{N}\cdots \frac{N-(p-1)}{N}\leq \beta_p^2p.
	\end{align*}
	Also, there exists  a constant $C>0$ independent of $p,N,\mathbf i$ such that for any $0\leq k_1,k_2,k_3,k_4\leq 4$, if we let $d=k_1+k_2+k_3+k_4,$ then
	\begin{align}\label{add:eq10}
		\Bigl|\partial^{k_1}_{g_{\mathbf i_1,N}}\partial^{k_2}_{g_{\mathbf i_2,N}}\partial^{k_3}_{g_{\mathbf i_3,N}}\partial^{k_4}_{g_{\mathbf i_4,N}} G_{N,\beta}(\alpha)\Bigr|&\leq  C\Delta_p^{d}G_{N,\beta}(\alpha)\leq C\Delta_p^d(p-1)^dG_{N,\beta}(\alpha)
	\end{align}
and
\begin{align}\label{add:eq11}
		\Bigl|\partial^{k_1}_{g_{\mathbf i_1,N}}\partial^{k_2}_{g_{\mathbf i_2,N}}\partial^{k_3}_{g_{\mathbf i_3,N}}\partial^{k_4}_{g_{\mathbf i_4,N}} \prod_{l=1}^4\prod_{k=1}^{p-1}\gbs{\sigma_{i_{l,k}}}_{N,\beta}^{\alpha}\Bigr|&\leq C\Delta_p^{d}(p-1)^d.
	\end{align} 
These can be established by an induction argument on $d.$
Now we divide the collection of $(\mathbf i_1,\mathbf i_2,\mathbf i_3,\mathbf i_4)$ into three cases and compute, respectively, an upper bound for the summand in \eqref{eqn:Ap4} under each case. In the following discussion, $C_1,C_1',C_2,C_2',\ldots$ are absolute constants independent of $N$ and $p.$
	\begin{itemize}
		\item {\bf Case I: all 4 tuples are distinct.} Applying Gaussian integration by part and the chain rule, we get 
		\begin{align*}
			&\ \ \ \Big|\E  G_{N,\beta}(\alpha)g_{\mathbf i_1,N}g_{\mathbf i_2,N}g_{\mathbf i_3,N}g_{\mathbf i_4,N} \prod_{l=1}^{4}\prod_{k=1}^{p-1}\gbs{\sigma_{i_{l,k}}}_{N,\beta}^{\alpha}\Big|
			\le C_1\Delta_p^4(p-1)^4\E G_{N,\beta}(\alpha) .
		\end{align*}
	Since the number of choices for $(\mathbf i_1,\mathbf i_2, \mathbf i_3, \mathbf i_4)$ in Case I are no more than $\binom{N-1}{p-1}^4$, the summation in \eqref{eqn:Ap4} for Case I is bounded by
	\begin{align*}
		&\Delta_p^4\sum_{\alpha \in \Sigma_{N}}\sum_{{\tiny \rm Case\,\,I}} \Big|\E  G_{N,\beta}(\alpha)g_{\mathbf i_1,N}g_{\mathbf i_2,N}g_{\mathbf i_3,N}g_{\mathbf i_4,N} \prod_{l=1}^{4}\prod_{k=1}^{p-1}\gbs{\sigma_{i_{l,k}}}_{N,\beta}^{\alpha}\Big|\\
		&\leq \Delta_p^4\cdot \binom{N-1}{p-1}^4\cdot C_1\Delta_p^4(p-1)^4\leq C_1\beta_p^8p^8.
	\end{align*}
		\item {\bf Case II: there are three distinct tuples in $(\mathbf i_1, \mathbf i_2, \mathbf i_3, \mathbf i_4)$.} Without loss of generality, suppose $\mathbf i_1 = \mathbf i_2$ and they are both different from distinct $\mathbf i_3,\mathbf i_4$. In this case, again using Gaussian integration by part twice and the chain rule, each summand in \eqref{eqn:Ap4} is bounded in absolute value by
		\[
		\Bigl|\E G_{N,\beta}(\alpha)g^2_{\mathbf i_1,N}g_{\mathbf i_3,N}g_{\mathbf i_4,N}\prod_{l=1}^{4} \prod_{k=1}^{p-1}\gbs{\sigma_{i_{l,k}}}_{N,\beta}^{\alpha}\Bigr|\leq C_2 \Delta_p^2(p-1)^2\E g_{\mathbf i_1,N}^2 G_{N,\beta}(\alpha).
		\]
		It follows that
		\begin{align*}
			&\Delta_p^4\sum_{\alpha \in \Sigma_{N}}\sum_{{\tiny \rm Case\,\,II}} \Big|\E  G_{N,\beta}(\alpha)g_{\mathbf i_1,N}g_{\mathbf i_2,N}g_{\mathbf i_3,N}g_{\mathbf i_4,N} \prod_{l=1}^{4}\prod_{k=1}^{p-1}\gbs{\sigma_{i_{l,k}}}_{N,\beta}^{\alpha}\Big|\\
			&\leq C_2'\Delta_p^4\cdot \binom{N-1}{p-1}^3\cdot \Delta_p^2(p-1)^2 \cdot \E g_{\mathbf i_1,N}^2\\
			&\leq C_2'\Delta_p^6 \binom{N-1}{ p-1}^3(p-1)^2\leq C_2'\beta_p^6p^5
		\end{align*}
		\item{\bf Case III: there are no more than two distinct tuples.} In this case, we have three possibilities, each bounded in absolute value respectively as follows:
		\begin{align*}
		\Bigl|\E G_{N,\beta}(\alpha)g^2_{\mathbf i_1,N}g^2_{\mathbf i_2,N}\prod_{l=1}^{4} \prod_{k=1}^{p-1}\gbs{\sigma_{i_{l,k}}}_{N,\beta}^{\alpha}\Bigr| &\le \E G_{N,\beta}(\alpha)g^2_{\mathbf i_1,N}g^2_{\mathbf i_2,N} ,\\
	\Bigl|\E G_{N,\beta}(\alpha)g^1_{\mathbf i_1,N}g^3_{\mathbf i_2,N}\prod_{l=1}^{4} \prod_{k=1}^{p-1}\gbs{\sigma_{i_{l,k}}}_{N,\beta}^{\alpha}\Bigr| &\le \E G_{N,\beta}(\alpha) \left| g^1_{\mathbf i_1,N}g^3_{\mathbf i_2,N}\right|,\\
	\Bigl|\E G_{N,\beta}(\alpha)g^4_{\mathbf i_1,N}\prod_{l=1}^{4} \prod_{k=1}^{p-1}\gbs{\sigma_{i_{l,k}}}_{N,\beta}^{\alpha}\Bigr|&\le \E G_{N,\beta}(\alpha)g^4_{\mathbf i_1,N}.
		\end{align*}
	Consequently,
	\begin{align*}
		&\Delta_p^4\sum_{\alpha \in \Sigma_{N}}\sum_{{\tiny \rm Case\,\,III}} \Big|\E  G_{N,\beta}(\alpha)g_{\mathbf i_1,N}g_{\mathbf i_2,N}g_{\mathbf i_3,N}g_{\mathbf i_4,N} \prod_{l=1}^{4}\prod_{k=1}^{p-1}\gbs{\sigma_{i_{l,k}}}_{N,\beta}^{\alpha}\Big|\\
		&\leq C_3\Delta_p^4\cdot \binom{N-1}{p-1}^2\cdot \left [\E g^2_{\mathbf i_1,N}g^2_{\mathbf i_2,N} + \E  \left| g^1_{\mathbf i_1,N}g^3_{\mathbf i_2,N}\right| + \E g^4_{\mathbf i_1,N}\right] \\
			&\leq C_3'\Delta_p^4\cdot\binom{N-1}{p-1}^2 \le C_3'\beta_p^4p^2.
	\end{align*}
	\end{itemize}
	Combining all three cases, we have
	\begin{align*}
		\E \bgbs{A_{p}^{4}}_{N,\beta} 
		&\le  C_4\bigl(\beta_{p}^{4}p^2 + \beta_{p}^{6}p^{5}+\beta_p^{8}p^8\bigr).
	\end{align*}
	Since $\sum_{p\ge 2} 2^{p}\beta_{p}^{2} < \infty$, we have $\beta_{p}^{2}= o(2^{-p})$ as $p\to \infty$. Choosing $p_0$ large enough such that $\beta_{p} \le 2^{-p/2}$ and $p^2 < 2^{p/4}$ for all $p>p_0$, it follows that
	\begin{align*}
		\sum_{p>p_{0}}\bigl(\E \bgbs{A_{p}^{4}}_{N,\beta}\bigr)^{1/4}&\le C_4 \sum_{p> p_{0}}\bigl(\beta_{p}^{4}p^2 + \beta_{p}^{6}p^{5}+\beta_p^{8}p^8\bigr)^{1/4}\le 3C_4 \sum_{p> p_{0}}\beta_{p}p^{2}\\
	&\le 3C_4 \sum_{p>p_0} 2^{-p/2}p^2 \le 3C_4 \sum_{p>p_0} 2^{-p/4} <\infty.
	\end{align*}
	For the summability for the series of $B_p^\rho$, the proof is essentially the same; the only change is that in \eqref{eqn:Ap4}, $\gbs{\sigma_{j}}_{N,\beta}^{\alpha}$ will be replaced by $s_{j}^\rho$. 
Notice that 
	$
	|s_j^\rho| \le 1$
	and any partial derivatives of $s_{j}^\rho$ of degree $d\leq 4$ with respect to the variables $(g_{\mathbf i,N})_{\mathbf i}$ are bounded by $\Delta_p^d$ up to an absolute constant independent of $p,N$ and $\mathbf i$. For example, 
	\begin{align*}
		\bigg|\frac{\partial s_{j}^{\rho}}{\partial g_{\mathbf i, N}} \bigg|&=\bigg| \Delta_p \frac{\bgbs{\tau_j\tau_{i_1}\cdots\tau_{i_{p-1}}\tau_N\sinh(X_{N,\beta}(\tau)+h)}_{N-1,\beta'}^{\rho}}{\bgbs{\cosh(X_{N,\beta}(\tau)+h)}_{N-1,\beta'}^{\rho}}\Big.\\
		&\qquad \Big.- \Delta_p \frac{s_j^{\rho}\bgbs{\tau_{i_1}\cdots\tau_{i_{p-1}}\tau_N\sinh(X_{N,\beta}(\tau)+h)}_{N-1,\beta'}^{\rho}}{\bgbs{\cosh(X_{N,\beta}(\tau)+h)}_{N-1,\beta'}^{\rho}}\bigg|\le 2\Delta_p.
	\end{align*}
    More general partial derivatives can be controlled by an induction argument on the number of differentiations.
    This  implies that \eqref{add:eq11} with $\gbs{\sigma_{j}}_{N,\beta}^{\alpha}$ replaced by $s_j^\rho$ is also valid.
	We omit the rest of the details. 
\end{proof}

\begin{proof}[\bf Proof of Proposition \ref{Coro:tailofXn}]
	Similar to \eqref{eqn:Xnbound-holder}, we have
	\[
	\E \bbbgbs{\Big[\sum_{p>p_0}A_p^\alpha\Big]^{4}}_{N, \beta} \le \Big(\sum_{p>p_0}\bigl(\E \bgbs{A_{p}^{4}}_{N,\beta}\bigr)^{\frac{1}{4}}\Big)^{4}.
	\]
    Since $\bigl(\E \bgbs{A_{p}^{4}}_{N,\beta}\bigr)^{1/4}$ is summable, as proved in Lemma \ref{lemma:Xn2bd}, the right hand side can be made arbitrarily small by choosing $p_0$ sufficiently large.  The other assertion can be treated similarly.
\end{proof}

\begin{proof}[\bf Proof of Proposition \ref{Coro:Xn-Negligible}] 
	Note that for any $p_1, p_2,p_3, p_{4}\ge 2$, using H\"older's inequality yields
	$$\E \bgbs{A_{p_{1}}^\alpha A_{p_2}^\alpha A_{p_3}^\alpha A_{p_{4}}^\alpha}_{N,\beta} \le \Big(\E \bgbs{|A_{p_1}^\alpha|^{4}}_{N,\beta}\E \bgbs{|A_{p_2}^\alpha|^{4}}_{N,\beta}\E \bgbs{|A_{p_3}^\alpha|^{4}}_{N,\beta}\E\bgbs{|A_{p_{4}}^\alpha|^{4}}_{N,\beta}\Big)^{\frac{1}{4}},$$ 
	which implies that
	\begin{align}
		\label{eqn:Xnbound-holder}
		\E \bgbs{\big[X_{N,\beta}\bigl(\gbs{\sigma}_{N,\beta}^{\alpha}\bigr)\big]^{4}}_{N, \beta} = \E \bbbgbs{\Big(\sum_{p\ge 2}A_{p}^\alpha\Big)^{4}}_{N,\beta}\le \Big(\sum_{p\ge 2}\big(\E \bgbs{|A_{p}^\alpha|^{4}}_{N,\beta}\big)^{\frac{1}{4}}\Big)^{4}.
	\end{align}
	By the Cauchy-Schwarz inequality and Lemma \ref{lemma:Xn2bd}, 
	\begin{align*}
		\E \bgbs{\big[X_{N,\beta}(\gbs{\sigma}_{N,\beta}^{\alpha})\mathbbm 1_{\{G_{N-1,\beta'}(A_{\ominus}^{\rho})<\delta\}}\big]^2}_{N, \beta}
		&\le \Big(\E \bgbs{X^4_{N,\beta}(\gbs{\sigma}_{N,\beta}^{\alpha})}_{N,\beta}\E\bgbs{\mathbbm 1_{\{G_{N-1,\beta'}(A_{\ominus}^{\rho})<\delta\}}}_{N, \beta} \Big)^{1/2}\\
		&\le \sqrt{K}\Big(\E\bgbs{\mathbbm 1_{\{G_{N-1,\beta'}(A_{\ominus}^{\rho})<\delta\}}}_{N, \beta} \Big)^{1/2}.
	\end{align*}
	Thus, \eqref{eqn:approxXn-sigma-smallW} follows from \eqref{add:lem2:eq2}. The proof of \eqref{eqn:approxXn-s-smallW} is exactly the same.
\end{proof}

\section{Proof of Lemma \ref{lemma:CpDp-diff}}\label{app2}

\begin{proof}[\bf Proof of Lemma \ref{lemma:CpDp-diff}] Fist of all, for any $\tau \in \Sigma_{N-1}, \boldsymbol \tau = (\tau^1,\ldots,\tau^{p-1}) \in \Sigma_{N-1}^{p-1}$, $X_{N,\beta}$ and $Z_{N,p}$ are centered Gaussian random variables with variances bounded by $C_\beta$ and $p$ respectively, which result in 
	\begin{align*}
		\e_{g_{\cdot N}}\bigl\la |Z_{N,p}(\boldsymbol{\tau})|^{k}\bigr\ra_{N-1,\beta'}&\leq p^{k/2}(k-1)!!,\\
		\e_{g_{\cdot N}}\bigl\la \cosh^{k}(X_{N,\beta}(\tau)+h)\bigr\ra_{N-1,\beta'}&\leq e^{C_\beta k^2/2}\cosh^k(h),\\
		\e_{g_{\cdot N}}\cosh^k(X_{N,\beta}(\tau)+h)&\leq  e^{C_\beta k^2/2}\cosh^k(h).
	\end{align*}
Using the nested structure \eqref{eqn:nestedsets} and the H\"older inequality with $p$ conjugate exponents $2r(p-1),2r(p-1),\ldots,2r(p-1)$, and $2r/(2r-1)$, we have
	\begin{align*}
		|D_p^\alpha-D_p^\rho|^{2r}&\leq \bbbgbs{\prod_{l=1}^{p-1}\cosh (X_{N,\beta}(\tau^l)+h)\bigg(\prod_{l=1}^{p-1}\mathbbm 1_{A_{\oplus}^{\rho}}(\tau^l) - \prod_{l=1}^{p-1}\mathbbm 1_{A_{\ominus}^{\rho}}(\tau^l)\bigg)}_{N-1,\beta'}^{2r}\\
  &\leq \bgbs{\cosh^{2r(p-1)} (X_{N,\beta}(\tau)+h)}_{N-1,\beta'}\bbbgbs{\bigg(\prod_{l=1}^{p-1}\mathbbm 1_{A_{\oplus}^{\rho}}(\tau^l) - \prod_{l=1}^{p-1}\mathbbm 1_{A_{\ominus}^{\rho}}(\tau^l)\bigg)^{\frac{2r}{2r-1}}}_{N-1,\beta'}^{2r-1}\\
   &\leq \bgbs{\cosh^{2r(p-1)} (X_{N,\beta}(\tau)+h)}_{N-1,\beta'}\bbbgbs{\prod_{l=1}^{p-1}\mathbbm 1_{A_{\oplus}^{\rho}}(\tau^l) - \prod_{l=1}^{p-1}\mathbbm 1_{A_{\ominus}^{\rho}}(\tau^l)}_{N-1,\beta'},
	\end{align*}
 where the last inequality holds since $\prod_{l=1}^{p-1}\mathbbm 1_{A_{\oplus}^{\rho}}(\tau^l) - \prod_{l=1}^{p-1}\mathbbm 1_{A_{\ominus}^{\rho}}(\tau^l)\in \{0,1\}$.
 Since
 \begin{align*}
     \prod_{l=1}^{p-1}\mathbbm 1_{A_{\oplus}^{\rho}}(\tau^l) - \prod_{l=1}^{p-1}\mathbbm 1_{A_{\ominus}^{\rho}}(\tau^l)&\leq \sum_{l=1}^{p-1}\mathbbm{1}_{A_\oplus^\rho\setminus A_\ominus^\rho}(\tau^l),
 \end{align*}
 we have
 \begin{align*}
    |D_p^\alpha-D_p^\rho|^{2r} &\leq (p-1) \bgbs{\cosh^{2r(p-1)} (X_{N,\beta}(\tau)+h)}_{N-1,\beta'}\bgbs{\mathbbm 1_{A_{\oplus}^{\rho}\setminus A_{\ominus}^\rho}(\tau)}_{N-1,\beta'}.
 \end{align*}
From this, by a change of measure for $\alpha=(\rho, \alpha_N)\sim G_{N,\beta}$ as in Lemma \ref{add:lem1} and the Cauchy-Schwarz inequality, we obtain the second assertion,
\begin{align*}
	\E \bgbs{(D_p^\alpha-D_p^{\rho})^{2r}}_{N,\beta}&\leq \e\Bigl[\sum_{\rho}G_{N-1,\beta'}(\rho)\bgbs{\mathbbm 1_{A_{\oplus}^{\rho}\setminus A_{\ominus}^\rho}(\tau)}_{N-1,\beta'}\\
	&\qquad\cdot\e_{g_{\cdot N}}\bigl[\cosh(X_{N,\beta}(\rho)+h)\bgbs{\cosh^{2r(p-1)} (X_{N,\beta}(\tau)+h)}_{N-1,\beta'}\bigr]\Bigr]\\
	&\leq \eta_N p e^{(4r^2 p^2 +1)C_\beta}\cosh^{2rp}(h).
\end{align*}
For the first assertion, we similarly have
	\begin{align*}
		|C_p^\alpha-C_p^{\rho}|^{2r}
		&\le \lrgbs{Z_{N,p}^2(\boldsymbol{\tau})}^r_{N-1,\beta'}\bbbgbs{\prod_{l=1}^{p-1}\cosh^2(X_{N,\beta}(\tau^l)+h)\bigg(\prod_{l=1}^{p-1}\mathbbm 1_{A_{\oplus}^{\rho}}(\tau^l) - \prod_{l=1}^{p-1}\mathbbm 1_{A_{\ominus}^{\rho}}(\tau^l)\bigg)^2}_{N-1,\beta'}^{r}\\
		&\le (p-1)\bgbs{Z_{N,p}^{2r}(\boldsymbol{\tau})}^r_{N-1,\beta'}\bgbs{\cosh^{2r(p-1)}(X_{N,\beta}(\tau)+h)}_{N-1,\beta'}\bgbs{\mathbbm 1_{A_{\oplus}^{\rho}\setminus A_\ominus^\rho}(\tau) }_{N-1,\beta'},
	\end{align*}
 where the first inequality used the Cauchy-Schwarz inequality and the second inequality was obtained by an analogous argument for $|D_p^\alpha-D_p^\rho|^{2r}$.
Via a change of measure for $G_{N,\beta}$ as above, we can then apply the H\"older inequality in the expectation $\e_{g_{\cdot N}}$ with thee conjugates exponents $3,3,3$ to get the desired bound,
\begin{align}
	\begin{split}\label{add:eq6}
	&\e\bgbs{|C_p^\alpha-C_p^{\rho}|^{2r}}_{N,\beta}\\
	&\leq \e\Bigl[\sum_{\rho}G_{N-1,\beta'}(\rho)\bgbs{\mathbbm 1_{A_{\oplus}^{\rho}\setminus A_\ominus^\rho}(\tau^l) }_{N-1,\beta'}\\
	&\qquad\cdot\e_{g_{\cdot N}}\bigl[\cosh(X_{N,\beta}(\rho)+h) \bgbs{Z^{2r}_{N,p}(\boldsymbol{\tau})}_{N-1,\beta'}\bgbs{\cosh^{2r(p-1)}(X_{N,\beta}(\tau)+h)}_{N-1,\beta'}\bigr]\Bigr]\\
	&\leq  \eta_N(p-1)p^{r}\bigl[(6r-1)!!\bigr]^{1/3}e^{(3/2+6r^2 p^2)C_\beta}\cosh^{2rp}(h).
	\end{split}
\end{align}
\end{proof}

\section{\label{app3}Proof of Lemma \ref{prop:Xn2bd-2}}
\begin{proof}[\bf Proof of Lemma \ref{prop:Xn2bd-2}]
	The proof is essentially the same as that for Proposition \ref{Coro:tailofXn}. First of all, we claim that there exists a constant $K= K(\beta, h)>0$ such that for all $N\ge 1$ and any small $\epsilon>0$,
		\begin{align}\label{add:eq13}
	    \sum_{p=2}^\infty \bigl(\e \la |E_p^\rho|^4\ra_{N,\beta}\bigr)^{1/4}&\leq K.
	\end{align}
This part of the argument is analogous to the proof of Lemma \ref{lemma:Xn2bd}, but in a slightly simpler manner.
We begin by rewriting
	\begin{align*}
		\E \bgbs{|E^{\rho}_{p}|^{4}}_{N,\beta} &= \E \bbbgbs{\Big(\frac{\beta_{p}^{2}p!}{N^{p-1}}\Big)^2 \bigg(\sum_{1\le i_1<\cdots<i_{p-1}\le N-1}^{N-1} g_{i_1,\ldots, i_{p-1},N}\gbs{\sigma_{i_1}}_{N,\beta'}^{\rho}\cdots \gbs{\sigma_{i_{p-1}}}_{N-1,\beta'}^{\rho}\bigg)^{4}}_{N,\beta}\\
		&=\Big(\frac{\beta_{p}^{2}p!}{N^{p-1}}\Big)^2\sum_{\alpha \in \Sigma_{N}}\sum_{\mathbf i_1,\ldots, \mathbf i_{4}}\E \Big[ G_{N,\beta}(\alpha)\prod_{l=1}^{4}g_{\mathbf i_l,N} \prod_{k=1}^{p-1}\gbs{\sigma_{i_{l,k}}}_{N-1,\beta'}^{\rho}\Big].
	\end{align*}
	When applying Gaussian integration by parts to control the last equation, we only need to differentiate $G_{N,\beta}(\alpha)$ with respect to $g_{\mathbf i_l, N}$ and the bounds of the partial derivatives of $G_{N,\beta}(\alpha)$ given by \eqref{add:eq10}. An identical argument as in the proof of Lemma \ref{Coro:tailofXn} implies our claim \eqref{add:eq13}, the summability of $\bigl(\e \la |E_p^\rho|^4\ra_{N,\beta}\bigr)^{1/4}$. With this claim,  our assertion follows immediately since, similar to \eqref{eqn:Xnbound-holder}, $$
	\E\bbbgbs{\bigg[\sum_{p> p_0}E_p^\rho\bigg]^4}_{N,\beta}\leq \bigg(\sum_{p> p_0}\bigl(\e \la |E_p^\rho|^4\ra_{N,\beta}\bigr)^{1/4}\bigg)^4,
	$$
	and the right hand side can be made arbitrarily small by choosing $p_0$ sufficiently large. 
\end{proof}
\end{appendices}


\begin{thebibliography}{10}
\footnotesize
\bibitem{ABS21}
A. Adhikari, C. Brennecke, P. von Soosten and H.-T. Yau.
\newblock Dynamical approach to the TAP equations for the Sherrington–Kirkpatrick model.
\newblock {\em Journal of Statistical Physics}, 183(35):1--27, 2021.


\bibitem{AC15}
A.~Auffinger and W.-K. Chen.
\newblock The {P}arisi formula has a unique minimizer.
\newblock {\em Communications in Mathematical Physics}, 335(3):1429--1444,
  2015.

\bibitem{auffinger2019spin}
A.~Auffinger and A.~Jagannath.
\newblock On spin distributions for generic $p$-spin models.
\newblock {\em Journal of Statistical Physics}, 174(2):316--332, 2019.

\bibitem{AJ191}
A.~Auffinger and A.~Jagannath.
\newblock Thouless-{A}nderson-{P}almer equations for generic {$p$}-spin
  glasses.
\newblock {\em The Annals of Probability}, 47(4):2230--2256, 2019.


\bibitem{barra2015multi}
A. Barra, P. Contucci, E. Mingione, and D. Tantari.
\newblock Multi-species mean field spin glasses. Rigorous results
 \newblock {\em Annales Henri Poincar{\'e}}, 16:691--708, 2015.

\bibitem{bayati2011dynamics}
M.~Bayati and A.~Montanari.
\newblock The dynamics of message passing on dense graphs, with applications to
  compressed sensing.
\newblock {\em IEEE Transactions on Information Theory}, 57(2):764--785, 2011.

\bibitem{belius2022high}
D.~Belius.
\newblock High temperature {TAP} upper bound for the free energy of mean field
  spin glasses.
\newblock {\em arXiv preprint arXiv:2204.00681}, 2022.

\bibitem{belius2019tap}
D.~Belius and N.~Kistler.
\newblock The {TAP}--{Plefka} variational principle for the spherical {SK}
  model.
\newblock {\em Communications in Mathematical Physics}, 367(3):991--1017, 2019.

\bibitem{Bolthausen14}
E.~Bolthausen.
\newblock An iterative construction of solutions of the {TAP} equations for the
  {S}herrington-{K}irkpatrick model.
\newblock {\em Communications in Mathematical Physics}, 325(1):333--366, 2014.

\bibitem{Bolthausen19}
E.~Bolthausen.
\newblock A {M}orita type proof of the replica-symmetric formula for {SK}.
\newblock In {\em Statistical mechanics of classical and disordered systems},
  volume 293 of {\em Springer Proceedings in Mathematics \& Statistics}, pages 63--93. Springer, Cham,
  2019.

\bibitem{BrenneckeYau21}
E.~Brennecke and H.-T. Yau.
\newblock The replica symmetric formula for the SK model revisited.
\newblock {\em Journal of Mathematical Physics}, 63 (073302): 1--12, 2022.

\bibitem{Chatterjee10}
S.~Chatterjee.
\newblock Spin glasses and {S}tein's method.
\newblock {\em Probability Theory and Related Fields}, 148(3-4):567--600, 2010.

\bibitem{CP18}
W.-K. Chen and D.~Panchenko.
\newblock On the {TAP} free energy in the mixed {$p$}-spin models.
\newblock {\em Communications in Mathematical Physics}, 362(1):219--252, 2018.

\bibitem{CPS18}
W.-K. Chen, D.~Panchenko, and E.~Subag.
\newblock Generalized {TAP} free energy.
\newblock {\em Communications on Pure and Applied Mathematics}, 2018.

\bibitem{CPS19}
W.-K. Chen, D.~Panchenko, and E.~Subag.
\newblock The generalized {TAP} free energy {II}.
\newblock {\em Communications in Mathematical Physics}, 381(1):257--291, 2021.

\bibitem{ChenTang2021}
W.-K. Chen and S.~Tang.
\newblock On {C}onvergence of the {C}avity and {B}olthausen’s {TAP}
  {I}terations to the {L}ocal {M}agnetization.
\newblock {\em Communications in Mathematical Physics}, 386:1209--1242, 2021.

\bibitem{ATline}
J. R.~L. de~Almeida and D.~J. Thouless. 
\newblock {Stability of the sherrington-kirkpatrick solution of a spin glass
  model.}
\newblock{\em J. {P}hys. {A}}, 11(5):983--990, 1978.

  
\bibitem{el2021optimization}
A.~El~Alaoui, A.~Montanari, and M.~Sellke.
\newblock Optimization of mean-field spin glasses.
\newblock {\em The Annals of Probability}, 49(6):2922--2960, 2021.

\bibitem{J2017}
A.~Jagannath.
\newblock {A}pproximate {U}ltrametricity for {R}andom {M}easures and
  {A}pplications to {S}pin {G}lasses.
\newblock {\em Commications on Pure and Applied Mathematics}, 70:611--664,
  April 2017.

\bibitem{javanmard2013state}
A.~Javanmard and A.~Montanari.
\newblock State evolution for general approximate message passing algorithms,
  with applications to spatial coupling.
\newblock {\em Information and Inference: A Journal of the IMA}, 2(2):115--144,
  2013.

\bibitem{kabashima2016phase}
Y.~Kabashima, F.~Krzakala, M.~M{\'e}zard, A.~Sakata, and L.~Zdeborov{\'a}.
\newblock Phase transitions and sample complexity in bayes-optimal matrix
  factorization.
\newblock {\em IEEE Transactions on information theory}, 62(7):4228--4265,
  2016.

\bibitem{MPV87}
M.~M\'{e}zard, G.~Parisi, and M.~A. Virasoro.
\newblock {\em Spin glass theory and beyond}, volume~9 of {\em World Scientific
  Lecture Notes in Physics}.
\newblock World Scientific Publishing Co., Inc., Teaneck, NJ, 1987.

\bibitem{mezard1985microstructure}
M.~M{\'e}zard and M.~A. Virasoro.
\newblock The microstructure of ultrametricity.
\newblock {\em Journal de Physique}, 46(8):1293--1307, 1985.

\bibitem{montanari2021optimization}
A.~Montanari.
\newblock Optimization of the {S}herrington--{K}irkpatrick {H}amiltonian.
\newblock {\em SIAM Journal on Computing}, 0(0):FOCS19--1, 2021.

\bibitem{montanari2015non}
A.~Montanari and E.~Richard.
\newblock Non-negative principal component analysis: Message passing algorithms
  and sharp asymptotics.
\newblock {\em IEEE Transactions on Information Theory}, 62(3):1458--1484,
  2015.

\bibitem{montanari2021estimation}
A.~Montanari and R.~Venkataramanan.
\newblock Estimation of low-rank matrices via approximate message passing.
\newblock {\em The Annals of Statistics}, 49(1):321--345, 2021.

\bibitem{panchenko2008differentiability}
D.~Panchenko.
\newblock On differentiability of the {P}arisi formula.
\newblock {\em Electronic Communications in Probability}, 13:241--247, 2008.

\bibitem{panchenko2013parisi}
D.~Panchenko.
\newblock The {P}arisi ultrametricity conjecture.
\newblock {\em Annals of Mathematics}, pages 383--393, 2013.

\bibitem{Pan13}
D.~Panchenko.
\newblock {\em The {S}herrington-{K}irkpatrick model}.
\newblock Springer Monographs in Mathematics. Springer, New York, 2013.

\bibitem{panchenko2014parisi}
D.~Panchenko.
\newblock The {P}arisi formula for mixed $ p $-spin models.
\newblock {\em The Annals of Probability}, 42(3):946--958, 2014.

\bibitem{parisi1979infinite}
G.~Parisi.
\newblock Infinite number of order parameters for spin-glasses.
\newblock {\em Physical Review Letters}, 43(23):1754, 1979.

\bibitem{parisi1980sequence}
G.~Parisi.
\newblock A sequence of approximated solutions to the {SK} model for spin
  glasses.
\newblock {\em Journal of Physics A: Mathematical and General}, 13(4):L115,
  1980.

\bibitem{parisi1983order}
G.~Parisi.
\newblock Order parameter for spin-glasses.
\newblock {\em Physical Review Letters}, 50(24):1946, 1983.

\bibitem{SK72}
D.~Sherrington and S.~Kirkpatrick.
\newblock Solvable model of a spin glass.
\newblock {\em Physical Review Letters}, 35:1792--1796, 1972.

\bibitem{Subag2018}
E.~Subag.
\newblock Following the ground-states of full-{RSB} spherical spin glasses.
\newblock {\em Communications on Pure and Applied Mathematics}, 74:1021--1044, 2020.

\bibitem{subag2021tap}
E.~Subag.
\newblock {TAP} approach for multi-species spherical spin glasses {I}: general theory.
\newblock {\em arXiv preprint arXiv:2111.07132}, 2021.

\bibitem{subag2021tap2}
E.~Subag.
\newblock {TAP} approach for multi-species spherical spin glasses {II}: the
  free energy of the pure models.
\newblock {\em The Annals of Probability}, 51(3):1004-1024, 2021.

\bibitem{talagrand2006parisi}
M.~Talagrand.
\newblock The {P}arisi formula.
\newblock {\em Annals of Mathematics}, pages 221--263, 2006.

\bibitem{talagrand2010construction}
M.~Talagrand.
\newblock Construction of pure states in mean field models for spin glasses
\newblock {\em Probability Theory and Related Fields}, 148:601--643, 2006.

\bibitem{Tal111}
M.~Talagrand.
\newblock {\em Mean field models for spin glasses. {V}olume {I}}, volume~54 of
  {\em Ergebnisse der Mathematik und ihrer Grenzgebiete. 3. Folge. A Series of
  Modern Surveys in Mathematics [Results in Mathematics and Related Areas. 3rd
  Series. A Series of Modern Surveys in Mathematics]}.
\newblock Springer-Verlag, Berlin, 2011.
\newblock Basic examples.

\bibitem{Tal112}
M.~Talagrand.
\newblock {\em Mean field models for spin glasses. {V}olume {II}}, volume~55 of
  {\em Ergebnisse der Mathematik und ihrer Grenzgebiete. 3. Folge. A Series of
  Modern Surveys in Mathematics [Results in Mathematics and Related Areas. 3rd
  Series. A Series of Modern Surveys in Mathematics]}.
\newblock Springer, Heidelberg, 2011.
\newblock Advanced replica-symmetry and low temperature.


\bibitem{TAP}
D.~J. Thouless, P.~W. Anderson, and R.~G. Palmer.
\newblock Solution of `solvable model of a spin glass'.
\newblock {\em Philosphical Magazine}, 35(3):593--601, 1977.

\bibitem{zdeborova2016statistical}
L.~Zdeborov{\'a} and F.~Krzakala.
\newblock Statistical physics of inference: Thresholds and algorithms.
\newblock {\em Advances in Physics}, 65(5):453--552, 2016.

\end{thebibliography}
\end{document}